\documentclass[11pt,reqno,a4paper]{amsart}
\usepackage{geometry}

\usepackage{comment}

\usepackage[all]{xy} 

\usepackage{pgfplots}
\pgfplotsset{compat=1.15}
\usepgfplotslibrary{fillbetween}
\usepackage{mathrsfs}

\usepackage{amsmath, amssymb, amsfonts, latexsym, mdwlist, amsthm}
\usepackage{subfig}
\usepackage{graphicx}
\usepackage{wrapfig}

\usepackage[bookmarks, colorlinks, breaklinks, pdftitle={},
pdfauthor={Soheyla Feyzbakhsh}]{hyperref}
\hypersetup{linkcolor=blue,citecolor=blue,filecolor=black,urlcolor=blue}

\usepackage{tikz}
\usetikzlibrary{calc,trees,positioning,arrows,chains,shapes.geometric,%
    decorations.pathreplacing,decorations.pathmorphing,shapes,%
    matrix,shapes.symbols}

\def\bQ{\ensuremath{\mathbb{Q}}}
\def\bR{\ensuremath{\mathbb{R}}}
\def\bZ{\ensuremath{\mathbb{Z}}}
\def\bC{\ensuremath{\mathbb{C}}}
\DeclareMathOperator{\Pic}{Pic}
\DeclareMathOperator{\ev}{ev}

\def\cD{\ensuremath{\mathcal{D}}}
\def\cE{\ensuremath{\mathcal{E}}}
\def\cF{\ensuremath{\mathcal{F}}}
\def\cG{\ensuremath{\mathcal{G}}}
\def\cH{\ensuremath{\mathcal{H}}}

\def\cM{\ensuremath{\mathcal{M}}}
\def\cO{\ensuremath{\mathcal{O}}}

\def\cT{\ensuremath{\mathcal{T}}}
\def\cW{\ensuremath{\mathcal{W}}}
\def\cV{\ensuremath{\mathcal{V}}}

\def\mm{\overline{\mathcal{M}}}

\DeclareMathOperator{\ch}{ch}

\usepackage{paralist}
\setdefaultenum{(a)}{(i)}{}{}
\usepackage{enumitem} 

\makeatletter
\newtheorem*{rep@theorem}{\rep@title}
\newcommand{\newreptheorem}[2]{%
\newenvironment{rep#1}[1]{%
 \def\rep@title{#2 \ref{##1}}%
 \begin{rep@theorem}}%
 {\end{rep@theorem}}}
\makeatother

\newtheorem{Thm}{Theorem}[section]
\newreptheorem{Thm}{Theorem}
\newtheorem{Prop}[Thm]{Proposition}

\newtheorem{Lem}[Thm]{Lemma}

\newtheorem{Cor}[Thm]{Corollary}
\newreptheorem{Cor}{Corollary}

\newreptheorem{Con}{Conjecture}

\newtheorem{thm-int}{Theorem}

\theoremstyle{definition}
\newtheorem{Def-s}[Thm]{Definition}
\newtheorem{Def}[Thm]{Definition}
\newtheorem{Rem}[Thm]{Remark}

\newtheorem{Ex}[Thm]{Example}
\newtheorem{Claim}[Thm]{Claim}

\def\Ker{\mathop{\mathrm{Ker}}\nolimits}
\def\Coh{\mathop{\mathrm{Coh}}\nolimits}
\def\Hom{\mathop{\mathrm{Hom}}\nolimits}
\def\hom{\mathop{\mathrm{hom}}\nolimits}

\def\Gr{\mathop{\mathrm{Gr}}\nolimits}
\def\Ext{\mathop{\mathrm{Ext}}\nolimits}
\def\ext{\mathop{\mathrm{ext}}\nolimits}
\def\ev{\mathop{\mathrm{ev}}\nolimits}

\newcommand{\ignore}[1]{}

\newcommand{\An}[1]{\textcolor{blue}{#1}}

\begin{document}

\title{Hurwitz-Brill-Noether theory via $K3$ surfaces and stability conditions}

\begin{abstract}
We develop a novel approach to the Brill--Noether theory of curves endowed with a degree $k$ cover of $\mathbf{P}^1$ via Bridgeland stability conditions on elliptic $K3$ surfaces.

We first develop the Brill--Noether theory on elliptic $K3$ surfaces via the notion of Bridgeland stability type for objects in their derived category. As a main application, we show that curves on elliptic $K3$ surfaces serve as the first known examples of smooth $k$-gonal curves which are general from the viewpoint of Hurwitz--Brill--Noether theory. In particular, we provide new proofs of the main non-existence and existence results in Hurwitz--Brill--Noether theory. Finally, using degree-$k$ Halphen surfaces, we construct explicit examples of curves defined over number fields which are general from the perspective of Hurwitz--Brill--Noether theory.
\end{abstract}


\author[G. Farkas]{Gavril Farkas}
\address{Gavril Farkas: Institut f\"ur Mathematik,   Humboldt-Universit\"at zu Berlin \hfill \newline\texttt{}
\indent Unter den Linden 6,
10099 Berlin, Germany}
\email{{\tt farkas@math.hu-berlin.de}}

\author[S. Feyzbakhsh]{Soheyla Feyzbakhsh}
\address{Soheyla Feyzbakhsh: Department of Mathematics,   Imperial College \hfill \newline\texttt{}
\indent London SW72 AZ, United Kingdom}
\email{{\tt s.feyzbakhsh@imperial.ac.uk}}

\author[A. Rojas]{Andrés Rojas}
\address{Andrés Rojas: Departament de Matemàtiques i Informàtica, Universitat de Barcelona \hfill \newline\texttt{}
\indent Gran Via de les Corts Catalanes 585,
08007 Barcelona, Spain}
\email{{\tt andresrojas@ub.edu}}

\maketitle
\setcounter{tocdepth}{1}
\tableofcontents

\section{Introduction}
The Brill-Noether theorem, asserting that for a general curve $C$ of genus $g$ the dimension of the variety of linear systems 
\[
W^r_d(C):=\{L\in \mbox{Pic}^d(C): h^0(C,L)\geq r+1\}
\]
equals the Brill-Noether number $\rho(g,r,d)=g-(r+1)(g-d+r)$, is one of the cornerstones of the theory of algebraic curves. After having been formulated as what we would nowadays call, a plausibility statement in the later 19th century, its proof was completed by Griffiths-Harris, Gieseker and Eisenbud-Harris in the 1980s using degeneration and limit linear series, see \cite{griffiths-harris}, \cite{gieseker:petri}, \cite{EH1}, \cite{EH2} and \cite{fulton-lazarsfeld}. A few years later, Lazarsfeld \cite{lazarsfeld-BNP} found a completely different approach to this problem by showing that every smooth curve $C$ on a $K3$ surface $X$ with $\mbox{Pic}(X)\cong \mathbb Z\cdot C$  satisfies the Brill-Noether theorem. These two approaches, via degeneration and limit linear series respectively via $K3$ surfaces have been then brought together in \cite{arbarello-explicit}, where it has been showed that a suitably general curve $C$ on a Halphen surface (that is, a rational surface shown in \cite{arbarello-bruno-sernesi} to be a limit of polarized $K3$ surfaces) satisfies the Brill-Noether theorem.

\vskip 4pt

Hurwitz-Brill-Noether theory is a much more recent development and concerns the loci $W^r_d(C)$, when $[C, A]$ is a general point of the Hurwitz space $\cH_{g,k}$ classyfing pairs consisting  of a genus $g$ curve $C$ and a pencil $A\in W^1_k(C)$ inducing a degree $k$ cover $C\rightarrow \mathbf{P}^1$. Without loss of generality, one may assume that $d\leq g-1$. After initial work of Coppens and Martens \cite{coppens-martens}, Pflueger \cite{pflueger} showed via tropical methods that if 
\begin{equation}\label{eq:Hurwitz-BN}
\rho_k(g, r,d):=\mbox{max}_{\ell=0, \ldots,r} \bigl\{\rho(g, r-\ell, d)-\ell k\bigr\}
\end{equation}
then for such a curve we have $\mbox{dim } W^r_d(C)\leq \rho_k(g,r,d)$; then Jensen-Ranganathan \cite{jensen-ranganathan}, also via tropical geometry, established the existence part and showed that indeed $\mbox{dim } W^r_d(C)=\rho_k(g,r,d)$. 
These results have been greatly refined through the notion of \emph{splitting type} of a linear system $L\in W^r_d(C)$, first considered in \cite{larson:inv} and \cite{cook-powell-jensen}. In particular, the results of
 H.~Larson \cite{larson:inv} and E.~Larson-H.~Larson-Vogt \cite{larson-larson-vogt}, describe via degeneration all the irreducible components of $W^r_d(C)$ for a general point $[C, A]\in \cH_{g,k}$ and single out an  open subset where these components  are smooth. For recent related work on these matters we also refer to \cite{coppens-new}.

\vskip 4pt

The goal of this paper is to provide a radically different approach to Hurwitz-Brill-Noether theory via \emph{Bridgeland stability conditions}. We consider $K3$ surfaces $X$ with $\Pic(X)\cong \bZ\cdot H\oplus \bZ\cdot E$, where $H$ is an ample class such that $H^2=2g-2$, and $|E|$ is an elliptic pencil with $E\cdot H=k$. In this way, the curves $C\in |H|$ are endowed with a degree $k$ pencil $A:=\cO_C(E)\in W^1_k(C)$. Such degree $k$ elliptic $K3$ surfaces form an irreducible Noether-Lefschetz divisor in the moduli space $\mathcal{F}_g$ of polarized $K3$ surfaces of genus $g$. 

\vskip 3pt
As a key step toward our final goal in Hurwitz–Brill–Noether theory, we first develop the Brill–Noether theory for elliptic K3 surfaces, as outlined below.

\subsection{Bridgeland stability types} Let $\cD(X)$ denote the bounded derived category of coherent sheaves on $X$. Stability conditions on $\cD(X)$, introduced by Bridgeland in his fundamental papers \cite{bridgeland:stability-condition-on-triangulated-category} and \cite{bridgeland:K3-surfaces}, generalize to objects in $\cD(X)$  the traditional notion of Gieseker stability of sheaves. For $\epsilon \in \mathbb{Q}_{>0}$, we consider the polarization 
\[
H_{\epsilon} := E + \epsilon H \in \Pic(X)_{\mathbb{Q}}.
\]
on $X$. With respect to $H_{\epsilon}$, we study a ray\footnote{To carry out a wall-crossing analysis, we in fact work over a two-dimensional slice of stability conditions. In the Introduction, for simplicity, we focus here on the central ray of this slice, where the walls correspond to certain non-negative real numbers.} of stability conditions $\sigma_w$ for $w > 0$. One has a  slope function $\nu_w$ defined on a fixed abelian subcategory $\Coh^0(X) \subseteq \mathcal{D}(X)$ parametrizing $2$-term complexes of sheaves on $X$, see also Theorem \ref{thm-space}.  Accordingly, one has a concept of $\sigma_w$-stability, defined in terms of $\nu_w$-slopes, on the category $\Coh^0(X)$. 

For any fixed Mukai vector $v = (\ch_0, \ch_1, \ch_2 + \ch_0)$, there are finitely many walls $w_1 < \cdots < w_q$ such that the moduli space $\cM_{\sigma_w}(v)$ parametrizing $\sigma_w$-stable objects in $\Coh^0(X)$ of class $v$ remains unchanged in the regions between these walls.

In the case of $K3$ surfaces of Picard number $1$, as studied in \cite{bayer:brill-noether, li-bn}, if $\ch_1(v)$ is minimal, then there is no wall for the class $v$ along the ray $\sigma_w$ for $w>0$, that is, the only wall we encounter is the one induced by  $\mathcal{O}_X$ (at $w=0$), via the exact triangle
\[
\mathcal{O}_X \otimes \Hom(\mathcal{O}_X, F) \stackrel{\ev}\longrightarrow F \longrightarrow \mathcal{E} ,
\]
where $\mathcal{E}$ denotes the shifted dual of the  \emph{Lazarsfeld–Mukai bundle} \cite{lazarsfeld-BNP, aprodu:lazarsfeld}  in the case where $F$ is a line bundle on a curve on $X$. In our setting however, line bundles of the form $\mathcal{O}_X(eE)$ for $e > 0$  define additional walls, making the wall-crossing analysis significantly more challenging. Motivated by the classical notion of scrollar invariants \cite{coppens-martens} and by H. Larson's recent concept of splitting type \cite{larson:inv} for line bundles on curves, we introduce the notion of \emph{Bridgeland stability type} (Definition~\ref{Def.splitting}) for arbitrary objects in the derived category, which enables us to handle such walls systematically.

\vskip 4pt

Start with a $\sigma_{w_0}$-stable object $F \in \Coh^0(X)$ for some $w_0 > 0$. We say that $F$ has \emph{stability type}
\[
\overline{e} = \bigl((e_1, m_1), \ldots, (e_p, m_p)\bigr),
\]
where $e_1 > \cdots > e_p \geq 0$ and $m_i > 0$, if the following holds: As we move down from $w_0$ toward the origin, the object $F$ is first destabilized at a wall $w_1 < w_0$ by an object $\mathcal{O}_X(e_1E) \otimes \Hom(\cO_X(e_1E), F)\cong \cO_X(e_1E)^{\oplus m_1}$, and the corresponding destabilizing quotient $F_1$ is $\sigma_{w_1}$-stable. We repeat this process with $F_1$, finding the first wall below $w_1$ destabilising it, and so on. In total, we obtain a series of destabilising short exact sequences in $\Coh^0(X)$
\begin{align*}
 0 \longrightarrow \mathcal{O}_X(e_1E)^{\oplus m_1} \longrightarrow &F \longrightarrow F_1 \longrightarrow 0, \\
 0 \longrightarrow  \mathcal{O}_X(e_2E)^{\oplus m_2} \longrightarrow &F_1 \longrightarrow F_2 \longrightarrow 0, \\
&\vdots \\
  0 \longrightarrow \mathcal{O}_X(e_pE)^{\oplus m_p} \longrightarrow &F_{p-1} \longrightarrow F_p \longrightarrow 0,
\end{align*}
where each $F_i$ is stable along the wall that destabilizes $F_{i-1}$, and the final object $F_p$ satisfies $\Hom(\mathcal{O}_X, F_p) = 0$.
\vskip 4pt

For the remainder of this subsection, we fix the Mukai vector
$$
v := (r_0, H - a_0E, s_0 + r_0),
$$
with $r_0 \leq 0$, $a_0 \geq 0$, and $s_0 < 0$. We choose $\epsilon > 0$ sufficiently small, depending on this class $v$. Our first result states that, for suitable $w$, any $\sigma_w$-stable object of class $v$ has an associated stability type (see Theorem~\ref{thm.splitting}):

\begin{Thm}\label{thm.intro1} 
Fix $w_0 > 0$ such that $\nu_{w_0}(v) < 0$. Every $\sigma_{w_0}$-stable object $F \in \Coh^0(X)$ of Mukai vector $v$ admits a stability type
$\overline{e} = \bigl((e_1, m_1), \ldots, (e_p, m_p)\bigr)$ with $p \geq 0$. 

Moreover, the following inequalities hold:
\begin{enumerate}
    \item $\displaystyle \sum_{i=1}^p m_i \leq \hom(\mathcal{O}_X, F)  \leq \sum_{i=1}^p m_i(e_i + 1),$
    \item $\displaystyle m_1(e_1 + 1) \leq \hom(\mathcal{O}_X, F).$
\end{enumerate}
\end{Thm}

In other words, flexibility in the choice of the polarization $H_\epsilon$ paves the way to performing a systematic study of wall-crossing for $v$. Observe also that the inequalities in Theorem \ref{thm.intro1} indicate a strong link between the stability type and the Brill-Noether properties of an object in $\mbox{Coh}^0(X)$.

We expect that the above theorem represents a first step toward a broader framework for studying wall-crossing phenomena on elliptic $K3$ surfaces—a topic of independent interest that we postpone to future work. This perspective has potential applications to Brill–Noether theory on moduli spaces of sheaves (in the spirit of \cite{markman, leyenson, CNY-K3, CNY-abelian} for $K$-trivial surfaces), as well as to the study of higher rank vector bundles on $k$-gonal curves (see \cite{bakker-farkas, soheyla-li, soheyla:restriction} for examples of such study on curves on Picard rank one $K3$ surfaces).

\vskip 4pt

The following theorem describes in geometric terms the moduli spaces of Bridgeland stable objects with a fixed stability type $\overline{e}$ (see Theorem~\ref{Prop-scheme structure}):

\begin{Thm}\label{thm.intro2} 
Let $w > 0$ be generic, and let $\overline{e} = \bigl((e_1, m_1), \ldots, (e_p, m_p)\bigr)$ be any stability type. Then the subset
\[
\cM_{\sigma_w}(v, \overline{e}) := \bigl\{ F \in \cM_{\sigma_w}(v) : \text{$F$ has stability type } \overline{e} \bigr\}
\]
admits a natural scheme structure as an iterated Grassmann bundle inside $\cM_{\sigma_w}(v)$. If non-empty, the space $\cM_{\sigma_w}(v, \overline{e})$ is smooth, quasi-projective, and irreducible of dimension
\[
\Bigl( v - \sum_{i=1}^p m_i(1, e_iE, 1) \Bigr)^2 + 2 + \sum_{j=1}^p m_j \Bigl( \Bigl\langle v - \sum_{i=1}^j m_i(1, e_iE, 1), (1, e_jE, 1) \Bigr\rangle - m_j \Bigr),
\]
where $\langle -, - \rangle$ denotes the Mukai pairing.
\end{Thm}

\vskip 4pt

Theorem~\ref{thm.intro2} does not guarantee the non-emptiness of the moduli space $\cM_{\sigma}(v,\overline{e})$. However, it follows from the definition of Bridgeland stability type that the inequalities
\begin{equation}\label{nec.condition-intro1}
\Bigl(v - \sum_{i=1}^p m_i(1, e_iE, 1)\Bigr)^2 + 2 \geq 0, \qquad
\nu_w\bigl(\cO_X(e_1E)\bigr) < \nu_w(v) < 0
\end{equation}
are necessary conditions for $\cM_{\sigma}(v, \overline{e})$ to be non-empty.

These conditions also turn out to be sufficient for a distinguished class of stability types that we call \emph{balanced stability types}\footnote{The term “balanced” is used in alignment with H.~Larson’s terminology in~\cite{larson:inv}.}. These are stability types of the form
\[
\overline{e} = \bigl((e+1, m_1), (e, m_2)\bigr), \quad \text{with } \  e, m_1, m_2 \geq 0.
\]

\begin{Thm}\label{thm.intro3}
Let $\overline{e}$ be a balanced stability type as above, and assume $m_1 + m_2 \leq k + r_0$. If $w > 0$ satisfies the inequalities~\eqref{nec.condition-intro1}, then the moduli space 
$
\mathcal{M}_{\sigma_w}(v, \overline{e})
$
is non-empty.
\end{Thm}

Theorem \ref{thm.intro3}, which is Theorem \ref{thm-nonempty} in the paper, is our main existence result in the Brill-Noether theory of elliptic $K3$ surfaces.

\vskip 4pt

\subsection{Hurwitz-Brill-Noether theory via Bridgeland stability types}
We now apply our results on the elliptic $K3$ surface $X$ to the study of the Brill–Noether loci $W^r_d(C)$ for pairs $[C, A] \in \mathcal{H}_{g,k}$, where $C \in |H|$ and $A := \mathcal{O}_C(E)$. Our first result in this direction shows that a general curve $C \in |H|$ satisfies the Hurwitz–Brill–Noether theorem. This can be viewed as the Hurwitz-theoretic analogue of Lazarsfeld’s classical result \cite{lazarsfeld-BNP} for Brill–Noether generality of curves on $K3$ surfaces with Picard number one.

\vspace{1mm}

\begin{Thm}\label{thm.intro4}  
For any \( d \leq g - 1 \) and \( r \geq 0 \), the following hold:  
\begin{enumerate}  
    \item\label{IntroThm4 a} If \( C \in |H| \) is a general curve, then \( \dim W^r_d(C) = \rho_k(g,r,d) \).  
    \item\label{IntroThm4 b} If \( \rho_k(g,r,d) < 0 \), then \( W^r_d(C) = \emptyset \) for every integral curve $C\in |H|$.

\end{enumerate}  
\end{Thm}  

\vspace{1mm}

Note that for a \emph{general $k$-gonal curve} Theorem \ref{thm.intro4} has been established as a combination of results from \cite{pflueger, jensen-ranganathan, larson:inv}. However, not a single example of a smooth $k$-gonal curve verifying the Hurwitz-Brill-Noether Theorem has been known before and our Theorem \ref{thm.intro4} shows that general curves on an elliptic $K3$ surface enjoy this property. 

\vskip 3pt

To prove Theorem~\ref{thm.intro4}, for an integral curve $C \in |H|$, we regard (the pushforward to $X$ of) any line bundle $L \in \overline{\Pic}^d(C)$ as an $H_\epsilon$-Gieseker stable sheaf on $X$ with Mukai vector $v = (0, H, 1 + d - g)$. For $\epsilon>0$ small enough, Theorem \ref{thm.intro1} asserts that $L$ has an associated Bridgeland stability type. Then one deduces from Theorem \ref{thm.intro2} that in the Gieseker moduli space $\cM_{H_\epsilon}(v)$, the subset
\[
\bigl\{ 
F\in\cM_{H_\epsilon}(v): \;h^0(X, F)\geq r+1
\bigr\}
\]
has dimension at most $g+\rho_k(g,r,d)$, by considering all the stability types compatible with the condition $h^0(X, F)\geq r+1$. Furthermore, under the assumption $\rho_k(g,r,d)<0$, all the allowed stability types violate inequalities \eqref{nec.condition-intro1}. This proves part \eqref{IntroThm4 b} and the upper bound of part \eqref{IntroThm4 a} in Theorem \ref{thm.intro4}. On the other hand, we prove the existence statement $\dim W^r_d(C) \geq \rho_k(g, r, d)$ by addressing the more general problem of describing irreducible components of $W^r_d(C)$, as outlined below.

\vspace{3mm}

Fix a point $[C, A] \in \mathcal{H}_{g,k}$ and any integer $r \geq 0$. For any integer $\ell$ satisfying the inequalities $\max\{0, r+2-k\} \leq \ell \leq r$, define $e := \left\lfloor \frac{\ell}{r+1 - \ell} \right\rfloor$, 
and write
\begin{equation}\label{ell}
\ell = e(r+1-\ell) + m_1, \quad m_2 := r+1 - \ell - m_1,
\end{equation}
so that $0 \leq m_1 \leq r - \ell$. We then define the following locus (see also~\eqref{eq:defl}):
\begin{align*}
V^{r}_{d, \ell}(C, A) := \bigl\{ L \in \Pic^d(C) :\;&
h^0(C, L \otimes A^{-e-2}) = 0,\quad h^0(C, L \otimes A^{-e-1}) = m_1, \\
& h^0(C, L \otimes A^{-e}) = 2m_1 + m_2,\quad h^0(C, L) = r+1 \bigr\}\,\footnotemark. %
\end{align*}
\footnotetext{In the language of~\cite{larson:inv}, $V^r_{d,\ell}(C)$ is the locus of degree $d$ line bundles whose splitting type has non-negative part $(\underbrace{e,\dots,e}_{m_2}, \underbrace{e+1,\dots,e+1}_{m_1})$.}
The conditions $h^0\bigl(C, L \otimes A^{-e-1}\bigr) = m_1$ and $h^0\bigl(C, L \otimes A^{-e}\bigr) = 2m_1 + m_2$ imply, as explained in Proposition~\ref{prop:degloc_2}, that $h^0(C, L) \geq r+1$. This shows that $V^r_{d,\ell}(C,A)$ can be realized as a degeneracy locus over (an open subset of) the Brill–Noether variety $W_{d - (e+1)k}^{m_1 - 1}(C)$, associated to the vector bundle morphism globalizing the multiplication maps
\[
H^0(C, A) \otimes H^0\bigl(C, \omega_C \otimes A^{e+1} \otimes L^{\vee} \bigr)
\longrightarrow 
H^0\bigl(C, \omega_C \otimes A^{e+2} \otimes L^{\vee} \bigr).
\]
The expected dimension of this degeneracy locus is precisely $\rho(g, r - \ell, d) - \ell k$.

\vspace{3mm}

Returning to the setup of elliptic $K3$ surfaces, our next result describes the geometry of the varieties $V^r_{d,\ell}(C,A)$ for curves in elliptic $K3$ surfaces, which in particular completes the proof of Theorem~\ref{thm.intro4}.

\begin{Thm}\label{thm.intro5}  
Let $X$ be a general $K3$ surface with $\mathrm{Pic}(X)\cong \mathbb Z\cdot H\oplus \mathbb Z\cdot E$ as above. Fix $d \leq g-1$, $r \geq 0$ and $\ell \in \bZ$ such that
$\max\{0, r+2-k\} \leq \ell \leq r$. Then for a general curve $C\in|H|$, the variety $$V^r_{d,\ell}\bigl(C,\cO_C(E)\bigr)$$
is smooth of the expected dimension $\rho(g, r-\ell, d)-\ell k$. In particular, it is empty if $\rho(g, r-\ell, d)-\ell k<0$.
\end{Thm}  

In order to prove Theorem \ref{thm.intro5} we apply Theorem~\ref{thm.intro3} to the balanced stability type $\overline{e} = \bigl((e+1, m_1), (e, m_2)\bigr)$ as defined in~\eqref{ell}. It follows that the moduli space $\cM_{H_\epsilon}(v, \overline{e})$ of $H_\epsilon$-Gieseker stable sheaves with Mukai vector $v = (0, H, 1 + d - g)$ and stability type $\overline{e}$ is non-empty. By Theorem~\ref{thm.intro2}, this moduli space is smooth, quasi-projective, and irreducible of the expected dimension; notably, this dimension equals $g + \rho(g, r - \ell, d) - \ell k$. In the second part of the proof of Theorem \ref{thm.intro5}, we consider the natural support map
\[
\cM_{H_\epsilon}(v, \overline{e}) \longrightarrow |H|.
\]
In Section \ref{sec:dominance}, we prove that this map is dominant by exhibiting reducible curves of the form $C + J$ in the linear system $|H|$, where $C \in |H - E|$ and $J \in |E|$ are chosen generically, such that the fiber over the point $[C + J] \in |H|$ contains a component of dimension $\rho(g, r - \ell, d) - \ell k$. Note that our inductive step alters the parameters according to the rule
\[
(g - k,\, d - k,\, r - 1,\, \ell - 1) \mapsto (g,\, d,\, r,\, \ell),
\]
and that under this change  the Brill–Noether number in question remains constant, that is, 
$\rho(g - k, r - \ell, d - k) - (\ell - 1)k = \rho(g, r - \ell, d) - \ell k$.
Applying Theorems~\ref{thm.intro1} and~\ref{thm.intro2} to curves in the linear system $|H - E|$, we reduce to the base case $\ell = 0$, which must be treated separately; see Theorem~\ref{thm:existence_limit}. Dominance of the support map, when coupled with a close relation between balanced Bridgeland stability types and the non-negative part of the splitting type (see Subsection \ref{subsec:comparison}), is enough to derive Theorem \ref{thm.intro5}.

\vspace{2mm}

As a consequence of Theorem \ref{thm.intro5}, we have a precise description of the loci $V^r_{d, \ell}(C, A)$ for a general point of the Hurwitz space.

\vspace{1mm}
\begin{Cor}
Fix $d \leq g-1$, $r \geq 0$ and $\ell \in \bZ$ such that
$\max\{0, r+2-k\} \leq \ell \leq r$. Then for a general point $[C, A]\in \cH_{g,k}$, the variety $V^r_{d,\ell}\bigl(C, A\bigr)$
is smooth of the expected dimension $\rho(g, r-\ell, d)-\ell k$. 
\end{Cor}

It is an interesting question to understand in full generality the relation between the Bridgeland stability type and (the non-negative part of) the splitting type of a line bundle on an integral curve $C \in |H|$. There are indications that these two sets invariants might not always coincide. A more conservative  guess would be that for any stability type $\overline{e} = \bigl((e_1, m_1), \ldots, (e_p, m_p)\bigr)$ and any integral curve $C\in|H|$, the two loci of sheaves $L \in \overline{\mathrm{Pic}}^d(C)$ such that $i_*L$ has stability type $\overline{e}$, respectively those such the non-negative part of the splitting type of $L$ equals
$(\underbrace{e_p,\ldots,e_p}_{m_p},\ldots,\underbrace{e_1,\ldots,e_1}_{m_1})$,
to have the same closure in $\overline{\mbox{Pic}}^d(C)$. We prove this property for balanced stability types in Proposition \ref{comparisontypes}.

\vspace{3mm}

\subsection{Hurwitz-Brill-Noether theory over number fields}
It is now natural to ask whether one can write down a smooth Hurwitz-Brill-Noether general $k$-gonal curve of genus $g$ defined over a number field, or even over $\mathbb Q$. In Section \ref{sec:numberfield} we explain how the results of this paper can be used to solve this problem.

\vskip 3pt

A Halphen surface of degree $k\geq 2$ is obtained by blowing-up $\mathbf{P}^2$ at points $p_1, \ldots, p_9$ lying on a unique plane cubic $J$, such that  $p_1+\cdots+p_9\in J$ is a torsion point of order $k$  (with respect to the group law of $J$). Set $X:=\mbox{Bl}_{\{p_1, \ldots, p_9\}}(\mathbf{P^2})$ and let $h$ be the hyperplane class on $X$ and $E_1, \ldots, E_9$ the exceptional divisors. Following \cite{arbarello-explicit}, 
$$\Lambda_g:=\bigl|3gh-gE_1-\cdots-gE_8-(g-1)E_9\bigr|$$ is then a linear system of curves of genus $g$ on $X$. Note that $h^0\bigl(C, \cO_C(kJ)\bigr)\geq 2$, that is, every curve $C\in \Lambda_g$ is endowed with a degree $k$ pencil. The following result combines Theorems \ref{thm:duval} and \ref{prop:hbn_rationals}:

\begin{Thm}\label{thm.intro7}
A pair $\bigl[C, \cO_C(kJ)\bigr]\in \cH_{g,k}$, where $C\in \Lambda_g$ is a general curve, verifies the Hurwitz-Brill-Noether theorem, that is, $$\mathrm{dim }\  W^r_d(C)=\rho_k(g,r,d).$$
In particular, for every prime $k$, there exist Hurwitz-Brill-Noether general curves of genus $g$ defined over a number field $K$ with $[K:\mathbb Q]\leq k^2-1$.
\end{Thm}
It remains a very interesting open question whether one can write down Hurwitz-Brill-Noether general $k$-curves of genus $g$ defined over the rationals.

\vspace{3mm}

\subsection*{Acknowledgements} We thank Marian Aprodu, Arend Bayer, Samir Canning, Bert van Geemen, Andreas Leopold Knutsen, Eric Larson, Hannah Larson, Margherita Lelli-Chiesa, Bjorn Poonen, Richard Thomas and Xiaolei Zhao for many helpful discussions.

Farkas was supported by the Berlin Mathematics Research Center MATH+ and the ERC Advanced Grant \textsc{SYZYGY} (No.~834172). Feyzbakhsh was supported by the Royal Society URF/R1/231191. Rojas was supported by the ERC Advanced Grant \textsc{SYZYGY} (No.~834172).

\section{Hurwitz-Brill-Noether loci}\label{HBNloci}

In this  introductory section, we associate to every smooth curve $C$ endowed with a degree $k$ pencil $(A,V)\in G^1_k(C)$ the Hurwitz-Brill-Noether loci $W^r_{d, \ell}(C, A)$,  then we explain a novel derivation of their expected dimension as degeneracy loci over classical Brill-Noether loci associated to $C$.   

\vskip 4pt

We fix integers $g, k\geq 2$ and denote by $\widetilde{\mathcal{H}}_{g,k}$ the Hurwitz stack classifying triples consisting of a stable curve $C$ of genus $g$, a line bundle $A\in \mbox{Pic}^k(C)$ and a $2$-dimensional space of sections $V\subseteq H^0(C, A)$. When $V=H^0(C,A)$, we refer to points in $\widetilde{\cH}_{g,k}$ as pairs $[C, A]$. It is well known that $\widetilde{\mathcal{H}}_{g,k}$ is irreducible of dimension $2g+2k-5=3g-3+\rho(g,1,k)$. Let $\mathcal{H}_{g,k}$ be the open substack of $\widetilde{\cH}_{g,k}$ corresponding to triples $[C,A,V]$ as above, where $C$ is a smooth curve. We shall use several times that as long as $g\geq a(k-1)$, we have $\mbox{Sym}^a H^0(C,A)\cong H^0(C, A^{a})$ for a general $[C, A]\in \mathcal{H}_{g,k}$.

\vskip 4pt

We fix  a non-negative integer $\ell$ with $\mbox{max}\{0, r-k+2\} \leq \ell  \leq r$ and such that 
\begin{equation}\label{eq:ineq11}
\rho(g, r-\ell,d)-\ell k\geq 0.
\end{equation}

We aim to study the loci $W^r_d(C)$ for a triple $[C, A, V]\in \mathcal{H}_{g,k}$. Without loss of generality, by Riemann-Roch, we may assume $d\leq g-1$. Set $e:=\bigl\lfloor \frac{\ell}{r+1-\ell}\bigr\rfloor$ and write 

\begin{equation}\label{eq:a1}
\ell=e(r+1-\ell)+m_1,
\end{equation}
where $0\leq m_1\leq r-\ell$. Setting $m_2:=r+1-\ell-m_1\geq 1$, we then also have 
$$r+1=m_1(e+2)+m_2(e+1).$$

\vskip 4pt

\begin{Prop}\label{ineqBN}
Assume inequality (\ref{eq:ineq11}) holds. Then for every smooth curve $C$ of genus $g$, the locus $W_{d-(e+1)k}^{m_1-1}(C)$ is non-empty. Furthermore, $d-(e+1)k\geq 0$.
\end{Prop}
\begin{proof}
We apply the main existence results of Brill-Noether theory, see e.g. \cite{fulton-lazarsfeld} and to that end we need to show that the expected dimension $\rho\bigl(g, m_1-1, d-(e+1)k\bigr)$ of the determinantal variety $W^{m_1-1}_{d-(e+1)k}(C)$ is non-negative. Indeed, one has 
$$\rho\bigl(g, m_1-1, d-(e+1)k\bigr)=\rho(g, r-\ell,d)-\ell k+(r-\ell-m_1+1)(g+ek-d+r-\ell+m_1).$$
Note that $r-\ell-m_1\geq 0$ by (\ref{eq:a1}), whereas 
$$g+ek-d+r-\ell+m_1=g-d+r+ek-e(r+1-\ell)\geq g-d+r-\ell\geq 1,$$
since we assumed that $d\leq g-1$. Since by assumption $\rho(g, r-\ell, d)-\ell k\geq 0$, it follows from \cite{fulton-lazarsfeld} that $W^{m_1-1}_{d-(e+1)k}(C)$ is non-empty and of dimension at least $\rho\bigl(g, m_1-1, d-(e+1)k\bigr)$. Clearly, since $\rho\bigl(g, m_1-1, d-(e+1)k\bigr)\geq 0$, it also follows that $d-(e+1)k\geq 0$.
\end{proof}

\begin{Prop}\label{prop:bpf_trick}
Assume $(A,V)\in G^1_k(C)$ is a base point free pencil and let $M$ be a line bundle on $C$. Then for any $a>0$ we have the inequality
$$h^0\bigl(C, M\otimes A^{a+1}\bigr) \geq (a+1)\cdot h^0(C, M\otimes A)-a\cdot h^0(C,M).$$    
\end{Prop}
\begin{proof}
We prove this via the base point free pencil trick \cite[p. 126]{arbarello:geometry-of-algebraic-curves}. Since $V$ is a base point free pencil on $C$, for $m\leq a$ the kernel of the multiplication map $V\otimes H^0(C, M\otimes A^{m})\longrightarrow H^0(C, M\otimes A^{m+1})$ can be identified with $H^0(C, M\otimes A^{m-1})$. Therefore
\[
h^0(C, M \otimes A^{m+1}) - h^0(C, M \otimes A^m) \geq h^0(C, M \otimes A^m) - h^0(C, M \otimes A^{m-1}),
\]
so summing over \( m = 1, \ldots, a \) implies that  
\[
h^0(C, M \otimes A^{a+1}) - h^0(C, M \otimes A) \geq a \big(h^0(C, M \otimes A) - h^0(C, M)\big),
\]  
as claimed.
\end{proof}

\vskip 4pt

Applying Proposition \ref{prop:bpf_trick} for $a=e+1$, if $M\in W^{m_1-1}_{d-(e+1)k}(C)$ is such that 
    \begin{equation}\label{eq:conditions}
h^0(C,M)=m_1  \ \ \mbox{ and  } \ \ h^0(C, M\otimes A)\geq r-\ell+m_1+1=2m_1+m_2, 
\end{equation}
where we recall that $m_1$ is given by (\ref{eq:a1}), after
setting $L:=M\otimes A^{e+1}\in \mbox{Pic}^d(C)$,
we have $$h^0(C,L)\geq (e+1)(r-\ell+m_1+1)-e m_1=e(r-\ell+1)+m_1+(r-\ell+1)=r+1.$$

\vskip 4pt

\begin{Def}\label{def:wtilde}
Let $\widetilde{W}^{m_1-1}_{d-(e+1)k}(C)$ be the union of all components of $W^{m_1-1}_{d-(e+1)k}(C)$ having the expected dimension $\rho\bigl(g, m_1-1, d-(e+1)k\bigr)$ and corresponding to a general point $M$ satisfying $H^0(C, M\otimes A^{\vee})=0$
\end{Def}

Applying \cite[Lemma 3.5]{arbarello:geometry-of-algebraic-curves}, a general point $M$ of each component of $W^{m_1-1}_{d-(e+1)k}(C)$ satisfies $h^0(C, M)=m_1$. Furthermore, from (\ref{eq:a1}) we have $m_1-1\leq r-\ell-1\leq k-3$, hence by applying \cite[Proposition 2.3.1]{coppens-martens} (see also Proposition \ref{prop:coppens-martens}), we obtain that 
$\widetilde{W}^{m_1-1}_{d-(e+1)k}(C)\neq \emptyset$ for a general element $[C,A]\in \mathcal{H}_{g,k}$.

\vskip 3pt

\begin{Rem}
For a general $(C, A)\in \mathcal{H}_{g,k}$, the equality $W^{m_1-1}_{d-(e+1)k}(C)=\widetilde{W}_{d-(e+1)k}^{m_1-1}(C)$ does not hold whenever $m_1\geq 2$. In fact, using the identity
$$d-(e+m_1)k=\rho\bigl(g, m_1-1, d-(e+1)k\bigr)+(m_1-1)(g+ek+m_1-d)$$
we obtain via (\ref{eq:a1}) that $d-(e+m_1)k\geq 0$. In particular line bundles of the form $M=A^{m_1-1}(D)$, with $D$ being an effective divisor on $C$ of degree $d-(e+m_1)k$, satisfy $h^0(C, M)\geq m_1$. Such line bundles depend on 
$d-(e+m_1)k>\rho \bigl(g, m_1-1, d-(e+1)k\bigr)$ parameters and they lie in $W^{m_1-1}_{d-(e+1)k}(C)\setminus \widetilde{W}^{m_1-1}_{d-(e+1)k}(C)$.
\end{Rem}

\vskip 4pt

The condition (\ref{eq:conditions}) defines a determinantal subvariety of (an open subvariety of) the locus $\widetilde{W}^{m_1-1}_{d-(e+1)k}(C)$, as we shall explain. 

\begin{Def}\label{def:Vm}
For a smooth curve $C$, let $V_{d-(e+1)k}^{m_1-1}(C)$ be the subset of $\mbox{Pic}^{d-(e+1)k}(C)$ parametrizing bundles $M$ such that $h^0(C,M)=m_1$ and $h^0(C, M\otimes A^{\vee})=0$.
\end{Def}

For  $M\in V^{m_1-1}_{d-(e+1)k}(C)$, the Riemann-Roch theorem gives:
\[
h^0(C, \omega_C \otimes M^{\vee}) = g + m_1 - 1 + (e+1)k - d, \quad  
h^0(C, \omega_C \otimes M^{\vee} \otimes A) = g + (e+2)k - d - 1.
\]

\vskip 4pt

We can now take a global version of Definition \ref{def:Vm} and we denote by 
\begin{equation}\label{eq:nu}
\nu\colon \cV_{d-(e+1)k}^{m_1-1}\longrightarrow \cH_{g,k}
\end{equation} the stack of elements $[C, A, V, M]$, where $[C, A, V]\in \cH_{g,k}$ and $M\in V_{d-(e+1)k}^{m_1-1}(C)$.
Over $\cV_{d-(e+1)k}^{m_1-1}$ we have two tautological vector bundles $\cE$ and $\cF$ with fibres over a point $[C, A, V, M]$ given by
$$\cE_{|[C, A, M]}=H^0(C, \omega_C\otimes M^{\vee})\  \mbox{ and } \ \ \cF_{|[C, A, M]}=H^0(C, \omega_C\otimes M^{\vee}\otimes A).$$ The local freeness of $\cE$ and $\cF$ follows from Grauert's theorem. 
As explained above, $\mbox{rk}(\cE)=g+m_1-1+(e+1)k-d$ and $\mbox{rk}(\cF)=g+(e+2)k-d-1$. Let $\mathbb E$ be the tautological rank $2$ vector bundle over $\cH_{g,k}$ with fibres $\mathbb E_{|[C, A, V]}=V$, for a point $[C, A, V]\in \cH_{g,k}$.

\vskip 3pt

There is a morphism of vector bundles 
\begin{equation}\label{eq:deglocus_BNH}
\phi \colon \nu^*\mathbb E\otimes \cE\longrightarrow \cF
\end{equation}
whose fibre over a point  $[C, A, V, M]$ is the multiplication map 
$$\phi_{C,A,V,M}\colon V\otimes H^0(C, \omega_C\otimes M^{\vee})\longrightarrow H^0(C, \omega_C\otimes M^{\vee}\otimes A).$$
We set $s:=g-d+r+e(k-r-1+\ell)$  and denote by $\mathfrak{Deg}(\phi)$ the degeneracy locus of the morphism $\phi$ consisting of those  $[C, A, V, M]$ such that $\mbox{dim } \mbox{Ker}(\phi_{C,A,M})\geq s$. For a point $[C, A, V]\in \cH_{g,k}$, we set $\mathfrak{Deg}_{(C,A, V)}(\phi):=\mathfrak{Deg}(\phi)\cap \nu^{-1}\bigl([C,A, V]\bigr)\subseteq V_{d-(e+1)k}^{m_1-1}(C)$.

\begin{Prop}\label{prop:degloc_2}
For each $[C,A, V]\in \cH_{g,k}$ one has an injective  map 
$$\mathfrak{Deg}_{(C,A,V)}(\phi)\longrightarrow W^r_d(C) \ \ \mbox{  given by} \  \  \ M\mapsto M\otimes A^{e+1}.$$ Moreover, if $\mathfrak{Deg}_{(C,A, V)}(\phi)$ is non-empty, then all of its components have dimension at least $\rho(g, r-\ell, d)-\ell k$.    
\end{Prop}
\begin{proof}
Observe that $M\in V_{d-(e+1)k}^{m_1-1}(C)$ satisfies (\ref{eq:conditions}) if and only if $[C,A,V,M]\in \mathfrak{Deg}(\phi)$. Indeed, via the Base Point Free Pencil Trick $\mbox{Ker}(\phi_{C, A, V,M})\cong H^0(C, \omega_C\otimes M^{\vee}\otimes A^{\vee})$ and by Riemann-Roch we obtain that 
$$h^0(C, A\otimes M)=\mbox{dim } \mbox{Ker}(\phi_{C, A, V, M})+d-ek+1-g,$$ from which the conclusion follows. From the general theory of degeneracy loci, see e.g. \cite{arbarello:geometry-of-algebraic-curves}, every component of 
$\mathfrak{Deg}_{(C,A,V)}(\phi)$ has dimension at least 
\begin{align*}
\mbox{dim } \widetilde{W}^{m_1-1}_{d-(e+1)k}(C) - s \cdot \bigl(\text{rk}(\cF) - 2 \cdot \text{rk}(\cE) + s \bigr) 
=g - m_1(g - d + (e + 1)k + m_1 - 1) \\
- \left( g - d + r + e(k - r - 1 + \ell) \right)(r - \ell - m_1 +1) = \rho(g, r - \ell, d) - \ell k.
\end{align*}
\end{proof}

\begin{Def}\label{HBNdegloci}
 For a point $[C, A, V]\in \cH_{g,k}$, let $W^r_{d, \ell}(C, A)$ be the closure of the image of $\mathfrak{Deg}_{(C,A)}(\phi)$.
Inside $W^r_{d, \ell}(C, A)$ we identify the following open subvariety  
\begin{equation}\label{eq:defl}
\begin{aligned}
V_{d, \ell}^r(C, A):=\Bigl\{L\in\Pic^d(C): h^0(C, L\otimes A^{-e-1})= m_1,\;\   h^0(C, L\otimes A^{-e-2})=0,\\ \ h^0(C, L\otimes A^{-e})=  2m_1+m_2, \;\ h^0(C,L)=r+1 \Bigr\}.
\end{aligned}
\end{equation}
\end{Def}
Although each component of $W^r_{d, \ell}(C, A)$ is of dimension at least $\rho(g, r-\ell, d)-\ell k$, Proposition \ref{prop:degloc_2} does not establish the nonemptiness of $W^r_{d, \ell}(C, A)$ when inequality (\ref{eq:ineq11}) is satisfied. However, by semicontinuity, we have the following:

\begin{Prop}\label{prop:semicont}
Assume $(\ref{eq:ineq11})$ is satisfied. If there exists a point $[C_0,A_0, V_0]\in \widetilde{\cH}_{g,k}$ such that $W^r_{d, \ell}(C_0, A_0, V_0)$ has a component of dimension $\rho(g, r-\ell,d)-\ell k$, then $W^r_{d, \ell}(C,A, V)\neq \emptyset$ for a general element $[C, A, V]\in \widetilde{\cH}_{g, k}$.    
\end{Prop}


\vskip 3pt

\begin{Rem}
By applying Proposition \ref{prop:bpf_trick}, clearly $W^r_{d,\ell}(C,A)\subseteq W^r_d(C)$. However it is not the case that the inclusion $W^{r-1}_{d,\ell}(C,A)\subseteq W^r_{d,\ell}(C,A)$ necessarily holds. This is due to the fact that the description (\ref{eq:ineq11}) involves the parameter $m_1$ and it is possible that $m_1$ decreases as $r$ increases.
\end{Rem}

\begin{Rem}\label{rem:l=r}
Some of the loci in Definition \ref{HBNdegloci} have a transparent description, others less so. For instance, when $\ell=r$, then $W^r_{d, r}(C,A)$ can be identified with the translate $A^r+W_{d-rk}(C)\subseteq \mbox{Pic}^d(C)$. Clearly, both $W_{d,r}^r(C,A)$ and $V_{d,r}^r(C,A)$ have the expected dimension $d-rk$ predicted by Proposition \ref{prop:degloc_2} for a general $[C,A]\in \cH_{g,k}$, whenever $rk\leq d\leq g-1$.  
\end{Rem}

\subsection{Brill-Noether theory via splitting types}
We recall the definition of the splitting type of a linear system following \cite{cook-powell-jensen}  and \cite{larson:inv}.
Later we shall establish some connections between this invariant and the stability type of a linear system on a curve on an elliptic $K3$ surface, which is defined in terms of Bridgeland stability. We fix an integral curve $C$ and a finite map
\[
\pi\colon C \longrightarrow \mathbf{P}^1
\]
of degree $k$. If $L\in W^r_d(C)$, then $\pi_*L$ is a rank $k$ vector bundle on $\mathbf{P}^1$ and it splits as a direct sum of line bundles
\begin{equation}\label{eq:splitting}
\pi_*L\cong\cO_{\mathbf{P}^1}(f_1)^{\oplus n_1}\oplus \cdots \oplus\cO_{\mathbf{P}^1}(f_q)^{\oplus n_q}
\end{equation}
where $f_1>\cdots >f_q$, $n_i>0$ and $n_1+\cdots+n_q=k$. The collection $\overline{f}_L:=\bigl((f_i,n_i)\bigr)_{i=1}^q$ is called the \emph{splitting type} of $L$ (with respect to $\pi$).
Note that $f_1 n_1+\cdots+n_q f_q=d+1-g-k$, in particular $f_q<0$. With this notation, if $L\in V^r_{d, \ell}(C)$ then the non-negative part $\bigl(\pi_*L\bigr)^{\geq 0}$ of $\pi_*L$ in the splitting (\ref{eq:splitting}) is given by 
$$\bigl(\pi_*L\bigr)^{\geq 0}\cong \cO_{\mathbf{P}^1}(e+1)^{\oplus m_1}\oplus\cO_{\mathbf{P}^1}(e)^{\oplus m_2}.$$
In particular, $m_1+m_2=r+1-\ell$ is the rank of $\bigl(\pi_*L\bigr)^{\geq 0}$. Since $\pi_*(L)$ must also have at least one negative summand, we obtain that $r+1-\ell\leq \mbox{rk}(\pi_*L)-1=k-1$, that is, $\mbox{max}\{0, r+2-k\}\leq \ell$, which is precisely our original assumption on $\ell.$

\section{Stability conditions on elliptic $K3$ surfaces}

\subsection{Degree $k$ elliptic $K3$ surfaces} The first step in our analysis is the following observation, which using the surjectivity of the period map for $K3$ surfaces, produces smooth curves endowed with a pencil of prescribed degree.

\begin{Prop}\label{thm-k3 surface}
Fix integers $g \geq 3$ and  $k\geq 2$. Then there exists a smooth $K3$ surface $X$ with $\Pic(X)\cong  \bZ \cdot H \oplus \bZ\cdot E$ such that 
\begin{enumerate}
\item $H^2 = 2g-2$, $E\cdot H = k$ and $E^2 =0$, 
\item $E$ is a smooth, irreducible elliptic curve,
\item $H$ is ample and base point free. 
\end{enumerate} 
\end{Prop}

\begin{proof} This is well known to the experts. By \cite[Lemma~2.2]{knutsen-gonality-k3curves} one only needs to show ampleness of $H$. Since $H$ is nef and lies in the positive cone, it suffices to check $H\cdot R >0$ for all $(-2)$-curves $R$ on $X$. Such curves exist only when $k|g$, and are of class $H- \frac{g}{k}·E$ (to discard effectiveness of $\frac{g}{k}·E-H$, one uses that any divisor in the linear system $\bigl|\frac{g}{k}·E\bigr|$ is a union of $\frac{g}{k}$ curves in the elliptic pencil $|E|$, cf. \cite[Proposition 2.6]{saint:projective-models-of-k3-surfaces}). Hence $H\cdot R=g-2>0$ as required.
\end{proof}

In the sequel, we fix integers $g\geq 3$ and $k\geq 2$, and a $K3$ surface $X$ as given in Proposition \ref{thm-k3 surface}.
Throughout the paper we refer to such an $X$ as a \emph{degree $k$ elliptic $K3$ surface}. We have already used \cite[Proposition 2.6]{saint:projective-models-of-k3-surfaces} guaranteeing the reducibility of any curve in a linear system $|qE|$ with $q\geq2$. For later use, we summarize this statement:

\begin{Lem}\label{multpencil} For any $q\geq 1$, the following hold:
    \begin{enumerate}
        \item\label{symisom} The natural map $\mathrm{Sym}^q H^0(\cO_X(E))\longrightarrow H^0(\cO_X(qE))$ is an isomorphism. In particular, $h^0(X, \cO_X(qE))=q+1$ and $h^1(X, \cO_X(qE))=q-1$.
        \item\label{evmatrix} Under this isomorphism, there is a short exact sequence
        \begin{equation}\label{evalutationseq}
            0\longrightarrow \cO_X(-E)^{\oplus q}\overset{i}{\longrightarrow}\mathrm{Sym}^q H^0(\cO_X(E))\otimes\cO_X\cong \cO_X^{\oplus q+1}\overset{\mathrm{ev}}{\longrightarrow} \cO_X(qE)\longrightarrow 0
        \end{equation}
        where $\mathrm{ev}$ is the evaluation of global sections. More precisely, given a basis $(s,t)$ of $H^0(X, \cO_X(E))$, the maps $\mathrm{ev}$ and $i$ are respectively given by the matrices
        \[
        \begin{pmatrix}
s^q & s^{q-1}t & ... & st^{q-1} & t^q
\end{pmatrix},\;\;
\begin{pmatrix}
t & 0 & 0 &  ... & 0 & 0 \\
-s & t & 0 &  ... & 0 & 0 \\
0 & -s & t &   ... & 0 & 0 \\
0 & 0 & -s &   ... & 0 & 0 \\
\vdots &  \vdots & \vdots &  ... & \vdots & \vdots \\
0 & 0 & 0 &   ... & -s & t \\
0 & 0 & 0 &   ... & 0 & -s
\end{pmatrix}.
        \]
    \end{enumerate}
\end{Lem}

\subsection{Multiples of pencils on curves on $K3$ surfaces}
It is known that for a general element $[C, A]\in \mathcal{H}_{g,k}$,  the dimensions $h^0(C, A^a)$ are as small as possible, see \cite{ballico} or \cite[Proposition 2.1.1]{coppens-martens}. Those (deformation-theoretic) proofs cannot be easily extended to cover the case of curves lying on a $K3$ surface. The following statement is of independent interest and will turn out to be of importance in showing the non-emptiness of the variety $W^r_{d, r+2-k}(C,A)$, if $C$ lies on an elliptic $K3$ surface.

\begin{Thm}\label{thm:ballico}
Let $X$ be a degree $k$ elliptic $K3$ surface with $\mathrm{Pic}(X)\cong \mathbb Z\cdot H\oplus \mathbb Z\cdot E$. Denoting by $E_C$ the restriction $\cO_C(E)$, then for a general curve $C\in |H|$ we have that 
$$h^0\bigl(C, E_C^a)=\mathrm{max}\bigl\{ak+1-g, a+1\bigr\}.$$
\end{Thm}

Theorem \ref{thm:ballico} states that $h^0(C, E_C^a)=a+1$ for $g\geq a(k-1)$, whereas for $g<a(k-1)$ one has $H^1(C, E_C^a)=0$, therefore by Riemann-Roch 
$h^0\bigl(C, E_C^a\bigr)=ak+1-g$.

\begin{proof}
Let us write $g=n(k-1)+b$, where $b\leq k-2$. We prove by induction on $n$ that for a general element $C\in |H|$ one has 
\begin{equation}\label{eq:dim_multiples}
h^0\bigl(C, E_C^a\bigr)=\mbox{max}\{ak+1-g, a+1\}.
\end{equation}

Assume first $g\leq k-1$. We pick any smooth curve $C\in |H|$ and consider the twist of the short exact sequence $0\longrightarrow \cO_X(-C)\longrightarrow \cO_X\longrightarrow \cO_C\longrightarrow 0$ by $\cO_X(E)$. Taking cohomology we obtain the exact sequence
$$ H^1(\cO_X(E))\longrightarrow H^1(C, E_C)\longrightarrow H^2(\cO_X(E-C)).$$
Since $(C-E)^2=2g-2-2k\leq -4$, it follows that $0=h^0(\cO_X(C-E))=h^2(\cO_X(E-C))$. Together with the vanishing $h^1(\cO_X(E))=0$, this yields $h^1(E_C)=0$ which proves the case $a=1$. For $a\geq2$, observe that $h^1(E_C^a)=0$ since
\[
H^1(E_C^a)^\vee\cong H^0(\omega_C\otimes E_C^{-a}))\subset H^0(\omega_C\otimes E_C^{-1})\cong H^1(E_C)^\vee=0,
\]
which proves (\ref{eq:dim_multiples}) for $g\leq k-1$.

\vskip 3pt

Assume now $g\geq k$, in which case $|H-E|\neq \emptyset$. We aim to establish (\ref{eq:dim_multiples}) for a general curve $Y\in |H|$. We pick general curves $C\in |H-E|$ and $J\in |E|$, and write $C\cdot J=x_1+\cdots+x_k$. We choose a smooth curve $Y_0\in |H|$ and consider the pencil in $|H|$ spanned by $Y_0$ and $C+J$. Write 
$$Y_0\cdot J=y_1+\cdots+y_k \ \mbox{ and } \ Y_0\cdot C=z_1+\cdots+z_{2g-2-k}.$$ 

Let $\epsilon \colon \widetilde{X}\longrightarrow X$ be the blow-up of $X$ at the $2g-2$ points $y_1, \ldots, y_k, z_1, \ldots, z_{2g-2-k}$, and denote by $E_{y_1}, \ldots, E_{y_k}$ and $E_{z_1}, \ldots, E_{z_{2g-2-k}}$ the corresponding exceptional divisors on $\widetilde{X}$. We write $C'$ (resp.~$J'$) for the strict transform of $C$ (resp.~$J$). We then have a fibration $f\colon \widetilde{X}\rightarrow \mathbf{P}^1$ induced by the linear system $$\bigl|\epsilon^*(H)-E_{y_1}-\cdots-E_{y_k}-E_{z_1}-\cdots-E_{z_{2g-2-k}}\bigr|$$ and having the curve $Y_1:=C'+J'$ as a fibre. To simplify notation, we identify $C$ with $C'$ and $J$ with $J'$ on $\widetilde{X}$.

\vskip 4pt

Assume (\ref{eq:dim_multiples}) fails for a given $a\in \mathbb N$, for every curve $Y\in |H|$. Then by semicontinuity
\begin{equation}\label{semicont}
h^0\bigl(Y_1, \epsilon^*(E^a)(-J)\bigr)>\mathrm{max}\{a+1, ak+1-g\}. 
\end{equation}
Note that one has the canonical identification
\begin{equation}\label{eq:twist_j}
H^0\bigl(Y_1,\epsilon^*(E^a)(-J)\bigr)\cong \mbox{Ker}\Bigl\{ H^0\bigl(C, E_C^{a-1}\bigr)\oplus H^0\bigl(J, \cO_J(x_1+\cdots+x_k)\bigr)\stackrel{\mathrm{ev}}\longrightarrow \mathbb C_{x_1, \ldots, x_k}\Bigr\}.
\end{equation}

\vskip 3pt

If $b=0$, thus $g=n(k-1)$, we write $g(C)=g-k=(n-2)(k-1)+k-2$. If $a\geq n+1$, then $a-2\geq n-1$ and by induction we observe that 
$h^0\bigl(C, E_C^{a-1}\bigr)-h^0\bigl(C, E_C^{a-2}\bigr)=k$, hence the evaluation $\mathrm{ev}\colon H^0\bigl(C, E_C^{a-1}\bigr)\longrightarrow \mathbb C_{x_1, \ldots, x_k}$ is surjective. Using (\ref{eq:twist_j}) we obtain 
\[
h^0\bigl(Y_1, \epsilon^*(E^a)(-J)\bigr)=h^0\bigl(C, E_C^{a-1}\bigr)+h^0\bigl(J, \cO_J(x_1+\cdots+x_k)\bigr)-k=ak+1-g,
\]
which contradicts \eqref{semicont}.  If on the other hand $a\leq n-1$, then $a-1\leq n-2$ and from (\ref{eq:twist_j}) we obtain $h^0\bigl(Y_1, \epsilon^*(E^a)(-J)\bigr)\leq h^0\bigl(C, E_C^{a-1}\bigr)+1=a+1$, contradicting \eqref{semicont} again.

\vskip 3pt

If $b\geq 1$, we have $g'=g(C)=g-k=(n-1)(k-1)+b-1$, and by induction (\ref{eq:dim_multiples}) holds for $C$. If $a\leq n$, then $a-1\leq n-1$ and using (\ref{eq:twist_j}) we obtain 
$$h^0\bigl(Y_1, \epsilon^*(E^a)(-J)\bigr)\leq h^0\bigl(C, E_C^{a-1})+1=a+1,$$ which violates (\ref{semicont}). If $a\geq n+2$, then by induction $h^0\bigl(C, E_C^{a-1}\bigr)=(a-1)k+1-g+k=ak+1-g$
and the map $\mathrm{ev}\colon H^0(C, E_C^{a-1})\rightarrow \mathbb C_{x_1, \ldots, x_k}$ is surjective, which implies that 
the map $\mathrm{ev}\colon H^0(C, E_C^{a-1})\rightarrow \mathbb C_{x_1, \ldots, x_k}$ is surjective as well. It follows from \eqref{eq:twist_j} that $h^0\bigl(Y_1, \epsilon^*(E^a)(-J)\bigr)=h^0(C, E_C^{a-1})=ak+1-g$, contradiction.

\vskip 3pt

This reasoning leaves the cases (i) $b=0, a=n$, respectively, (ii) $b\geq 1, a=n+1$ uncovered. Since they are handled in an essentially identical way, we describe in details only case (i). Thus $g=n(k-1)$ and by the induction hypotheses $h^0\bigl(C, E_C^{n-1}\bigr)=n+1$, while $h^0\bigl(C, E_C^{n-2}\bigr)=n-1$. In order to conclude,  it suffices to show that the map $\mathrm{ev}$ in (\ref{eq:twist_j}) is surjective.

\vskip 3pt

We specialize further and choose $Y_0\in |H|$ to be a smooth curve general with respect to the property that it passes through the points $x_3, \ldots, x_k$. In this case, keeping the notation above, $y_j=z_j=x_j$ for $j=3, \ldots, k$ and the fibration $f\colon \widetilde{X}\rightarrow \mathbf{P}^1$ has the curve $Y_1:=C'+J'+E_{x_3}+\cdots+E_{x_k}$ as one of its fibres. Note $C'\cdot J'=x_1+x_2$ and $E_{x_i}\cdot C'=z_i$ and $E_{x_i}\cdot J'=y_i$ respectively, where $\epsilon(y_i)=\epsilon(z_i)=x_i$ for $i=3, \ldots, k$. Assuming (\ref{eq:dim_multiples}) fails for every curve in $|H|$, we have
$h^0\bigl(Y_1, \epsilon^*(E^n)(-J'-E_{x_3}-\cdots -E_{x_k})\bigr)\geq n+2$.

After identifying the curves $C'$ with $C$ and $J'$ with $J$ respectively, the restrictions of the twist $\beta:=\cO_{\widetilde{X}}\bigl(-J'-E_{x_3}-\cdots-E_{x_k}\bigr)$ to the components of $Y_1$ are
given by $$\beta_{| C} \cong \cO_C(-x_1-\cdots-x_k), \ \beta|_{J}\cong \cO_{J}(x_1+x_2) \  \mbox{ and } \beta_{|E_{x_i}}\cong \cO_{E_{x_i}}(1) \  \mbox{ for } i=3, \ldots, k.$$
Since each map $\mathrm{ev}_i\colon H^0\bigl(E_{x_i}, \cO_{E_{x_i}}(1)\bigr)\rightarrow \mathbb C_{z_i, y_i}$ is an isomorphism, by writing down the Mayer-Vietoris sequence on $Y_1$ one has the canonical identification 
$$H^0\bigl(Y_1, \epsilon^*(E^n)(-\beta)\bigr)=\mbox{Ker}\Bigl\{H^0\bigl(C, E_C^{n-1}\bigr)\oplus H^0\bigl(J, \cO_J(x_1+x_2)\bigr)\longrightarrow \mathbb C_{x_1, x_2}\Bigr\}.$$
Now observe that by varying $J\in |E|$ while fixing $C$, the points $x_1, x_2\in C$ can be chosen such that the evaluation map $H^0\bigl(C, E_C^{n-1})\rightarrow \mathbb C_{x_1, x_2}$ is surjective ($n\geq 2$), for else, we would obtain that $H^0\bigl(C, E_C^{n-2}\bigr)$ has codimension $1$ in $H^0\bigl(C, E_C^{n-1}\bigr)$, which is not the case. Therefore  $h^0\bigl(Y_1, \epsilon^*(E_C^{n-1})(-\beta)\bigr)=h^0(C, E_C^{n-1})=n+1$. This finishes the proof.
\end{proof}

As observed in \cite{coppens-martens} for general $k$-gonal curves, Theorem \ref{thm:ballico} implies the following:

\begin{Prop}\label{prop:coppens-martens}
Let $X$ be an elliptic $K3$ surface with $\mathrm{Pic}(X)\cong \mathbb Z\cdot H\oplus \mathbb Z\cdot E$ such that $H\cdot E=r+2$. Then for every $d\leq g-1$ such that $\rho(g,r,d)\geq 0$, there exists a component $Z$ of $W^r_d(C)$  having dimension $\rho(g,r,d)$, whose general point corresponds to a line bundle $L$ with $H^0(C, L\otimes E_C^{\vee})=0$.
\end{Prop}
\begin{proof}
We set $a:=g-d+r-1$. Note that $g\geq (a+1)(r+1)$, hence by using Theorem \ref{thm:ballico}, we have that $h^0(C, A^a)=a+1$, therefore $h^0(C, \omega_C\otimes A^{-a})=g-a(r+1)$. It follows that for a general effective divisor $D$ on $C$ of degree $\rho(g,r,d)=g-(a+1)(r+1)$, one has  $h^0\bigl(C,\omega_C\otimes A^{-a}(-D)\bigr)=r+1$.  Line bundles $L:=\omega_C\otimes A^{-a}(-D)$ of this type fill-up a component $Z$ of $W^r_d(C)$ of dimension 
$\rho(g,r,d)$, see also \cite[Proposition 2.3.1]{coppens-martens}. Since Theorem \ref{thm:ballico} also guarantees that $h^0(C, A^{a+1})=a+2$, we also have $h^0\bigl(C, \omega_C\otimes A^{-a-1}\bigr)=g-(a+1)(r+1)=\rho(g,r,d)$, therefore $$H^0\bigl(C, L\otimes A^{\vee}\bigr)=H^0(C, \omega_C\otimes A^{-a-1}(-D))=0,$$ for a general point $L\in Z$, corresponding to a general divisor $D$ of degree ${\rho(g,r,d)}$.
\end{proof}

\subsection{Derived categories of $K3$ surfaces} 
Let $\cD(X)$ denote the bounded derived category of coherent sheaves on $X$. For an object $F\in\cD(X)$, we will denote its Chern character, respectively, its Mukai vector by
\begin{gather*}
\ch(F):=\bigl(\ch_0(F),\ch_1(F),\ch_2(F)\bigr),\\
v(F):=
\ch(F)\cdot\sqrt{\mathrm{td}(X)}=\bigl(\ch_0(F),\ch_1(F),\ch_2(F)+\ch_0(F)\bigr).
\end{gather*}
Under the canonical identifications of $H^0(X,\bZ)$ and $H^4(X,\bZ)$ with $\bZ$, this rule defines surjective maps $\ch, v \colon K_0(\cD(X)) \rightarrow \Lambda := \bZ\oplus\Pic(X)\oplus\bZ$. We consider the following symmetric bilinear form on $\Lambda$:
\begin{equation}\label{def:quadraticform}
\begin{split}
\bigl\langle(r_1,x_1H+y_1E,s_1),(r_2,x_2H+y_2E,s_2)\bigr\rangle := \\
(x_1H+y_1E)\cdot (x_2H+y_2E)-r_1s_2-r_2s_1.
\end{split}
\end{equation}
By Riemann-Roch, for given $F,F'\in\cD(X)$ we have
\[
-\bigl\langle v(F),v(F')\bigr\rangle = \chi(F,F'):=\sum_i (-1)^i\cdot \ext^i(F,F'),
\]
where $\ext^i(F,F'):=\dim_\bC \Ext^i(F,F')$ (we will write $\hom(F,F'):=\dim_\bC \Hom(F,F')$ in the case $i=0$). We also recall that the quantity 
\[\Delta(F):=\ch_1(F)^2-2\ch_0(F)\ch_2(F)=v(F)^2+2\ch_0(F)^2
\]
is called the \emph{discriminant} of an object $F\in\cD(X)$. 

\vspace{1mm}

Let us fix $\epsilon \in \mathbb{Q}_{>0}$ 
and set 
\begin{equation}\label{polarisation}
H_{\epsilon} := E + \epsilon H\in\Pic(X)_\bQ.    
\end{equation}
As a result of Proposition \ref{thm-k3 surface}, the class $H_{\epsilon}$ lies in the ample cone. Consider the projection
\begin{align}\label{pro}
    \Pi_{\epsilon}: K_0(\cD(X))\setminus \{F:\ch_0(F)=0\} & \longrightarrow\bR^2\nonumber\\
    F &\longmapsto \left(\frac{H_{\epsilon}\cdot\ch_1(F)}{H_{\epsilon}^2\cdot\ch_0(F)},\frac{\ch_2(F)}{H_{\epsilon}^2\cdot \ch_0(F)}\right).
\end{align}
The following technical lemma plays a crucial role in the next subsection (as it will guarantee that \eqref{U} defines a set of Bridgeland stability conditions). It establishes that there is no sequence of spherical classes whose projections accumulate towards the origin. Recall that a class $F\in K_0(\cD(X))$ is called \emph{spherical} if $v(F)^2=-2$. 

\begin{Lem}\label{lem-no spherical}
There is no sequence of vectors $v_n=(r_n, \, t_nH+u_nE,\, s_n)\in\Lambda$ such that the following conditions hold:
\begin{enumerate}
    \item\label{lem-no spherical:a} $r_n \neq 0$ and $v_n \neq \pm(1, 0,1)$, 
    \item\label{lem-no spherical:b} $v_n$ is the Mukai vector of a spherical class, i.e. $v_n^2=(t_nH+u_nE)^2-2r_ns_n =-2$,
    \item\label{lem-no spherical:c} $\frac{H_{\epsilon}\cdot (t_nH+u_nE)}{r_n}\underset{n}{\longrightarrow}0$ and $\frac{s_n}{r_n} \underset{n}{\longrightarrow}1$. 
\end{enumerate}
\end{Lem}
\begin{proof}
    Write $\epsilon=\frac{a}{b}$ ($a,b\in\bZ_{>0}$). 
    If $t_n=0$, then $s_n= r_n = \pm 1$ by \eqref{lem-no spherical:b} and thus $u_n=0$ for all $n\gg0$ by \eqref{lem-no spherical:c}. Hence we may assume $t_n\neq 0$ for all $n$, so that \eqref{lem-no spherical:b} is equivalent to
    \begin{equation}\label{u-n}
    u_nk=\frac{r_ns_n}{t_n}-\frac{1}{t_n}-t_n(g-1).    
    \end{equation}
    
    We obtain
    \begin{align}\label{limit}
         b\cdot\frac{H_{\epsilon}\cdot (t_nH+u_nE)}{r_n}  =\frac{(aH+bE)(t_nH+u_nE)}{r_n}=\frac{at_n(2g-2)+au_nk+bt_nk}{r_n}=\nonumber\\
          \overset{\eqref{u-n}}{=}\frac{at_n(g-1)+\frac{ar_ns_n}{t_n}-\frac{a}{t_n}+bt_nk}{r_n}=\frac{t_n}{r_n}\left(a(g-1)+bk\right)+\frac{ar_n}{t_n}\Bigl(\frac{s_n}{r_n}-\frac{1}{r_n^2}\Bigr).
    \end{align}

Observe that $\frac{s_n}{r_n} - \frac{1}{r_n^2}>\frac{1}{2}$ for all $n\gg0$. Otherwise, using condition \eqref{lem-no spherical:c} we would have that $M:=\{n\;:\;r_n=s_n=\pm 1\}$ is an infinite set. But then for all $n\in M$
\begin{equation}\label{relation:t_n u_n}
    0=(t_nH+u_nE)^2=t_n(t_n(2g-2)+2u_nk)\;\Longrightarrow \;u_nk=-(g-1)t_n,
\end{equation}
and this is a contradiction as it implies
\[
        0=\lim_{n\in M}H_{\epsilon}\cdot(t_nH+u_nE)=\lim_{n\in M}\Bigl(t_nk+\epsilon(t_n(2g-2)+u_nk)\Bigr)\overset{\eqref{relation:t_n u_n}}{=}\lim_{n\in M} t_n(k+\epsilon(g-1)).
\]

Now since $a(g-1)+bk>0$ and the left-hand side in \eqref{limit} goes to zero, both terms on the right-hand side must also go to zero as $n \to \infty$. The first term implies that $\frac{t_n}{r_n} \to 0$, but this causes the second term to diverge to infinity which is a contradiction. 
\end{proof}

\subsection{A two-dimensional slice of Bridgeland stability conditions}
Given $\epsilon\in\bQ_{>0}$, we will work with a two-dimensional slice of stability conditions on $\cD(X)$ associated with the polarization $H_{\epsilon}$, for which we give a brief account. Further details can be found in Bridgeland's original work \cite{bridgeland:stability-condition-on-triangulated-category, bridgeland:K3-surfaces}.

Given a coherent sheaf $F$, we define its \emph{$H_\epsilon$-slope} by 
$$\mu_{H_{\epsilon}}(F):=\frac{H_{\epsilon}\cdot \ch_1(F)}{H_{\epsilon}^2\cdot \ch_0(F)},$$ with the convention $\mu_{H_{\epsilon}}(F)=+\infty$ if $\ch_0(F)=0$.
This leads to the usual notion of $\mu_{H_{\epsilon}}$-stability on the category $\Coh(X)$ of coherent sheaves. The key idea of Bridgeland was to replace $\Coh(X)$ by other abelian subcategories of $\cD(X)$, equipped with a suitable slope.

\vskip 4pt

To that end, for any $b\in\bR$ we consider the full subcategories of $\Coh(X)$
\begin{align*}
    \cT_{b}:=\Bigl\{F\in\Coh(X):\mu_{H_{\epsilon}}(Q)>b\text{ for all quotients $F\twoheadrightarrow Q$}\Bigr\}  \ \mbox{ and }\\
    \cF_{b}:=\Bigl\{F\in\Coh(X): \mu_{H_{\epsilon}}(G)\leq b\text{ for all subsheaves $G\hookrightarrow F$}\Bigr\},
\end{align*}
which form a torsion pair in $\Coh(X)$, see 
\cite[Definition 3.2]{bridgeland:K3-surfaces} or \cite[Proposition~2.2]{bayer:brill-noether} for further details. Their tilt gives a \emph{bounded t-structure} on $\cD(X)$ with heart
\[
\Coh^{b}(X):=\Bigl\{F\in\cD(X):\cH^{-1}(F)\in\cF_{b}, \; \cH^{0}(F)\in\cT_{b},\; \cH^i(F)=0 \text{ for }i\neq0,-1\Bigr\}.
\]
Alternatively, objects of $\Coh^{b}(X)$ are those isomorphic (in $\cD(X)$) to $2$-term complexes $G^{-1}\stackrel{d}\longrightarrow G^0$ with $\mbox{ker}(d)\in \cF_b$ and $\mbox{coker}(d)\in \cT_b$. In particular $\Coh^b(X)$ is an abelian subcategory of $\cD(X)$: short exact sequences $0\longrightarrow F'\longrightarrow F\longrightarrow Q\longrightarrow 0$ correspond to those exact triangles $F'\longrightarrow F\longrightarrow Q\longrightarrow F'[1]$ in $\cD(X)$ so that $F, F', Q\in \Coh^b(X)$. 

\vskip 3pt

For $(b,w)\in\bR^2$ we let  $Z_{b,w}\colon K_0(\cD(X))\to\bC$ be the group homomorphism defined as
\begin{equation*}
    Z_{b,w}(F) := -\ch_2(F) +w \ch_0(F)H_{\epsilon}^{2}  +i \bigl(\ch_1(F)\cdot H_{\epsilon} - b \ch_0(F) H_{\epsilon}^2 \bigr). 
\end{equation*}
It is clear that $Z_{b,w}$ factors through $\ch \colon K_0(\cD(X)) \to \Lambda$. Sometimes we will also denote by $Z_{b,w}$ the induced map $\Lambda\to\bC$.

\vskip 4pt

Let us first state Bridgeland's result describing stability conditions on $\cD(X)$; then we expand upon the statements.
By Lemma \ref{lem-no spherical}, there exists $\delta_{\epsilon}>0$ (depending on $\epsilon$) such that no projection of a spherical class is at distance less than $\delta_\epsilon$ of the origin. Accordingly, if we define 
\begin{align}\label{U}
    U_{\epsilon} := \Bigl\{(b,w) \in \mathbb{R}^2 \colon 2w > b^2\Bigr\} \cup \Bigl\{(b,w) \in \mathbb{R}^2 \colon \text{$b \neq 0$ and $b^2+w^2 < \delta_{\epsilon}^2$}\Bigr\},
\end{align}
then Bridgeland's results \cite[Lemma 6.2]{bridgeland:K3-surfaces} imply the following:

\begin{Thm}\label{thm-space}
For any $(b,w) \in U_{\epsilon}$, the pair $\sigma_{b,w}:=\bigl(\Coh^b(X),Z_{b,w}\bigr)$ is a stability condition on $\cD(X)$. Moreover, this assignment defines a continuous map from $U_\epsilon$ to the space of stability conditions on $\cD(X)$.
\end{Thm}

We first explain the notion of $\sigma_{b,w}$-stability and the associated Harder--Narasimhan filtration. Observe that $\mathrm{Im}(Z_{b,w}(F))\geq0$ for any nonzero $F\in\Coh^b(X)$, and also $\mathrm{Re}(Z_{b,w}(F))<0$ whenever $\mathrm{Im}(Z_{b,w}(F))=0$. Intuitively, $\mathrm{Im}(Z_{b,w})$ can be regarded as a notion of ``rank" on $\Coh^b(X)$; then the ``degree" $-\mathrm{Re}(Z_{b,w})$ is positive for rank zero objects in $\Coh^b(X)$. By considering the \emph{slope function}
\begin{equation}\label{eq:slope_bw}
 \nu_{b,w} \colon \Coh^b(X) \to \mathbb{R} \cup \{+\infty\}, \quad \nu_{b,w}(F) := \begin{cases}
-\frac{\mathrm{Re}(Z_{b,w}(F))}{\mathrm{Im}(Z_{b,w}(F))} & \text{if $\mathrm{Im}(Z_{b,w}(F)) > 0$} \\
+\infty & \text{if $\mathrm{Im}(Z_{b,w}(F)) = 0$}
\end{cases}
\end{equation}
we obtain a notion of stability in $\Coh^b(X)$: an object $F\in \Coh^b(X)$ is $\sigma_{b,w}$-\emph{(semi)stable} if and only if for any proper subobject $F' \subset F$ in $\Coh^b(X)$ we have 
$$\nu_{b,w}(F') < (\leq)\ \nu_{b,w}(F/F').$$ 
Every object $F \in \Coh^b(X)$ admits a unique \textit{Harder--Narasimhan} (\textit{HN} for short) \textit{filtration}, namely a finite sequence of objects in $\Coh^b(X)$
\[0=F_0\subset F_1\subset F_2\subset\dots\subset F_m=F\]
whose factors $F_i/F_{i-1}$ are $\sigma_{b,w}$-semistable of decreasing $\nu_{b,w}$-slope.

Furthermore, every $\sigma_{b,w}$-semistable object $F\in\Coh^b(X)$ has a (not necessarily unique) finite \emph{Jordan-Hölder filtration} in $\Coh^b(X)$
\[0=G_0\subset G_1\subset G_2\subset\dots\subset G_n=F\]
whose factors $G_i/G_{i-1}$ are $\sigma_{b,w}$-stable of the same $\nu_{b,w}$-slope as $F$. These factors are unique up to relabelling, and are called the \emph{stable factors} of $F$.

\begin{Rem}
We will use the same plane for the image of the projection $\Pi_{\epsilon}$ defined in \eqref{pro}, and the $(b,w)$-plane containing $U_\epsilon$. In this way, if $F\in\Coh^b(X)$ and $\ch_0(F)\neq0$, then $\nu_{b,w}(F)$ is the slope of the line joining the points $(b,w)$ and $\Pi_{\epsilon}(F)$. If $F$ is moreover $\sigma_{b,w}$-semistable ($2w>b^2$), then $\Pi_{\epsilon}(F)$ lies outside the region $\left\{(b,w) \in \mathbb{R}^2 \colon 2w > b^2\right\}$. 
\end{Rem}

The fact that stability conditions can be deformed continuously is ensured by a technical requirement called the \emph{support property}. While we refer to \cite{bayer:deformation} for details, here we only consider its main application, namely that stability of objects is governed by a locally finite wall and chamber structure.

\begin{Def}\label{def:num_wall}
A \emph{numerical wall} for an object $F\in\cD(X)$ is a line segment $\ell\subset U_\epsilon$ determined by an equation of the form $\nu_{b,w}(F)=\nu_{b,w}(F')$, where $F'\in\cD(X)$ is such that $\left(\ch_0(F),H_\epsilon\cdot \ch_1(F),\ch_2(F)\right)$ and $\left(\ch_0(F'),H_\epsilon\cdot \ch_1(F'),\ch_2(F')\right)$ are non-proportional.
If $F$ is $\sigma_{b,w}$-semistable for $(b,w)\in \ell$ and 
unstable just above or below $\ell$, then $\ell$ is an \emph{actual wall} for $F$ along which $F$ gets destabilized\footnote{Here, we define numerical (or actual) walls only on our two-dimensional slice \( U_{\epsilon} \), but the intersections of the walls in the full stability manifold with our chosen slice may coincide with the entire slice \( U_{\epsilon} \).}.
\end{Def}

\begin{Prop}[Wall and chamber structure] \label{structurewalls}
Let $v=(v_0,v_1,v_2)\in\Lambda$ be any Mukai vector. Then there exists a locally finite set $\{\cW_i^v\}_{i\in I_v}$ of actual walls for objects of Mukai vector $v$, inducing a chamber decomposition of $U_\epsilon$ such that:
\begin{enumerate}
    \item The extension of every actual wall passes through $\Pi_{\epsilon}(v)$ if $v_0\neq0$, or has fixed slope $\frac{v_2}{H_\epsilon\cdot  v_1}$ if $v_0=0$.

    \item For any object $F\in\cD(X)$ with $v(F)=v$, the $\sigma_{b,w}$-(semi)stability of $F$ remains unchanged along a chamber.
\end{enumerate}

	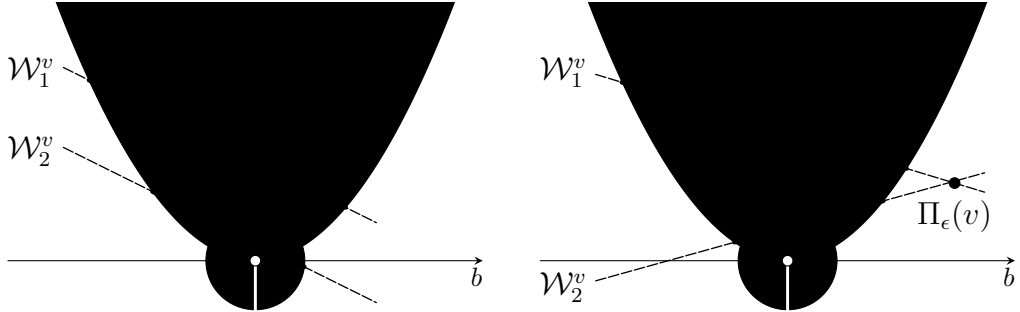
\begin{figure} [h]
	\begin{centering}
		
		\begin{tikzpicture}[line cap=round,line join=round,>=triangle 45,x=1.0cm,y=1.0cm]

    \fill[fill=gray!30!white] (-8,0) circle (.4cm);
    \fill[fill=gray!30!white] (-.5,0) circle (.4cm);
		
        \draw[->,color=black] (-10.5,0) -- (-5.5,0);
		\draw[->,color=black] (-3,0) -- (2,0);
		
		\fill [fill=gray!30!white] (-0.5,0) parabola (1.47, 3.03) parabola [bend at end] (-2.47,3.03) parabola [bend at end] (-0.5,0);

		\fill [fill=gray!30!white] (-8,0) parabola (-6.03, 3.03) parabola [bend at end] (-9.97,3.03) parabola [bend at end] (-8,0);

		\draw[->,color=black] (-8,-1) -- (-8,3.5);
		\draw[->,color=black] (-0.5,-1) -- (-0.5,3.5);

		\draw [] (-0.5,0) parabola (1.5,3.12); 
		\draw [] (-0.5,0) parabola (-2.5,3.12); 
		\draw [] (-8,0) parabola (-10,3.12); 
		\draw [] (-8,0) parabola (-6,3.12);
       \draw (-8,0) circle (.4cm);
		\draw (-.5,0) circle (.4cm);

		\draw[color=black, dashed] (-10.5,2.8) -- (-6,1);
		\draw[color=black, dashed] (-10.5,1.8) -- (-6.5,0.2);

       \draw[color=black,semithick] (-9.8,2.52) -- (-6.7,1.28);
       \draw[color=black,semithick] (-9.3,1.32) -- (-7.2,.48);
       
		\draw (-10.5,1.8) node [left] {$\mathcal{W}^v_2$};
		\draw (-10.5,2.8) node [left] {$\mathcal{W}^v_1$};

        \draw [color=white, line width=1.1pt] (-0.5,0)-- (-0.5,-1);

        \draw [color=white, line width=1.1pt] (-8,0)-- (-8,-1);
		
		\draw (.8,3.5) node [right] {\large{$v_0 \neq 0$}};
		\draw (-6.8,3.5) node [right] {\large{$v_0 = 0$}};
		\draw (-5.5,0) node [below] {$b, \frac{H_{\epsilon}\cdot \ch_1}{H_{\epsilon}^2\cdot \ch_0}$};
		\draw (-8,3.5) node [above] {$w, \frac{\ch_2}{H^2_{\epsilon}\cdot \ch_0}$};
		
		\draw (2,0) node [below] {$b, \frac{H_{\epsilon}\cdot \ch_1}{H_{\epsilon}^2\cdot \ch_0}$};
		\draw (-0.5,3.5) node [above] {$w, \frac{\ch_2}{H_{\epsilon}^2\cdot \ch_0}$};
	
		\draw (-7.3,2.5) node [right] {\Large{$U_{\epsilon}$}};
		\draw (0.2,2.5) node [right] {\Large{$U_{\epsilon}$}};

		\draw (1.5, 1.2) node [right] {$\Pi_{\epsilon}(v)$};
		
		\draw[color=black, dashed] (1.5, 1.2) -- (-2.9,2.8);
		\draw[color=black, dashed] (1.5, 1.2) -- (-2.2, .7);
		
		\draw (-2.9,2.8) node [left] {$\mathcal{W}^v_1$};
		\draw (-2.2, .7) node [left] {$\mathcal{W}^v_2$};

		\draw[color=black, semithick] (-2.3 ,2.58) -- (.88,1.423);
		\draw[color=black, semithick] (-1.5,.795) -- (0.7,1.092);

		\begin{scriptsize}

		\fill [color=black] (1.5,1.2) circle (2pt);
        \fill [color=white] (-.5,0) circle (2pt);
        \fill [color=white] (-8,0) circle (2pt);

		\end{scriptsize}
		
		\end{tikzpicture}
		
		\caption{Actual walls $\mathcal{W}^v_i$ for $v$ when $v_0=0$ and $v_0\neq0$ }
		
		\label{wall.figure}
		
	\end{centering}
	
\end{figure}

\end{Prop}

It follows that if \(\sigma\) varies within a chamber, the set of \(\sigma\)-semistable objects with Mukai vector \(v\) remains unchanged. Indeed, for any stability condition \(\sigma:=\sigma_{b,w}\), the moduli stack \(\mathfrak{M}_{\sigma}(v)\) of flat families of \(\sigma\)-semistable objects in $\Coh^b(X)$ of Mukai vector \(v\) is an Artin stack of finite type over \(\mathbb{C}\) \cite[Theorem 1.4 and Section 3]{toda-k3-surface}. In cases where a coarse moduli space parameterizing the semistable objects exists, we denote it by \(\mathcal{M}_{\sigma}(v)\). As discussed in detail in \cite{bayer:projectivity}, such a proper coarse moduli space exists, for instance, when \(v\) is primitive and \(\sigma\) is generic in the stability manifold with respect to \(v\).

Once a Mukai vector $v\in\Lambda$ is fixed, actual walls for $v$ in $U_\epsilon$ are bounded from above (see e.g.~\cite[Section 6.4]{macri-schmidt-lectures}); in particular, there is a ``largest actual wall". Above it there is the so-called \emph{Gieseker chamber}, where stability agrees with Gieseker stability for sheaves. This gives a fundamental example of Bridgeland stable objects. More precisely:

\begin{Thm}\label{Giesekerchamber}
Let $v=(v_0,v_1,v_2)\in\Lambda$. There exists $w_0\in \mathbb R$ such that:
\begin{enumerate}
    \item\label{GiesekerchamberA} If $v_0>0$, or $v_0=0$ and $v_1$ is effective, then for all $w\geq w_0$ and $b<\frac{H_\epsilon\cdot v_1}{H_\epsilon^2\cdot v_0}$ an object $F\in\Coh^b(X)$ with $v(F)=v$ is $\sigma_{b,w}$-(semi)stable if and only if $F$ is an $H_{\epsilon}$-Gieseker (semi)stable sheaf.

    \item\label{GiesekerchamberB} If $v_0<0$, then for all $w\geq w_0$ and $b>\frac{H_\epsilon\cdot v_1}{H_\epsilon^2\cdot v_0}$ an object $F\in\Coh^b(X)$ with $v(F)=v$ is $\sigma_{b,w}$-(semi)stable if and only if $F=R\mathcal{H}om(G,\cO_X)[1]$ for an $H_{\epsilon}$-Gieseker (semi)stable sheaf $G$.
\end{enumerate}
\end{Thm}



Let us recall that for any Mukai vector $v \in \Lambda$, there exists a projective variety $\cM_{H_\epsilon}(v)$ which is a coarse moduli space parameterizing S-equivalence classes of $H_{\epsilon}$-Gieseker semistable sheaves, see \cite[Section 4]{huybrechts:geometry-of-moduli-space-of-sheaves} for details.

\vskip 2pt

\subsection{Stability of special objects}
This subsection gathers useful lemmas concerning the stability of some distinguished objects along $U_\epsilon$.
Noting that the quadratic form $\langle\cdot,\cdot\rangle$ (defined in \eqref{def:quadraticform}) has signature $(2, 2)$ on $\Lambda_\bR$, we begin with a remark on the discriminant of semistable objects:

\begin{Lem}\label{lem-delta} Let $(b,w)\in U_\epsilon$. Then:
\begin{enumerate}
    \item\label{lem-delta:a} If $F$ is a $\sigma_{b,w}$-stable object, then either (i) $\Delta(F) \geq 0$, or (ii) $k|g$, $\ch_0(F) = 0$ and $\ch_1(F)=H- \frac{g}{k}E$. 
    \item\label{lem-delta:b} Assume $2w > b^2$ and let $F$ be strictly $\sigma_{b,w}$-semistable, with stable factors $\{F_i\}_{i \in I}$. If every factor satisfies $\Delta(F_i) \geq 0$, then $\Delta(F_i) \leq \Delta(F)$ for all $i$ with equality if and only if $\Delta(F_i) =\Delta(F) = 0$ and $\ch(F_i)$ is a multiple of $\ch(F)$. 
\end{enumerate}
\end{Lem}
\begin{proof}
Any $\sigma_{b,w}$-stable object $F$ is simple, hence $\hom(F, F) = \ext^2(F,F) =1$ and
\begin{equation*}
    \Delta(F) -2\ch_0(F)^2= -\chi(F, F) \geq -2.
\end{equation*}
If $\ch_0(F) \neq 0$, we get $\Delta(F) \geq 0$. If $\ch_0(F) = 0$ and $\Delta(F) <0$ we have $\Delta(F) = -2$, hence $\ch_1(F)^2=-2$ which implies $k|g$ and $\ch_1(F)= \pm(H-\frac{g}{k} E)$. It must be $\ch_1(F)= H-\frac{g}{k} E$ as claimed in part \eqref{lem-delta:a}, as $F\in \Coh^b(X)$ implies $0\leq \mathrm{Im} Z_{b,w}(F)=H_\epsilon\cdot \ch_1(F)$.

To prove part \eqref{lem-delta:b}, we first show that the kernel of $Z_{b,w}$ is negative definite with respect to $\Delta$. A class $v \in \ker Z_{b,w}$ can be written as 
    $v = (r, \ xH+yE, \ rwH_{\epsilon}^2)$ such that 
    \begin{equation}\label{negdef}
        yE\cdot H = \frac{br}{\epsilon}H_{\epsilon}^2 -xH^2 - \frac{x}{\epsilon}E\cdot H.
    \end{equation}
Since $2w > b^2$, we get 
    \begin{align*}
        \Delta(v) & \ \overset{\eqref{negdef}}{=} H_{\epsilon}^2 \left(\frac{2xbr}{\epsilon} -2r^2w \right) - x^2H^2 - \frac{2x^2}{\epsilon}E\cdot H \\
        & \ < H_{\epsilon}^2 \left(\frac{2xr}{\epsilon}b -r^2b^2 \right) - x^2H^2 - \frac{2x^2}{\epsilon}E\cdot H \leq H_{\epsilon}^2\frac{x^2}{\epsilon^2} - x^2H^2 - \frac{2x^2}{\epsilon}E\cdot H =0
    \end{align*}
    as claimed. Now, since $\nu_{b,w}(F_i)=\nu_{b,w}(F)$ for all $i\in I$, the points $Z_{b,w}(F_i)$ lie in a ray $\rho^+$ in the upper half plane starting from the origin. Then \cite[Lemma 11.7]{bayer:the-space-of-stability-conditions-on-abelian-threefolds} implies that $(Z_{b,,w})^{-1}(\rho^+) \cap \{\Delta \geq 0\}$ is a convex cone, and one can conclude part \eqref{lem-delta:b} via the same argument as in \cite[Lemma 3.9]{bayer:the-space-of-stability-conditions-on-abelian-threefolds}.  
\end{proof}

A consequence of Lemma \ref{lem-delta} is that, for appropriate Mukai vectors of rank 0, the Gieseker chamber contains all stability conditions above the parabola $\{2w=b^2\}\subseteq \mathbb R^2$:

\begin{Cor}\label{cor:nowall}
    For $q>0$ and $s\in\bZ$, the Mukai vector $v=(0,qE,s)\in\Lambda$ has no actual wall intersecting the region $\{2w>b^2\}$ providing that $k\nmid g$, or $k|g$ and $\epsilon<\frac{k}{g+1}$.
\end{Cor}
\begin{proof}
    Assume such an actual wall exists, and let $F$ be an object of Mukai vector $v$ that gets destabilized along this wall. Write $\{F_i\}_{i\in I}$ for the stable factors of $F$. 
    
    Since $\Delta(F)=0$, it suffices to check that $\Delta(F_i)\geq0$ for all $i\in I$; indeed, in that case it follows from Lemma \ref{lem-delta}.\eqref{lem-delta:b} that $\ch(F_i)$ and $\ch(F)$ are proportional for every $i\in I$. This implies $\nu_{b,w}(F_i)=\nu_{b,w}(F)$ for all $(b,w)\in U_\epsilon$, which is a contradiction.
    
    If $\Delta(F_i)<0$ for some $i$, then by Lemma \ref{lem-delta}.\eqref{lem-delta:a} we have $k|g$ and the sum of the Chern characters of stable factors with negative discriminant is of the form \(\bigl(0, a'(H-\frac{g}{k} E), s'\bigr)\), for some \(a' \in \mathbb{Z}_{>0}\) such that
    \begin{equation}\label{ineq:sphericalclasses}
    a'(H-\tfrac{g}{k} E)\cdot H_\epsilon < qE\cdot H_\epsilon\;\;\;\text{(i.e., $a'(k+\epsilon(g-2))<\epsilon qk$).}
    \end{equation}
    Hence the difference $\ch(F)-\bigl(0, a'(H-\frac{g}{k} E), s'\bigr)=\bigl(0,-a'H+(q-a'\frac{g}{k})E,s-s'\bigr)$ has non-negative discriminant, which is a contradiction since
    \[
    \left(-a'H+(q-a'\tfrac{g}{k})E\right)^2=2a'\left(a'(2g-1)-qk\right)\overset{\eqref{ineq:sphericalclasses}}{<}2a'\left(a'(g+1)-\tfrac{a'k}{\epsilon}\right)<0,
    \]
    (the last inequality being derived from the assumption $\epsilon<\frac{k}{g+1}$). 
\end{proof}

For $2w>b^2$, the following lemma characterizes $\sigma_{b,w}$-stable objects with $\ch_2=0$ and $\ch_0\neq0$ as (shifts of) multiples of the elliptic pencil on $X$:

\begin{Lem}\label{lem-eE}
Let \(F \in \Coh^b(X)\) satisfy \(v(F)=(r, qE, r)\), where \(r \neq 0\) and \(q\geq0\). If $F$ is \(\sigma_{b, w}\)-stable ($2w>b^2$), then $|r| =1$ and (up to a shift if $r=-1$) \(F \cong \mathcal{O}_X(\frac{q}{r}E)\). Conversely, \(\mathcal{O}_X(qE)\) (resp.~\(\cO_X(-qE)[1]\)) is $\sigma_{b,w}$-stable for all $b<\frac{H_\epsilon\cdot qE}{H_\epsilon^2}$ (resp.~for all $b>-\frac{H_\epsilon\cdot qE}{H_\epsilon^2}$) and $2w>b^2$.
\end{Lem}
\begin{proof}
Note that $\sigma_{b,w}$-stability of \(F\) implies \(\hom(F, F) = \ext^2(F, F) = 1\), hence
\begin{equation*}
    2r^2 = \chi(F, F) = 2 - \ext^1(F, F) \leq 2
\end{equation*}
which gives \(|r|=1\). In view of Theorem \ref{Giesekerchamber}, it suffices to check that there is no actual wall for the Mukai vector $(r, qE, r)$ along the region $\{2w>b^2, H_\epsilon\cdot qE-brH_\epsilon^2>0\}$.

Any object $F$ with Mukai vector $(r,qE,r)$ has \(\Delta(F) = 0\). Thus if $F$ gets destabilized along such a wall, there are stable factors of negative discriminant; otherwise, by Lemma \ref{lem-delta} all stable factors have Chern character multiple of \(\ch(F)\), hence cannot form a wall. 

We know that summation of Chern characters of all factors of negative discriminant is of the form \((0, a'(H-\frac{g}{k} E), s)\) for some \(a' \in \mathbb{Z}_{>0}\). Moreover, if $\sigma_{b_0,w_0}$ lies in the actual wall, then this summation computes the same $\nu_{b_0, w_0}$-slope as $F$:
\begin{equation*}
    \frac{s}{a'(H-\frac{g}{k} E)\cdot H_{\epsilon}} = \frac{-rw_0H_{\epsilon}^2}{qE\cdot H_{\epsilon}-b_0rH_\epsilon^2}.
\end{equation*}
Thus \(rs \leq 0\), which implies
\[
    \Delta(\ch(F) - \bigl(0, a'(H-\tfrac{g}{k} E), s)\bigr) = -2a'^2 - 2a'q E\cdot H + 2rs < 0,
\]
contradicting Lemma \ref{lem-delta}. Therefore there is no such wall, which finishes the proof.
\end{proof}

Along the horizontal line $w=0$, we have a different behavior (recall the role of $\delta_\epsilon$ in the definition \eqref{U} of $U_\epsilon$):

\begin{Lem}\label{lem-ox}
Assume $\epsilon<1$, and consider $\sigma := \sigma_{b,0}$ for any $b \in (-\delta_\epsilon, 0)$. Then:
\begin{enumerate}
    \item\label{lem-ox: a} $\cO_X$ and $\cO_X(-E)[1]$ are $\sigma$-stable objects.
   \item\label{lem-ox: b}  $\cO_X(qE)$ and $\cO_X(-(q+1)E)[1]$ are strictly $\sigma$-semistable for all $q \geq 1$. 
\end{enumerate}
\end{Lem}
\begin{proof}
Since $\cO_X$ and $\cO_X(-E)[1]$ are $\sigma$-semistable by Lemma \ref{lem-eE}, to prove \eqref{lem-ox: a} we only need to show that they are not strictly $\sigma$-semistable.
In the case of $\cO_X$, we can get arbitrary close to $\Pi_{\epsilon}(\cO_X) = (0, 0)$ as $b\to 0^-$; hence any exact triangle $F_2\rightarrow \cO_X\rightarrow F_1$ of $\sigma$-semistable objects with the same $\sigma$-slope satisfies $\Pi_{\epsilon}(F_1)=\Pi_{\epsilon}(\cO_X)=\Pi_{\epsilon}(F_2)$. But this implies that $\cO_X$ is strictly $\sigma_{b,w}$-semistable for $w>\frac{b^2}{2}$, contradicting Lemma \ref{lem-eE}.


\vskip 3pt

Assume $F_2 \rightarrow \cO_X(-E)[1] \rightarrow F_1$ is an exact triangle with $ \ch(F_1) = (r_0,\ tH+aE,\ 0)$, such that $F_1$ is $\sigma$-stable of the same $\nu_{b,0}$-slope as $\cO_X(-E)[1]$. Then 
\begin{equation}\label{v}
    -2 \leq v(F_1)^2 = t^2H^2 +2atE\cdot H-2r_0^2
\end{equation}
and in particular 
\begin{equation}\label{simplev}
-2 - 2taE\cdot H \leq t^2H^2.
\end{equation}
Also, the inequalities $0\leq \mathrm{Im} Z_{0,b}(F_1)\leq\mathrm{Im} Z_{0,b}(\cO_X(-E)[1])$  for all $b\in(-\delta_\epsilon,0)$ give
\begin{equation}\label{imparts}
    0 \leq tE\cdot H+t\epsilon H^2 +a\epsilon E\cdot H \leq \epsilon E\cdot H.
\end{equation}
Hence, if $t<0$ we obtain
\begin{equation*}
    \frac{\epsilon}{t}+ \frac{t}{2}\epsilon H^2 \overset{\eqref{simplev}}{\leq}  -a\epsilon E\cdot H \overset{\eqref{imparts}}
    {\leq} tE\cdot H+t\epsilon H^2 
\end{equation*}
which is not possible. If $t \geq 1$, we get 
\begin{equation*}
    E\cdot H \leq tE\cdot H \overset{\eqref{imparts}}
    {\leq} -t\epsilon H^2 + \epsilon E\cdot H - a \epsilon E\cdot H \overset{\eqref{simplev}}{\leq} \epsilon\left(-\frac{t}{2} H^2 + E\cdot H + \frac{1}{t}\right),
\end{equation*}
hence $\frac{H^2}{2}\leq1$ (by the assumption $\epsilon<1$) which implies $g=\frac{H^2}{2}+1\leq 2$, contradiction.

Thus $t=0$, and so \eqref{v} gives that either $r_0=0$ or $r_0= \pm 1$. The quotient $F_1$ must have a bigger slope than $\cO_X(-E)[1]$ above the horizontal line $w=0$, as $\cO_X(-E)[1]$ is $\sigma_{b,w}$-stable for $w> \frac{b^2}{2}$ by Lemma \ref{lem-eE}; therefore, $r_0=-1$. But this implies that $F$ has the same Chern character as $\cO_X(-E)[1]$, which is not possible. This proves that $\cO_X(-E)[1]$ is $\sigma$-stable and completes part \eqref{lem-ox: a}.  

For part \eqref{lem-ox: b}, note that $\cO_X(qE)$ and $\cO_X(-(q+1)E)[1]$ are $\sigma$-semistable if $q\geq1$, again by Lemma \ref{lem-eE}. They are strictly $\sigma$-semistable, as we have nonzero maps $\cO_X \to \cO_X(qE)$ and $\cO_X(-(q+1)E)[1]\to\cO_X(-E)[1]$ between objects of the same $\nu_{b,0}$-slope.
\end{proof}

We end with an observation (see  \cite[Lemma~6.5]{bayer:brill-noether} for a proof), used in Section \ref{stratification}: 

\begin{Lem}\label{lem-stability above the wall}
    Let $F_1,F_2\in\Coh^b(X)$ be $\sigma_{b,w_0}$-stable objects with $\nu_{b,w_0}(F_1)=\nu_{b,w_0}(F_2)$ and $\nu_{b,w}(F_1) < \nu_{b,w}(F_2)$ for $w>w_0$. Then for any extension 
    \begin{equation*}
        V^\vee\otimes F_1=F_1^{\oplus n} \longrightarrow F \longrightarrow F_2
    \end{equation*}
    induced by an $n$-dimensional subspace $V\subseteq \Ext^1(F_2,F_1)$, the object $F$ is $\sigma_{b,w}$-stable for all sufficiently small $w>w_0$.
\end{Lem}

\section{Wall-crossing for special Mukai vectors}\label{wallcrossing-special}
In this section we show how, by choosing a suitable polarization $H_{\epsilon}$, wall-crossing for Mukai vectors of the form $(r_0,H-a_0E,s_0)$ ($a_0\in\bZ_{\geq0}$) becomes significantly simpler. We show in Proposition \ref{prop-main-wallcrossing} that actual walls for $v$ intersecting the vertical axis $b=0$ admit an explicit description and can be classified into two types, according to the shape of destabilizing short exact sequences. The following key lemma will be applied repeatedly: 

\begin{Lem}\label{lem-key}
    Fix $m\geq0$. Then there is $\epsilon_m >0$ so that, if $\epsilon<\epsilon_m$, then any Chern character $(r, tH+qE, s)\in\Lambda$ with
    \begin{equation*}
        -rs \leq m \ , \qquad 0 \leq (tH+qE)\cdot H_{\epsilon} \leq H\cdot H_{\epsilon} \qquad \text{and} \qquad \Delta(r, tH+qE, s) \geq -2 
    \end{equation*}
    satisfies $t=0$ or $t=1$.  
\end{Lem}
\begin{proof}
    We know
\begin{equation}\label{eq-1}
    -2 \leq \Delta(r, tH+qE, s) = (tH+qE)^2 - 2rs \leq t^2H^2 + 2tqE\cdot H + 2m,
\end{equation}
and our second condition gives
\begin{equation}\label{eq-2}
    0 \leq tE\cdot H + t \epsilon H^2 + q \epsilon E\cdot H \leq E\cdot H + \epsilon H^2.
\end{equation}

If \( t \geq 2 \), then for \( \epsilon < \epsilon_m := \frac{E \cdot H}{H^2 + m + 1} \),
\begin{align*}
    (t-1)E\cdot H &\overset{\eqref{eq-2}}{<} -q\epsilon E\cdot H 
    \overset{\eqref{eq-1}}{\leq} \epsilon \Bigl(\frac{t}{2}H^2 + \frac{m+1}{t} \Bigr) < (t-1)E\cdot H,
\end{align*}
which is a contradiction. If \( t \leq -1 \) the first inequality in \eqref{eq-2} gives \( q \geq 0 \), and so 
\begin{equation*}
    -tE\cdot H \overset{\eqref{eq-2}}{<} q \epsilon E\cdot H 
    \overset{\eqref{eq-1}}{\leq} \epsilon \Bigl(-\frac{t}{2}H^2 - \frac{m+1}{t}\Bigr) < -tE\cdot H
\end{equation*}
for  \( \epsilon < \epsilon_m \), which is again a contradiction.
\end{proof}

Via Lemma \ref{lem-key} we can control all actual walls for the Mukai vectors of interest in later sections:

\begin{Prop}\label{prop-main-wallcrossing}
Fix a Mukai vector \(v = (r_0, H - a_0E, s_0+r_0)\) for some \(a_0 \in \mathbb{Z}_{\geq 0}\). 
There exists \(\epsilon_{v} > 0\)  such that, if \(\epsilon < \epsilon_{v}\) and \(F \in \Coh^0(X)\) is  \(\sigma_{0, w_0}\)-strictly semistable of Mukai vector \(v\) for some \(w_0 \geq 0\), then \(F\) sits in an exact triangle
\[
Q_1 \longrightarrow F \longrightarrow Q_2
\]
where \(\nu_{0, w_0}(F) = \nu_{0, w_0}(Q_i)\), and either $\ch_1(Q_1)$ or $\ch_1(Q_2)$ is a multiple of $E$. 

\vskip 1mm

\noindent If moreover $(s_0, w_0) \neq (0, 0)$, then up to relabeling the factors:
\begin{enumerate}
    \item \(Q_1\) is \(\sigma_{0, w_0}\)-stable with \(\ch_1(Q_1) = H - aE\) for some \(a \geq a_0\), and 
    \item\label{prop-main-wallcrossing:B} \(Q_2\) is \(\sigma_{0, w_0}\)-semistable and either:
    \begin{enumerate}
        \item[$(b_1)$] it is isomorphic, up to a shift, to \(\cO_X(mE)^{\oplus \frac{a-a_0}{m}}\) for some \(m \in \mathbb{Z}\), or 
        \item[$(b_2)$] every \(\sigma_{0, w_0}\)-stable factor of \(Q_2\) has $\ch_0=0$ (in particular, $\ch_0(Q_2)=0$) and $\ch_1$ is a multiple of $E$.
    \end{enumerate} 
\end{enumerate}
\end{Prop}

\vskip 2mm

\begin{Rem}\label{rem-00}
Let \(\ell\) be the line through the point $(0,w_0)$
that passes through \(\Pi_{\epsilon}(F)\) (if \(r_0 \neq 0\)) or has slope \(\frac{s_0}{(H - a_0E) \cdot H_{\epsilon}}\) (if $r_0=0$). We are interested in the upper part 
\[
\ell^+ := \ell \cap \bigl\{(b, w) \in U_{\epsilon} : w\geq 0\bigr\}
\]
of the line $\ell$ in \(U_{\epsilon}\).
In particular, if \(w_0 = 0\), in Proposition~\ref{prop-main-wallcrossing} \(\sigma_{0, w_0}\)-(semi)stability and \(\nu_{0, w_0}\)-slope refer to \(\sigma_{b, w}\)-(semi)stability and \(\nu_{b, w}\)-slope for \((b, w)\) lying on \(\ell^{+}\).
\end{Rem}
\vskip 3mm

\begin{proof}[Proof of Proposition \ref{prop-main-wallcrossing}]
Consider the line segment $\ell^+$ as described in Remark \ref{rem-00}. 
If $F$ is \(\sigma_{b',w'}\)-unstable for \((b', w')\) lying just above or just below the line \(\ell^+\), we consider the Harder-Narasimhan filtration with respect to $\sigma_{b',w'}$ 
\[
Q_0 \subset Q_1 \subset \cdots \subset Q_n = F.
\] 
If \( Q_i/Q_{i-1} \) is strictly \(\sigma_{b', w'}\)-semistable for some \( i \), then we further refine the filtration to account for the \(\sigma_{b', w'}\)-stable factors. In case \( F \) is \(\sigma_{b',w'}\)-semistable for \((b', w')\) values just above and below the line \(\ell^+\), we consider a JH filtration of \( F \) with respect to one side.

We denote by \( I := \{1,..., n'\} \) the index set of this refined filtration, for which all quotients \( F_i := Q_i/Q_{i-1} \) are \(\sigma_{b', w'}\)-stable with 
\begin{equation}\label{order}
\nu_{b',w'}(F_1) \geq \nu_{b',w'}(F_2) \geq \dots \geq \nu_{b',w'}(F_{n'}).     
\end{equation}
Write $\alpha_i :=\ch(F_i)=(r_i,t_iH+q_iE,s_i)$, with $i=1, \ldots, n$. We classify these stable factors into several different types.

\vskip 1mm

Consider $I':=\bigl\{i\in I:\;r_i>0,\ s_i>0\bigr\}$. Take $\epsilon<\epsilon_{m=0}$ (in the notations of Lemma \ref{lem-key}). Since $r_is_i >0$ for all $i \in I'$, Lemma \ref{lem-key} implies $t_i \in\{0, 1\}$ as $ \Delta(F_i) \geq v(F_i)^2 \geq -2$. But if $t_i=0$, then $v(F_i)^2 = -2r_i(r_i+s_i) <-2$ which is not possible, therefore $t_i=1$ for all $i\in I'$. 

\vskip 3pt

We claim that $|I'| \leq 1$, that is, there is at most one stable factor with $r_i>0$ and $s_i>0$. Indeed, consider $\alpha_{I'}:=\sum_{i \in I'} \alpha_i$. For any $i\in I'$, we have $\Delta(F_i)=v(F_i)^2+2r_i^2\geq 0$ as $r_i>0$. Since $\Ker Z_{0, w}$ is negative semi-definite for $w\geq 0$ with respect to the quadratic form $\Delta$, \cite[Lemma 11.6]{bayer:the-space-of-stability-conditions-on-abelian-threefolds} implies that $\Delta(\alpha_{I'}) \geq 0$. Therefore, Lemma \ref{lem-key} applied to the class $\alpha_{I'}$ yields $\sum_{i \in I'} t_i \in\{0, 1\}$, and so the set $I'$ consists of at most one element. Moreover, there are only finitely many possibilities (depending only on the class $v$, and not on $w_0$) for the Chern character $(r_i, H+q_iE, s_i)$ of this factor, as:
\begin{itemize}
    \item The inequality $H_{\epsilon}\cdot (H+q_iE)\leq H_{\epsilon}\cdot (H-a_0E)$ implies $q_i\leq -a_0$. 
    \item Since $r_is_i \neq 0$ and $v(F_i)^2 \geq -2$, we have $(H+q_iE)^2 \geq -2+2r_i(r_i+s_i) \geq 2r_is_i$, which implies 
    \begin{equation}\label{bound}
        1\leq|r_i|,|s_i|\leq \frac{(H-a_0E)^2}{2} \qquad \text{and} \qquad \frac{-H^2}{2E\cdot H} \leq q_i. 
    \end{equation}
\end{itemize}

Similarly, one can consider $I'':=\bigl\{i\in I:\;r_i<0,\ s_i <0\bigr\}$.
The same argument yields: $t_i=1$ for all $i \in I''$, $|I'| \leq 1$, and there are finitely many possibilities for the class of a factor in $I''$ (satisfying inequalities in \eqref{bound}).

\vskip 1mm

\textbf{Step 1}. We first assume that the wall has positive slope, that is, $\nu_{0, w_0}(F) > 0$. Then
$$I = I' \cup I'' \cup J' \cup J''$$
where $J':=\bigl\{i\in I:\;r_i < 0, \ s_i \geq 0\bigr\}$ and $J'' = \bigl\{i \in I: r_i =0, \ s_i >0 \bigr\}$.
If $i\in J'\cup J''$, then we have
\begin{align*}
0\geq r_{i} \geq \sum_{j\notin I'} r_j=r_0-\sum_{j\in I'} r_j \overset{\eqref{bound}}{\geq} r_0-\max\Bigl\{0,\frac{(H-a_0E)^2}{2}\Bigr\},\\
0\leq s_{i} \leq  \sum_{j\notin I''} s_j= s_0-\sum_{j\in I''} s_j \overset{\eqref{bound}}{\leq} s_0+\max\Bigl\{0,\frac{(H-a_0E)^2}{2}\Bigr\},
\end{align*}
which implies that there are finitely many possible values of $r_i$ and $s_i$ for all $i\in I$. We can thus find $M>0$ (depending only on $v$, and not on the slope of the wall) such that $-r_is_i< M$ for every $i\in I$. It follows from Lemma \ref{lem-key} that, if $\epsilon < \epsilon_M$, then we have $t_i\in\{0,1\}$ for all $i\in I$. 

Since $\sum_{i \in I} (t_iH+q_iE)=\ch_1(F)=H-a_0E$, we have $\sum_{i \in I} t_i =1$, hence there is a unique $i_0\in I$ with $t_{i_0}=1$. 
The corresponding stable factor $F_{i_0}$ satisfies the inequalities  $-2 \leq v(F_{i_0})^2 < (H+q_{i_0}E)^2+2M$, therefore 
\begin{equation}\label{q}
   q:= \frac{-2-H^2-2M}{2E\cdot H} < q_{i_0} \leq -a_0. 
\end{equation}

Any other factor satisfies $t_i=0$ (in particular $i\in J'\cup J''$) and $q_i\geq0$. Furthermore,
    \begin{equation}\label{qi}
    0\leq \sum_{i \neq i_0} q_i=-a_0-q_{i_0} \overset{\eqref{q}}{<} -a_0-q
    \end{equation}
which implies $q_i< -a_0-q$ for all $i \neq i_0$. Then either 
    \begin{enumerate}
        \item [(c.1)] $r_i =0$ for all $i \neq i_0$, or 
        \item [(c.2)] $r_{i_1} \neq 0$ for some $ i_1 \neq i_0 \in I$. 
    \end{enumerate}
    
In case (c.1), all the slopes $\nu_{b',w'}(F_i)$ for $i \neq i_0$ are equal (as $\nu_{b,w}(F_i)$ is independent of $b,w$). Thus the order in \eqref{order} implies that $F_{i_0}$ is either a subobject or quotient of $F$, as claimed in  case ($b_2$) of the main statement.    
\vskip 3pt

In case (c.2) we have $F_{i_1} \in J'$ (that is,  $r_{i_1}<0$, $s_{i_1}\geq0$), and the slope $\eta(\epsilon)$ of the wall $\cW(F,F_{i_1})$, depending on $\epsilon$, satisfies
\[
\eta(\epsilon) = \frac{s_{i_1}r_0-s_0r_{i_1}}{r_0q_{i_1}\epsilon E\cdot H-r_{i_1}H_{\epsilon}\cdot (H-a_0E)}\underset{\epsilon\to 0^+}{\longrightarrow}\frac{s_{i_1}r_0-s_0r_{i_1}}{-r_{i_1}E\cdot H}.
\]
Since there are finitely many possibilities for $\ch(F_{i_1})$, we may choose $\epsilon$ small enough (depending only on $v$) so that the inequality $\eta(\epsilon)\cdot\epsilon(-a_0-q)E\cdot H < 1$ holds. On the other hand, we know $\eta(\epsilon) = \nu_{0, w_0}(F_i)$ for all $i \in I$. If $i \neq i_0$ satisfies $i\in J''$, then 
\[
\nu_{0, w_0}(F_i) = \frac{s_i}{q_iE\cdot H_{\epsilon}} = \eta (\epsilon)< \frac{1}{\epsilon (-a_0-q)E\cdot H}
\]
and hence
\[s_i <\frac{q_iE\cdot H_{\epsilon}}{\epsilon (-a_0-q)E.H}=\frac{q_i}{-a_0 -q}\overset{\eqref{qi}}{\leq} 1,
\]
contradiction. Therefore, in case (c.2) we have $i\in J'$ for all $i\neq i_0$. It follows that \(\ch(F_i)\) is proportional to \(\ch(F_{i_1}) = (r_{i_1}, q_{i_1}E, 0)\) for all \(i \neq i_0\); by Lemma \ref{lem-eE}, we have \(r_i = -1\) and \(F_i = \mathcal{O}_X(-q_{i_1}E)[1]\) for all \(i \neq i_0\). Thus, the main statement in case ($b_1$) follows from the vanishing
\begin{equation*}
	\hom\left(\mathcal{O}_X(-q_{i_1}E)[1], \mathcal{O}_X(-q_{i_1}E)[2]\right) = \hom(\mathcal{O}_X, \mathcal{O}_X[1]) = 0. 
\end{equation*}

\vskip 2mm

\textbf{Step 2.} Now we consider walls of non-positive slope. 
\begin{itemize}
    \item[(I)] If $\nu_{0, w_0}(F) <0$, then  $$I = I' \cup I'' \cup \tilde{J}' \cup \tilde{J}''$$  
where $\tilde{J}':=\{i\in I:  \;r_i > 0, \ s_i \leq 0\}$ and $\tilde{J}'' = \{i \in I:  r_i =0, \ s_i <0 \}$.
\item[(II)] If $\nu_{0, w_0}(F) = 0$ and $w_0>0$, then 
$$I = I' \cup I'' \cup R$$  
where $R:=\{i\in I:\;r_i =0, \ s_i = 0\}$.
\item[(III)]  If $\nu_{0, w_0}(F) = 0$ and $w_0=0$, then $s_i=0$ for all $i \in I$.  
\end{itemize}
Then one can easily apply the same argument as in Step 1 to get the final claim.    
\end{proof}

As a result of the proof of Proposition \ref{prop-main-wallcrossing}, we get the well-known fact (see e.g. \cite[Lemma 6.24]{macri-schmidt-lectures}) that there are only finitely many actual walls for the class \(v\) intersecting the vertical line \(b = 0\) at points with \(w > 0\).

\vskip 3pt

The following criterion further restricts the possibilities described in Proposition \ref{prop-main-wallcrossing}:

\begin{Lem}\label{lem-wall-crossing-negative}
    In Proposition \ref{prop-main-wallcrossing}, we may choose \(\epsilon_{v}\) suitably such that if \(0 < \epsilon < \epsilon_{v}\), then the following hold:
    \begin{enumerate}
        \item[(1)]\label{lem-wall-crossing-negative: 1} If \(r_0 \leq 0\) and \(\nu_{0, w_0}(F) < 0\), or \(r_0 \geq 0\) and \(\nu_{0, w_0}(F) > 0\), then case $(b_2)$ cannot occur.
        \item[(2)] If \(r_0 \leq 0\), \(s_0 < 0\), and \(\nu_{0, w_0}(F) \geq 0\), or \(r_0 \geq 0\), \(s_0 > 0\), and \(\nu_{0, w_0}(F) \leq 0\), then case $(b_1)$ cannot occur.
    \end{enumerate}
\end{Lem}
\begin{proof}
We first assume \( r_0 \leq 0 \) and \( \nu_{0, w_0}(F) < 0 \). Then \(\Pi_{\epsilon}(F)\) lies on the left-hand side of the \((b, w)\)-plane, and the slope of the wall is greater than or equal to the slope of the numerical wall \(\mathcal{W}(F, \mathcal{O}_X)\), which tends to \(\frac{s_0}{E \cdot H}\) as \(\epsilon \to 0^+\).  
Hence, there is a lower bound for the slope of walls for $F$ with negative slope. On the other hand, in case ($b_2$) the slope $\nu_{b,w}(Q_2)$ (independent of $b,w$ since $\ch_0(Q_2)=0$) tends to \(-\infty\) as \(\epsilon \to 0^+\), and so this case cannot occur. A similar argument shows that case $(b_2)$ cannot occur when \( r_0 \geq 0 \) and \(\nu_{0, w_0}(F) > 0 \).  

Now suppose \( r_0 \leq 0 \), \( s_0 < 0 \), and \(\nu_{0, w_0}(F) > 0 \). As described in the proof of Proposition \ref{prop-main-wallcrossing}, there are finitely many possible values of \( q_i \) for stable factors \( F_i = \mathcal{O}_X(-q_i E)[1] \); for each of them, we have  
\[
\frac{H_{\epsilon} \cdot \mathrm{ch}_1(F)}{(H_{\epsilon})^2 \cdot \mathrm{ch}_0(F)} < \frac{H_{\epsilon} \cdot \mathrm{ch}_1(F_i)}{(H_{\epsilon})^2 \cdot \mathrm{ch}_0(F_i)}
\]  
as \(\epsilon \to 0^+\), which implies that the wall \(\mathcal{W}(F, F_i)\) has negative slope, and so the claim follows. Similarly $(b_1)$ cannot occur when \( r_0 \geq 0 \), \( s_0 > 0 \), and \(\nu_{0, w_0}(F) \leq 0 \).
\end{proof}


Let us point out that, via a similar argument as in Proposition \ref{prop-main-wallcrossing}, we can show that Gieseker stability and slope-stability coincide for our particular Mukai vectors.

\begin{Lem}\label{lem-slope-stability}
    Fix \(v = (r_0, H - a_0E, s_0+r_0)\) for some \(a_0 \in \mathbb{Z}_{\geq 0}\) and $r_0 >0$. We may choose \(\epsilon_{v}\) suitably such that if \(0 < \epsilon < \epsilon_{v}\), then a sheaf $F$ of Mukai vector $v$ is $H_{\epsilon}$-Gieseker stable if and only if it $\mu_{H_{\epsilon}}$-stable.
\end{Lem}
\begin{proof}
Assume there is a $H_{\epsilon}$-Gieseker stable sheaf $F$ of class $v$ which is not $\mu_{H_{\epsilon}}$-stable. Arguing as in Proposition~\ref{prop-main-wallcrossing}, one finds a lower bound $M$ for the quantity $\ch_0\cdot\ch_2$ of $\mu_{H_\epsilon}$-stable factors of an appropriate filtration of $F$. This gives a distinguished $\mu_{H_\epsilon}$-stable factor $Q_1$ with $\ch(Q_1) = (r, H - aE, s)$; any other stable factor has $\ch_1$ multiple of $E$. 
   
We know $0 < r < r_0$ and $a\geq a_0$, and then the condition $v(Q_1)^2 \geq -2$ (together with the inequality $rs\geq M$) provides an upper bound for $a$ depending on the class $v$ (independently of $\epsilon$). As for given $a$ and $r$ the equality
\[
\frac{{H_\epsilon}\cdot(H-aE)}{r}=\frac{{H_\epsilon}\cdot(H-a_0E)}{r_0}
\]
happens for (at most) one value of $\epsilon$, it follows that for $\epsilon_v$ sufficiently small there can be no non-trivial $\mu_{H_\epsilon}$-destabilizing factor, and hence $F$ must be $\mu_{H_{\epsilon}}$-stable.
\end{proof}

\subsection{Moduli spaces of stable objects} We conclude this section with a brief discussion on moduli spaces of \( \sigma_{b,w} \)-stable objects for our special Mukai vectors $v$. Keeping the convention of Remark \ref{rem-00} (that is, by $\sigma_{0,0}$ we mean a stability condition on the upper part of the numerical wall $\cW(\cO_X,v)$), we have the following result.

\begin{Prop}\label{moduli-stable}
    Fix a Mukai vector \(v = (r_0, H - a_0E, s_0+r_0)\) for some \(a_0 \in \mathbb{Z}_{\geq 0}\) such that $(r_0, s_0) \neq (0, 0)$. Pick $0< \epsilon < \epsilon_v$ (such that $\epsilon$ is generic if $r_0s_0 =0$) and let $\sigma := \sigma_{0, w}$ for $w \geq 0$ lie in no actual wall for $v$ in $U_\epsilon$. Then:
    \begin{enumerate}
        \item\label{moduli-stable:a} Any $\sigma$-semistable object of class $v$ is $\sigma$-stable.

        \item\label{moduli-stable:b} There is a coarse moduli space $\cM_{\sigma}(v)$ parameterizing $\sigma$-stable objects of class $v$ in $\Coh^b(X)$, which is a smooth projective hyperk\"ahler variety of dimension $v^2+2$ (in particular, it is nonempty if and only if $v^2\geq -2$).
    \end{enumerate}
\end{Prop}
\begin{proof}
We only prove \eqref{moduli-stable:a}. Then \eqref{moduli-stable:b} follows from a series of fundamental works, explained in detail in \cite[Section 6]{bayer:projectivity}. First, assume \( r_0 s_0 \neq 0 \). Suppose there exists an object \( F \in \Coh^0(X) \) that is strictly \( \sigma \)-semistable inside a chamber. Then $F$ admits a Jordan-Hölder filtration that is fixed within the chamber. However, $F$ has stable factors as described in Proposition \ref{prop-main-wallcrossing}.\eqref{prop-main-wallcrossing:B}, which do not have the same slope $\nu_{0,w+\delta}$ as \( F \) for \( \delta>0 \), leading to a contradiction.

Now suppose \( r_0 = 0 \), \( s_0 \neq 0 \). If an object \( F \in \Coh^0(X) \) of class \( v \) is strictly \( \sigma \)-semistable, then by \ref{prop-main-wallcrossing} there exists a Chern character \((0, (a - a_0)E, s) \) such  
\[
\frac{s_0}{H_\epsilon\cdot(H-a_0)E}=\frac{s}{(a-a_0)E\cdot H_\epsilon},\;\;\;0\leq (a - a_0)E \cdot H_{\epsilon} < (H - a_0 E) \cdot H_{\epsilon}.
\]
However, one can easily check that this happens only for discrete values of \( \epsilon \), so the claim holds for a generic choice of \( \epsilon \).
A similar argument applies to the case \( r_0 \neq 0 \), \( s_0 = 0 \). 
\end{proof}

The following provides sufficient conditions to guarantee that, along the  walls described in Proposition \ref{prop-main-wallcrossing}, a general stable object does not get destabilized:

\begin{Lem}\label{lem-dim}
Adopt the notations of Proposition \ref{prop-main-wallcrossing}, and assume $(s_0,w_0)\neq(0,0)$. Let $\sigma_0^+$ (resp.~$\sigma_0^-$) be a stability condition just above (resp.~just below) $\ell^+$. Then the locus in $\cM_{\sigma_0^+}(v)$ and $\cM_{\sigma_0^-}(v)$ of \(\sigma_{0, w_0}\)-strictly semistable objects has dimension strictly less than \(v^2 + 2=\dim\cM_{\sigma_0^\pm}(v)\), providing that we are in one of the following situations:  
\begin{enumerate}
    \item[(1)] Condition $(b_1)$ holds and: $\chi(\mathcal{O}_X(mE), v) \leq 0$ for $m>0$ (or $m=0$ and $\ell^+$ of negative slope), $\chi(\mathcal{O}_X(mE), v) \geq 0$ for $m<0$ (or $m=0$ and $\ell^+$ of positive slope).
    
    \item[(2)] Condition $(b_2)$ holds and $\chi(Q_2, v) + 2(a-a_0)< 0$. 
\end{enumerate}
\end{Lem}
\begin{proof}
    First assume that condition ($b_1$) holds for $m>0$, or for $m=0$ and $\ell^+$ of negative slope. When $\nu_{\sigma_0^+}(F)>\nu_{\sigma_0^+}(\cO_X(mE))$ (resp.~$\nu_{\sigma_0^+}(F)<\nu_{\sigma_0^+}(\cO_X(mE))$), the locally closed subset in $\cM_{\sigma_0^+}(v)$ of $\sigma_{0, w_0}$-semistable objects with $\hom(\cO_X(mE),F)=h\geq 1$ (resp.~$\hom(F,\cO_X(mE))=h\geq 1$) is either empty or of dimension 
    \begin{align*}
    (v-h(1,mE,1))^2+2+h\left(\langle v-h(1,mE,1), (1,mE,1)\rangle -h\right)=\\
    =v^2+2-h\langle v,(1,mE,1)\rangle-h^2<v^2+2,
    \end{align*}
    so the claim follows. The same argument works for $\cM_{\sigma_0^-}(v)$. Similarly, considering the shift $\cO_X(mE)[1]$, one can cover the remaining cases of the first assertion. 
    
    Now suppose condition ($b_2$) holds, so that $v(Q_2) = (0, (a-a_0)E, s)$. Such objects $Q_2$ vary in a $2\alpha$-dimensional moduli space, where $\alpha=\mathrm{gcd}(a-a_0,s)$, by \cite[Lemma~7.2]{bayer:projectivity}. Since $\hom(Q_1,Q_2)=0=\hom(Q_2,Q_1)$ then, if non-empty, the locus of $\sigma_{0, w_0}$-strictly semistable objects in $\cM_{\sigma_0^\pm}(v)$ has dimension less than or equal to
    \begin{align*}
        & (v-v(Q_2))^2+2+2\alpha +\langle v-v(Q_2), v(Q_2) \rangle = \\
        = \ & v^2+2  -\langle v, v(Q_2) \rangle + 2\alpha\leq v^2+2  -\langle v, v(Q_2) \rangle + 2(a-a_0),
    \end{align*}
    which is strictly less than $v^2+2$ under our assumption.
\end{proof}

\section{Stratification of the moduli of Bridgeland stable objects}\label{stratification}

Fix an elliptic $K3$ surface $X$ of degree $k > 2$ as in Proposition \ref{thm-k3 surface}, equipped with a two-dimensional slice of Bridgeland stability conditions \( U_{\epsilon} \) associated to the polarization \( H_\epsilon := E + \epsilon H \) for a fixed \( \epsilon \in \mathbb{Q}_{>0} \). In this section, we introduce the notion of \emph{Bridgeland stability type}, which is central in our analysis. It enables us to investigate the Brill-Noether stratification of moduli spaces of Bridgeland stable objects with appropriate Mukai vector, as detailed in Theorem \ref{thm.splitting} and Theorem \ref{thm-nonempty}. In Section \ref{section-application}, we apply these results to the specific Mukai vector \( (0, H, 1+d-g) \) to study the Brill-Noether theory of line bundles on curves in the linear system \( |H| \).

\subsection{Bridgeland stability type} Consider $p\in \mathbb N$ and $p$ pairs $$(e_1,m_1),\ldots, (e_p,m_p)\in\bZ_{\geq0}\times\bZ_{> 0}$$ such that $e_1>\cdots >e_p\geq0$. Write $\overline{e}:=\bigl((e_i, m_i)\bigr)_{i=1}^{p}$ (so that $\overline{e}=\emptyset$ if $p=0$).

\begin{Def}\label{Def.splitting}
 Let $F \in \text{Coh}^0(X)$ be a $\sigma_{0, w_0}$-stable object for some $w_0 >0$, such that $\ch_2(F) \neq 0$. 
  If $p\geq 1$, we say that $F$ is \emph{of (Bridgeland) stability type} $\overline{e}$ if:
\begin{enumerate}
    \item[(i)] When we move down from the point $(0, w_0)$, then  $F$ gets destabilized along the wall $\mathcal{W}_1 \subset U_{\epsilon}$ passing through $(0, w_1)$ for some $w_1 \in [0, w_0)$ via the destabilizing sequence 
    \begin{equation*}
        \Hom(\cO_X(e_1E), F) \otimes \cO_X(e_1E) \stackrel{\ev}\longrightarrow F \longrightarrow F_1,
    \end{equation*}
    such that $m_1 = \hom(\cO_X(e_1E), F)$ and $F_1$ is stable along $\cW_1$.  
\vspace{1mm}
    
    \item[(ii)] 
    Then we move down the vertical line $b=0$ and inductively obtain the object $F_{i}$ along the wall $\mathcal{W}_i \subset U_\epsilon$ where $F_{i-1}$ gets destabilized. The destabilizing sequence is given by   
     \begin{equation*}
         \Hom(\cO_X(e_iE), F_{i-1}) \otimes \cO_X(e_iE)  \stackrel{\ev}\longrightarrow F_{i-1} \longrightarrow F_i,
    \end{equation*}
    where $m_i = \hom(\cO_X(e_iE), F_{i-1})$ and $F_i$ is stable along the wall $\mathcal{W}_i$ which passes through $(0, w_i)$ for some $w_i \in [0,  w_{i-1})$.

    \vspace{1mm}
    
    \item[(iii)] The final object $F_p$ is stable along the numerical wall $\cW(\cO_X,F_p)$ made with the structure sheaf $\cO_X$.
\end{enumerate}    

If $p=0$, we say that $F$ is \emph{of (Bridgeland) stability type $\emptyset$} if it is stable along the numerical wall $\cW(\cO_X,F)$.
\end{Def}

\vspace{0.5mm}

\begin{Rem}\label{obs stabtype}
Two important points should be noted regarding Definition \ref{Def.splitting}:
\begin{itemize*}
    \item If $F \in \Coh^0(X)$ is $\sigma_{0, w_0}$-stable of stability type $\overline{e}$, then for any $1\leq j\leq p$ the quotient $F_j$ is $\sigma_{0, w_j}$-stable of stability type $\bigl((e_i, m_i)\bigr)_{i=j+1}^{p}$. In particular, the inequality $v(F_p)^2\geq -2$ holds.
    \item We know that the object $F_i$ (for $1 \leq i \leq p$) has Mukai vector
\[
v(F_i)= v(F) - \sum_{j=1}^{i}m_j v(\cO_X(e_jE)),
\]
and it is stable along the wall $\cW_{i}$ where it has the same slope as $\cO_X(e_iE)$. Since $\cO_X(e_iE)$ is also stable along $\mathcal{W}_i$ by Lemma \ref{lem-eE}, we have $\Hom(\cO_X(e_iE), F_i) = 0 = \Hom(F_i, \cO_X(e_iE))$ which implies 
\begin{equation*}
    \ext^1(F_i ,\cO_X(e_iE) ) = -\chi\left(F_i,\cO_X(e_iE)\right)=\langle v(F_i),v(\cO_X(e_iE))\rangle  . 
\end{equation*}
\end{itemize*}
\end{Rem}

\vspace{1.5mm}

These observations enable us to endow the set of stable objects of a fixed stability type with a natural algebro-geometric structure (given as an open subset of an iterated Grassmannian bundle over a certain Bridgeland moduli space): 

\begin{Thm}\label{Prop-scheme structure}
Fix a Mukai vector \(v = (r_0, H - a_0E, s_0+r_0)\) for some \(a_0 \in \mathbb{Z}_{\geq 0}\) 
and $s_0 \neq 0$ and a stability condition $\sigma_0:=\sigma_{0, w_0}$ with $w_0>0$ which does not lie on an actual wall \footnote{We always assume $\epsilon > 0$ is generic if $r_0=0$.} for class $v$. Then  
for any stability type $\overline{e}=\bigl((e_i, m_i)\bigr)_{i=1}^{p}$ where $p\geq0$, the subset
    \[
    \cM_{\sigma_0}(v,\overline{e}):=\bigl\{F\in\cM_{\sigma_0}(v): \;\text{$F$ is of stability type $\overline{e}$}\bigr\}
    \]
admits a natural scheme structure as a locally closed subscheme of $\cM_{\sigma_0}(v)$. 
Moreover, if $\cM_{\sigma_0}(v,\overline{e})$ is non-empty, then it is smooth and irreducible of dimension
    \[
    \left(v-\sum_{i=1}^p m_i(1,e_iE,1)\right)^2+2+\sum_{j=1}^p m_j\left(\bigl\langle v-\sum_{i=1}^j m_i(1,e_iE,1),(1,e_jE,1)\bigr\rangle-m_j\right).
    \]  
\end{Thm}
\begin{proof}
We use induction on the length $p$ of the stability type to endow $\cM_{\sigma_0}(v,\overline{e})$ with a scheme structure that satisfies the required properties.

If $p=0$ (namely $\overline{e}=\emptyset$), then $\cM_{\sigma_0}(v,\emptyset)\subseteq \cM_{\sigma_0}(v)$ is the (possibly empty) open subset of objects that remain stable along the numerical wall $\cW(\cO_X,v)$ defined by $\cO_X$ and any object of Mukai vector $v$.

\vskip 3pt

Assume  the lemma holds for all Mukai vectors with $\ch_1 = H-aE$ and for all stability types of length $\leq p-1$. Let $\sigma_1^+$ be a stability condition on the vertical line $b=0$, sufficiently close to $\cW_1=\cW(\cO_X(e_1E),v)$ and above $\cW_1$. If $\cM_{\sigma_0}(v,\overline{e})$ is non-empty, then necessarily $\cM_{\sigma_1^+}\bigl(v-m_1(1,e_1E,1),\;\overline{e}\setminus (e_1,m_1)\bigr)$ must be non-empty. Also by the induction hypothesis, it is a smooth, irreducible, locally closed subscheme of $\cM_{\sigma_1^+}\left(v-m_1(1,e_1E,1)\right)$ of dimension
    \[
    \left(v-\sum_{i=1}^p m_i(1,e_iE,1)\right)^2+2+\sum_{j=2}^p m_j\left(\bigl\langle v-\sum_{i=1}^j m_i(1,e_iE,1),(1,e_jE,1)\bigr\rangle-m_j\bigr)\right).
    \]

Let $\mathcal{G}$ denote the Grassmanian bundle over $\cM_{\sigma_1^+}\bigl(v-m_1(1,e_1E,1),\;\overline{e}\setminus (e_1,m_1)\bigr)$, whose fiber over an object $F_1$ is given by
$\Gr\left(m_1,\Ext^1(F_1,\cO_X(e_1E))\right)$. By Lemma \ref{lem-stability above the wall}, we have a natural morphism
\[
\varphi\colon \mathcal{G}\longrightarrow \cM_{\sigma_1^+}^{\mathrm{st}}(v)
\]
which sends a pair $(F_1,V)\in\cG$ to the object $F\in\cM_{\sigma_1^+}(v)$ sitting in the exact triangle
\begin{equation}\label{extF}
    V^\vee\otimes\cO_X(e_1E)\overset{d_1}{\longrightarrow} F\overset{d_2}{\longrightarrow} F_1\overset{d_3}{\longrightarrow} V^\vee\otimes\cO_X(e_1E)[1].
\end{equation}
It follows from Definition \ref{Def.splitting} that $\varphi$ establishes a (set-theoretic) bijection between $\mathcal{G}$ and the subset of $\sigma_1^+$-stable objects of class $v$ with stability type $\overline{e}$. 
Furthermore, $\varphi$ is a locally closed immersion. This is a consequence of the following:

\begin{Claim}
    The differential of $\varphi$ is injective at every point.
\end{Claim}
\begin{proof}[Proof of the claim]
The tangent space to $\cG$ at $(F_1,V)$ is given by
\[
T_{[F_1]}\left(\cM_{\sigma_1^+}\left(v-m_1(1,e_1E,1),\;\overline{e}\setminus (e_1,m_1)\right)\right)\times T_{[V]}\Gr\left(m_1,\Ext^1(F_1,\cO_X(e_1E))\right),
\]
and is contained in 
\[
T_{[F_1]}\left(\cM_{\sigma_1^+}\left(v-m_1(1,e_1E,1)\right)\right)\times T_{[V]}\Gr\left(m_1,\Ext^1(F_1,\cO_X(e_1E))\right).
\]
On the other hand, if $F=\varphi(F_1,V)$, we have canonical identifications:
\begin{itemize}
    \item $T_{[F_1]}\bigl(\cM_{\sigma_1^+}\left(v-m_1(1,e_1E,1)\right)\bigr)=\Ext^1(F_1,F_1)$.
    \item $T_{[V]}\Gr\left(m_1,\Ext^1(F_1,\cO_X(e_1E))\right)= \Hom(V,\Ext^1(F_1,\cO_X(e_1E))/V)$. By applying the long exact sequence of $\Ext^i(-,\cO_X(e_1E))$ to \eqref{extF} we also obtain 
    \[
    \Ext^1(F_1,\cO_X(e_1E))/V\cong \Ext^1(F,\cO_X(e_1E)),
    \]
    and therefore we canonically have
    \[T_{[V]}\Gr\left(m_1,\Ext^1(F_1,\cO_X(e_1E))\right)= V^\vee\otimes \Ext^1(F,\cO_X(e_1E)).
    \]
     \item $T_{[F]}\bigl(\cM_{\sigma_1^+}(v)\bigr)=\Ext^1(F,F)$.
\end{itemize}
Applying $\Ext^i(-,F_1)$ to \eqref{extF} gives $\Hom(F, F_1) = \mathbb{C}$ and the natural inclusion 
$$
\lambda\colon \Ext^1(F_1,F_1)\hookrightarrow \Ext^1(F,F_1).$$ 
Moreover, applying $\Ext^i(F,-)$ to \eqref{extF} results in
the long exact sequence
\[
0\to V^\vee\otimes \Ext^1(F,\cO_X(e_1E))\to \Ext^1(F,F)\overset{\gamma}{\to} \Ext^1(F,F_1)\overset{\gamma'}{\to} V^\vee\otimes \Ext^2(F,\cO_X(e_1E)) \to \dots
\]
Note that $\text{im}(\lambda) \subset \ker (\gamma')$, as any composition $F \overset{d_2}{\to} F_1 \to F_1[1] \overset{d_3[1]}{\to} V^{\vee} \otimes \cO_X(e_1E)[2]$ vanishes since $\hom(F_1,\cO_X(e_1E)[2])=\hom(\cO_X(e_1E),F_1)=0$. Thus we get the  diagram  
    \[
 \begin{xy}
\xymatrix{
   &   &  \Ext^1(F_1,F_1)\ar@{^{(}->}[d]
\\
V^\vee\otimes \Ext^1(F,\cO_X(e_1E))\ \ar@{^{(}->}[r] &\ \Ext^1(F,F)\ar@{->>}[r] & \text{im}(\gamma)
}
\end{xy}
\]
realizing the inclusion $T_{(F_1,V)}(\cG)\subseteq T_{[F]}\left(\cM_{\sigma_1^+}(v)\right)$ defined by $d\varphi$.
\end{proof}

Since $\varphi$ is a locally closed immersion, we can  endow $\cM_{\sigma_1^+}(v,\overline{e})$ with the scheme structure provided by $\mathcal{G}$; it becomes a smooth, irreducible, locally closed subscheme of $\cM_{\sigma_1^+}(v)$. 
On the other hand, every $F_1\in\cM_{\sigma_1^+}\bigl(v-m_1(1,e_1E,1),\;\overline{e}\setminus (e_1,m_1)\bigr)$ satisfies
\begin{align*}
        \dim \Gr(m_1,\Ext^1(F_1,\cO_X(e_1E))=m_1\left(\ext^1(F_1,\cO_X(e_1E))-m_1\right)=\\
        =m_1\left(-\chi(F_1,\cO_X(e_1E))-m_1\right)=m_1\bigl(\langle v-m_1(1,e_1E,1),(1,e_1E,1)\rangle-m_1\bigr).
    \end{align*}
This implies that $\mathcal{G}$ (hence $\cM_{\sigma_1^+}(v,\overline{e})$) has dimension
    \[
    \left(v-\sum_{i=1}^p m_i(1,e_iE,1)\right)^2+2+\sum_{j=1}^p m_j\left(\langle v-\sum_{i=1}^j m_i(1,e_iE,1),(1,e_jE,1)\rangle-m_j\right).
    \]

Finally, to conclude the proof we simply observe that $\cM_{\sigma_0}(v,\overline{e})$ is an open subset of $\cM_{\sigma_1^+}(v,\overline{e})$, consisting of those objects in $\cM_{\sigma_1^+}(v,\overline{e})$ that are $\sigma_0$-stable.
\end{proof}

\vskip 5pt

\subsection{Stratification of Brill-Noether loci} 
We now establish the main result of this section, which provides a stratification of moduli of stable objects. Fix a primitive Mukai vector $v\in\Lambda$ and $\sigma\in U_\epsilon$ lying on no actual wall for $v$, therefore  there is a proper coarse moduli space $\cM_\sigma(v)$ parameterizing $\sigma$-semistable object of class $v$ in $\Coh^b(X)$. For each $r\geq -1$ we define the closed subscheme
\begin{equation}\label{Wrd}
W_{\sigma}^r(v):= \Bigl\{F \in \cM_\sigma(v) \colon \hom(\cO_X, F) \geq r+1\Bigr\}\subseteq \cM_\sigma(v).    
\end{equation}
Let  $V_{\sigma}^r(v)$ be the  subset of $W_{\sigma}^r(v)$ corresponding to objects $F$ with $\hom(\cO_X,F) =r+1$.

\vskip 3pt

The following result, taking advantage of the wall crossing analysis performed in Section \ref{wallcrossing-special}, shows that every object in $V^r_\sigma(v)$ has a stability type, which is strongly constrained by $r$. 

\begin{Thm}\label{thm.splitting}
Fix $r\geq -1$ and a Mukai vector \(v = (r_0, H - a_0E, s_0+r_0)\in\Lambda\) with $r_0 \leq 0$, \(a_0 \geq0\) and $s_0<0$. There exists \(\epsilon(v, r) > 0\) such that if \(\epsilon < \epsilon(v, r)\) and $w_0>0$ satisfies $\nu_{0,w_0}(v)<0$, then 
    \begin{equation*}
        V_{\sigma_{0, w_0 }}^r (v) \subseteq \bigcup_{\overline{e} \in I} \cM_{\sigma_{0, w_0 }}(v, \overline{e}),
    \end{equation*}
      where $I$ is the (finite) set of stability types $\overline{e} =\bigl((e_i, m_i)\bigr)_{i=1}^p$ with $p \geq 0$ and 
      \begin{enumerate}
          \item\label{thm.splitting: ineq-a} $\sum_{i=1}^p m_i\leq r+1\leq \sum_{i=1}^pm_i(e_i+1)$,
          \item\label{thm.splitting: ineq-b} $m_1(e_1+1)\leq r+1$. 
      \end{enumerate}
\end{Thm}

\begin{proof}
First assume $r\geq 0$. Take $F\in V_{\sigma_{  0, w_0 }}^r (v)$ with $\nu_{0,w_0}(F)<0$. Note that $F$ cannot be stable along the numerical wall $\cW(F,\cO_X)$, as $\hom(\cO_X,F)=r+1>0$. Hence we encounter an actual  wall \( \mathcal{W}_1 \) for $F$ passing through $(0,w_1)$, for some $w_1\in[0,w_0)$.

\vskip 3pt

We apply Proposition \ref{prop-main-wallcrossing} for the class \( v \) and the wall $\cW_1$. The assumption $r_0\leq0$ gives $\nu_{0,w_1}(F)\leq \nu_{0,w_0}(F)<0$, thus by Lemma \ref{lem-wall-crossing-negative} the wall is of type ($b_1$), namely $F$ gets destabilized via a short exact sequence
\[
\mathcal{O}_X(e_1E) \otimes \Hom( \mathcal{O}_X(e_1E), F) \longrightarrow F \longrightarrow F_1
\]  
such that $e_1\geq0$ and \( F_1 \) is \( \sigma_{0, w_1} \)-stable. Writing \( m_1 := \hom(\mathcal{O}_X(e_1E), F) \), we get \( m_1(e_1 + 1) \leq r + 1 \)  by applying \( \Hom(\mathcal{O}_X, -) \) to the above triangle.

\vskip 3pt

If $F_1$ is stable along the numerical wall $\cW(\cO_X,F_1)$, then $F$ has stability type $\bigl((e_1,m_1)\bigr)$. In this case $\hom(\cO_X,F_1)=0$ which implies $m_1(e_1+1)=r+1$, hence inequalities \eqref{thm.splitting: ineq-a} and \eqref{thm.splitting: ineq-b} are trivially satisfied. Otherwise, $F_1$ gets destabilized along an actual wall $\cW_2$ passing through $(0,w_2)$, with $w_2\in[0,w_1)$. Since $\ch_0(F_1)=r_0-m_1<0$ and $\ch_2(F_1)=s_0<0$, we can apply Proposition \ref{prop-main-wallcrossing} and Lemma \ref{lem-wall-crossing-negative} with the class $v(F_1)$. It follows that the destabilizing subobject of $F_1$ along $\cW_2$ is $\cO_X(e_2E)\otimes\Hom(\cO_X(e_2E),F_1)$, where $0\leq e_2<e_1$.

\vskip 3pt
  
We apply this procedure inductively, constructing \( F_i \) via the destabilizing sequence 
\[
0 \longrightarrow \Hom(\mathcal{O}_X(e_iE), F_{i-1}) \otimes \mathcal{O}_X(e_iE) \longrightarrow F_{i-1} \longrightarrow F_i \longrightarrow 0,
\]
so that \( F_i \) is stable along \( \mathcal{W}_i := \mathcal{W}(\mathcal{O}_X(e_iE), F_{i-1}) \) and \( m_i := \hom(\mathcal{O}_X(e_iE), F_{i-1}) \). If \( F_i \) is stable along the numerical wall \( \mathcal{W}(\mathcal{O}_X, F_i) \), then the process ends. Eventually, we obtain that \( F \) has some stability type \( \overline{e} = \{(e_i, m_i)\}_{i=1}^p \) with $p\geq1$. It only remains to prove the required inequalities in \eqref{thm.splitting: ineq-a}.

To that end, observe that \(h^2(X, \mathcal{O}_X(e_iE)) = 0 \) for \( 1 \leq i \leq p-1 \) (as $e_i>0$) and \( h^2(X, \mathcal{O}_X(e_pE)) \leq 1 \) (as $e_p\geq0$). Moreover, $\ext^2(\cO_X,F)=\hom(F,\cO_X)=0$ by the $\sigma_{0,w_0}$-stability of $F$. It follows that
\begin{align*}
    r+1-s_0-2r_0 = \hom(\cO_X,F)-\chi(\cO_X,F) = \ext^1(\cO_X,F) \\
    \geq \ext^1(\cO_X,F_1) \geq \dots 
    \geq \ext^1(\cO_X,F_{p-1}) 
    \geq \ext^1(\cO_X,F_p) - m_p \\
    \geq -\chi(\cO_X,F_p) - m_p = 2(m_1 + \dots + m_{p-1}) + m_p-s_0-2r_0,
\end{align*}
    which implies \( 2(m_1 + \dots + m_{p-1}) + m_p \leq r + 1 \), and so \( \sum_{i=1}^p m_i \leq r + 1 \). 
    
    On the other hand, since \( \hom(\cO_X,F_p) = 0 \) (recall that \( F_p \) is stable along the numerical wall \( \mathcal{W}(F_p, \mathcal{O}_X) \)), we have \( \hom(\cO_X,F_{p-1}) = m_p(e_p + 1) \). Therefore,
    \[
    \hom(\cO_X,F_{p-2}) \leq \hom(\cO_X,F_{p-1}) + m_{p-1}(e_{p-1} + 1) = m_p(e_p + 1) + m_{p-1}(e_{p-1} + 1).
    \]
    Repeating this process, the inequality \( r + 1 = \hom(\cO_X,F) \leq \sum_{i=1}^p m_i(e_i + 1) \) follows.
\vskip 3pt

Finally, for $r=-1$ the assertion is that any $F\in\cM_{\sigma_{0,w_0}}(v)$ with $\hom(\cO_X,F)=0$ has empty stability type, i.e. it is stable along the numerical wall $\cW(\cO_X,F)$. This is an immediate application of Proposition \ref{prop-main-wallcrossing} and Lemma \ref{lem-wall-crossing-negative}.
 \end{proof}

\vskip 2pt

\begin{Rem}\label{partial compact}
    Given \(v = (r_0, H - a_0E, s_0+r_0)\in\Lambda\) as in Theorem \ref{thm.splitting} and a stability type $\overline{e} =\bigl((e_i, m_i)\bigr)_{i=1}^p$ with $e_p>0$, for $w_0>0$ consider the union 
   \begin{equation}\label{union}
       \widetilde{\cM}_{\sigma_{0,w_0}}(v,\overline{e}):=\bigcup_{\overline{e'}\in I_{\overline{e}}} \cM_{\sigma_{0,w_0}}\bigl(v,\overline{e}\cup\overline{e'}\bigr)
   \end{equation}
    where $I_{\overline{e}}$ is the set of types $\overline{e'}=\bigl((e_i',m_i')\bigr)_{i'=1}^{p'}$ ($p'\geq0$) with $e_1'<e_p$. Let us write $v_p:=v-\textstyle\sum_{i=1}^p m_i(1,e_iE,1)$. Note that $\widetilde{\cM}_{\sigma_{0,w_0}}(v,\overline{e})$ is defined by a finite union, indexed by the finite set of $\overline{e'}\in I_{\overline{e}}$ such that $\left(v_p-\textstyle\sum_{i=1}^{p'}m_i'(1,e_i'E,1)\right)^2\geq -2$.

    As a consequence of Theorem \ref{thm.splitting}, 
    there is $\epsilon(v,\overline{e})>0$ such that, if $\epsilon<\epsilon(v,\overline{e})$ and $w_0>0$ satisfies $\nu_{0,w_0}(v)<0$ and $\cM_{\sigma_{0,w_0}}(v,\overline{e})\neq\emptyset$, then $\widetilde{\cM}_{\sigma_{0,w_0}}(v,\overline{e})$
    is smooth and irreducible with $\cM_{\sigma_{0,w_0}}(v,\overline{e})$ as an open dense subset. Indeed, as proven in Theorem \ref{Prop-scheme structure}, $\cM_{\sigma_{0,w_0}}(v,\overline{e})$ is (an open subset of) an iterated Grassmann bundle over an open subset $\mathcal{U}\subseteq \cM_{\sigma_p}^{\mathrm{st}}(v_p)$; here $\sigma_p$ lies on the line $b=0$ and the numerical wall $\cW(\cO_X(e_pE),v_p)$, and $\cM_{\sigma_p}^{\mathrm{st}}(v_p)\subseteq \cM_{\sigma_p}(v_p)$ is the open subset of stable objects. Precisely, $\mathcal{U}$ parametrizes the objects that remain stable along $\cW(\cO_X,v_p)$.
    By Theorem \ref{thm.splitting}, for $\epsilon>0$ small enough,  any object in $\cM_{\sigma_p}^{\mathrm{st}}(v_p)$ has an associated stability type. Then $\widetilde{\cM}_{\sigma_{0,w_0}}(v,\overline{e})$ can be constructed as (an open subset of) an iterated Grassmann bundle over $\cM_{\sigma_p}^{\mathrm{st}}(v_p)$.
\end{Rem}

\vspace{3mm}

\subsection{Balanced stability types} 
To conclude this section we focus on a particular class of stability types that, in view of H. Larson's nomenclature \cite{larson:inv}, we will call \emph{balanced}.

\begin{Def}
    A stability type $\overline{e}$ is \emph{balanced} if it is of the form $\bigl((e+1,m_1),(e,m_2)\bigr)$ for some $e, m_1 \geq 0$ and $m_2>0$. 
\end{Def}

A balanced stability type $\overline{e}=\{(e+1,m_1),(e,m_2)\}$ determines the number of global sections. Indeed, since $\Ext^1\bigl(\cO_X(eE),\cO_X((e+1)E)\bigr)=0$, one checks that
\[
\cM_\sigma(v,\overline{e})\subseteq V^r_\sigma(v)
\]
for $r:=m_1(e+2)+m_2(e+1)-1$.

The geometric description in Theorem \ref{Prop-scheme structure} is conditional to the nonemptiness of $\cM_\sigma(v,\overline{e})$. This issue can be resolved for balanced stability types: 

\begin{Thm}\label{thm-nonempty}
     Fix a Mukai vector  $v = (r_0, H - a_0E, s_0 + r_0) \in \Lambda$ with \( r_0 \leq 0 \), \( a_0 \geq 0 \), and \( s_0 < 0 \). Consider the stability type  $\overline{e} := \bigl((e+1, m_1), (e, m_2)\bigr)$  
for \( e, m_1 \geq 0 \) and \( m_2 > 0 \) such that  
\begin{equation}\label{assumptionineq}
    m_1 + m_2 \leq k + r_0,
\end{equation}
\begin{equation}\label{necessaryassumption}
    \Bigl(v-m_1\cdot v(\cO_X((e+1)E))-m_2\cdot v(\cO_X(eE))\Bigr)^2\geq -2.
\end{equation}
There exists \(\epsilon(v, \overline{e}) > 0\) (depending on $v$ and $\overline{e}$), such that if \(\epsilon < \epsilon(v, \overline{e})\) and $w_0>0$ satisfies $\nu_{0,w_0}(\cO_X((e+1)E))<\nu_{0,w_0}(v)<0$, then   
$\mathcal{M}_{\sigma_{0,w_0}}(v, \overline{e})$ is non-empty.
\end{Thm}

Note that, according to Remark \ref{obs stabtype}, inequality \eqref{necessaryassumption} is a necessary condition for the non-emptiness of $\mathcal{M}_{\sigma_{0,w_0}}(v, \overline{e})$. Theorem \ref{thm-nonempty} establishes that, under the assumption \eqref{assumptionineq}, \eqref{necessaryassumption} is also a sufficient condition.

\vspace{2mm}

The proof of Theorem \ref{thm-nonempty} occupies the rest of this section. The essential step is provided by the following result:

\begin{Prop}\label{prop.general}
    Let $q\in\bZ_{\geq0}$ and $b\in\bR$, and let $Q\in\Coh^b(X)$ be an object such that: 
\begin{enumerate}
    \item\label{prop.general: a} $\ch(Q)=(r_0,H-a_0E,s_0)$ with $r_0<0$, $a_0\geq0$ and $s_0\leq0$.
    \item\label{prop.general: b} $\Hom(Q,\cO_J)=0$ for a general elliptic curve $J\in|E|$.
    \item\label{prop.general: c} $\Hom(Q,\cO_X(uE))=0$ for all $u\geq q$.
\end{enumerate}
If $0 < m \leq k-\hom(\cO_X((q-1)E), Q)$, then for a general $V \in\Gr\left(m,\Ext^1(Q,\cO_X(qE))\right)$ the object $Q_1$ sitting in an exact triangle 
    \[
    V^\vee\otimes\cO_X(qE)\longrightarrow Q_1\longrightarrow Q
    \]
satisfies $\hom(Q_1, \cO_X(uE)) = 0$ for all $u > q$. 
\end{Prop}
\begin{proof}
The claim amounts to the injectivity of the map
    \begin{equation}\label{multmap}
        V\otimes\Hom\left(\cO_X(qE),\cO_X(uE)\right)\longrightarrow \Ext^1(Q,\cO_X(uE))
    \end{equation}
    for all $u> q$, if $V\in\Gr\left(m,\Ext^1(Q,\cO_X(qE))\right)$ is a general subspace.

\vskip 4pt

Note that \eqref{multmap} is the restriction of the map
    \[
    \psi\colon \Hom(Q,\cO_X(qE)[1]) \otimes\Hom\left(\cO_X(qE),\cO_X(uE)\right)\longrightarrow \Hom(Q,\cO_X(uE)[1])
    \]
    given by composition. Moreover, $\psi$ is the result of applying $\Ext^1(Q,-)$ to the sequence
    \[
    0\longrightarrow \cO_X((q-1)E)^{\oplus u-q}\longrightarrow \Hom\left(\cO_X(qE),\cO_X(uE)\right)\otimes \cO_X(qE)\longrightarrow \cO_X(uE)\longrightarrow 0,
    \]
given by twist by $\cO_X(qE)$ of the sequence \eqref{evalutationseq}
. Since $\hom(Q,\cO_X(uE))=0$, one has  
    \begin{equation}\label{global kernel}
        \ker(\psi)\cong\Hom\bigl(Q,\cO_X((q-1)[1]\bigr)^{\oplus u-q}.
    \end{equation}
    The inclusion of this space in  $\Hom(Q,\cO_X(qE)[1]) \otimes\Hom\left(\cO_X(qE),\cO_X(uE)\right)$ is described by Lemma \ref{multpencil}.\eqref{evmatrix}: a tuple $(\alpha_1,...,\alpha_{u-q})\in \Hom(Q,\cO_X((q-1)E)[1])^{\oplus u-q}$ is mapped to 
        \[
        t\alpha_1\otimes s^{u-q}+(-s\alpha_1+t\alpha_2)\otimes s^{u-q-1}t+...+(-s\alpha_{u-q-1}+t\alpha_{u-q})\otimes st^{u-q-1}-s\alpha_{u-q}\otimes t^{u-q}.
        \]

        \vskip 3pt
        
    We claim that the two subspaces
    \begin{align*}
    V_t:=\Bigl\{t\alpha: \alpha\in \Hom(Q,\cO_X((q-1)E)[1])\Bigr\}\subseteq\Hom\bigl(Q,\cO_X(qE)[1]\bigr),\\
    V_s:=\Bigl\{s\alpha: \alpha\in \Hom(Q,\cO_X((q-1)E)[1])\Bigr\}\subseteq \Hom\bigl(Q,\cO_X(qE)[1]\bigr)
    \end{align*}
    are of codimension $\geq k-\hom(\cO_X((q-1)E), Q)$ in $\Hom\bigl(Q,\cO_X(qE)[1]\bigr)$. Indeed, applying the functor $\Hom(Q, -)$ to the short exact sequence \[
    0\longrightarrow\cO_X((q-1)E) \longrightarrow \cO_X(qE) \longrightarrow \cO_{J}\longrightarrow 0,
    \]
where $J\subseteq X$ is the elliptic curve corresponding to $s$ or $t$,  results in a long exact sequence
    \begin{align*}
    & \dots \longrightarrow \Hom\bigl(Q, \cO_X((q-1)E)[1]\bigr) \longrightarrow  \Hom\bigl(Q, \cO_X(qE)[1]\bigr)\\ & \overset{d}{\longrightarrow} \Hom(Q, \cO_J[1]) \longrightarrow \Hom\bigl(Q, \cO_X((q-1)E)[2]\bigr) \longrightarrow  \cdots    
    \end{align*}
    and therefore we have 
    \begin{align*}
      \dim \text{im}(d)  & \geq \hom\bigl(Q, \cO_J[1]\bigr) - \hom\bigl(Q, \cO_X((q-1)E)[2]\bigr)  \\
      & \geq -\chi(Q, \cO_J) - \hom\bigl(Q, \cO_X((q-1)E)[2]\bigr)= k - \hom(\cO_X((q-1)E), Q)
    \end{align*}
    (using that $\hom(Q, \cO_J[i]) = 0$ for  $i \neq 0, 1, 2$ since $Q, \cO_J\in \Coh^b(X)$).
\vskip 4pt

Since $m \leq k-\hom(\cO_X((q-1)E), Q)$, a general subspace $V\subseteq \mbox{Hom}(Q, \cO_X(qE)[1])$ of dimension $m$ has trivial intersection with both $V_{s}$ and $V_{t}$. For such a $V$, we show that
    \begin{equation*}
        \Bigl(V \otimes \Hom(\cO_X(qE), \cO_X(uE))\Bigr) \cap \Ker(\psi) = 0
    \end{equation*}
(that is, \eqref{multmap} is injective). Let $(\alpha_1, ..., \alpha_{u-q})\neq0$ be an element in the intersection, so that 
\begin{equation}\label{list}
     t\alpha_1\in V,\;\;-s\alpha_1+t\alpha_2\in V,\;\;...\;\;,-s\alpha_{u-q-1}+t\alpha_{u-q}\in V,\; \text{and} \;-s\alpha_{u-q}\in V.
\end{equation}

\vskip 1pt

Note that for an appropriate choice of $t\in H^0(\cO_X(E))$,  the following map is injective:  
    \[
    \gamma_t \colon \Hom(Q,\cO_X((q-1)E)[1])\longrightarrow \Hom(Q,\cO_X(qE)[1]),\;\;\;\alpha\mapsto t\alpha.
    \]
Indeed, it follows from the vanishing $\Hom(Q,\cO_X(qE))=0$ that $\ker(\gamma_t)=\Hom(Q,\cO_{J(t)})$, where $J(t)\in|E|$ is the elliptic curve corresponding to $t$. Thus injectivity of $\gamma_t$ for a general $t\in H^0(\cO_X(E))$ is just a consequence of \eqref{prop.general: b}. The first condition in \eqref{list} yields $t\alpha_1 \in V \cap V_t = 0$, hence $\alpha_1 = 0$ from the injectivity of $\gamma_t$. Inductively, the $i$-th condition in \eqref{list} for $i \leq u - q$ implies that $\alpha_i = 0$,
 which is a contradiction.
 \end{proof}

\vskip 4pt

\begin{proof}[Proof of Theorem \ref{thm-nonempty}]
It suffices to prove the statement for $\sigma_{0,w_0}$ lying in the Gieseker chamber (if $r_0=0$) or just below the horizontal line through $\Pi_{\epsilon}(v)$ (if $r_0\neq0$). 

\vskip 3pt

We start with the Mukai vector 
\begin{align*}
    v_2 := v - m_1\cdot  v(\cO_X((e+1)E)) - m_2\cdot v(\cO_X(eE)) 
\end{align*}
and the stability condition $\sigma_{e}^+$, which lies just above the line segment connecting $\Pi_{\epsilon}(v_2)$ to $\Pi_{\epsilon}(\cO_X(eE))$. Note that $v_2^2\geq -2$ is equivalent to our assumption \eqref{necessaryassumption}, in particular $\cM_{\sigma_{e}^+}(v_2)\neq\emptyset$. We first check that a general $F_2 \in \cM_{\sigma_{e}^+}(v_2)$ satisfies
\begin{equation}\label{vanishingF-2}
\hom(\cO_X(uE),F_2)=0=\hom(F_2,\cO_X(uE)),\text{ for all $u\geq0$}.
\end{equation}
Indeed, the vanishings
$\hom(\cO_X(uE),F_2)=0$ for $u> e$ and $\hom(F_2,\cO_X(uE))=0$ for $0\leq u\leq e$ follow from $\sigma_e^+$-stability of $F_2$.
Moreover, when moving (along the vertical line $b=0$) from $\sigma_{e}^+$ towards the horizontal line passing through $\Pi_{\epsilon}(v_2)$, Lemma \ref{lem-wall-crossing-negative} implies that we only encounter actual walls of type ($b_1$) created by $\cO_X(uE)$ for $u > e$. Similarly, when moving from $\sigma_{e}^+$ towards the wall $\cW(F_2,\cO_X)$ we only encounter walls of type ($b_1$) created by $\cO_X(uE)$ for $0\leq u \leq e$. However, we have
\begin{equation}\label{angle}
    \bigl\langle v_2, \cO_X(uE) \bigr\rangle = uE\cdot H -2r_0-s_0+ 2m_1 +2m_2 > 0, \qquad \text{for all } u \geq 0.
\end{equation}
Thus Lemma \ref{lem-dim} (and the finiteness of the number of walls for the class \( v_2 \)) implies \eqref{vanishingF-2} for a general \( F_2 \in \cM_{\sigma_{e}^+}(v_2) \). In particular, \( F_2 \) is \( \sigma \)-stable, where \( \sigma \) lies on the numerical wall \( \cW(F_2, \cO_X) \) or just below the horizontal line passing through \( \Pi_{\epsilon}(v_2) \).


Furthermore, the vanishing $\Hom(F_2,\cO_J)=0$ holds for a general curve $J\in|E|$. Indeed, consider the value $w_0>0$ for which $\nu_{0,w_0}(F_2)=0$. Then $F_2$ is $\sigma_{0,w_0}$-semistable and $\cO_J$ is $\sigma_{0,w_0}$-stable for any $J\in|E|$ (by Corollary \ref{cor:nowall}). If $\Hom(F_2,\cO_J)\neq0$, then $\cO_J$ is a stable quotient of $Q$ since $\nu_{0,w_0}(F_2)=0=\nu_{0,w_0}(\cO_J)$. Hence there can be only finitely many such $\cO_J$'s.

\vskip 1pt

We claim that a general $F_2 \in \cM_{\sigma_e^+}(v_2)$ satisfies
\begin{equation}\label{condition}
    0 < m_2 \leq k - \hom(\cO_X((e-1)E), F_2).
\end{equation}
We first prove the final statement assuming this claim and later return to justify it. 

Assuming (\ref{condition}) holds, we apply Proposition~\ref{prop.general} with $F_2$ in place of $Q$ and $e$ substituted for $q$. This shows that a general extension
\[
0 \longrightarrow \cO_X(eE)^{\oplus m_2} \longrightarrow F_1 \longrightarrow F_2 \longrightarrow 0
\]
satisfies $\hom(F_1, \cO_X(uE)) = 0$ for all $u > e$. Then $F_1$ is $\sigma_{e}^+$-stable and does not get destabilized along any wall $\cW(\cO_X(uE), F_1)$ ($u\geq e+1$) because $\hom(F_1, \cO_X(uE)) = 0$. One can argue as before to check the vanishing $\Hom(F_1,\cO_J)=0$ for a general $J\in|E|$. Thus by applying Proposition \ref{prop.general} again, we find that a general extension 
\[
0 \longrightarrow \cO_X((e+1)E)^{\oplus m_1} \longrightarrow F \longrightarrow F_1 \longrightarrow 0 
\]
satisfies $\hom(F, \cO_X(uE)) = 0$ for all $u > e+1$, provided $m_1 \leq k - \hom(\cO_X(eE), F_1) = k-m_2$. This clearly holds under Assumption \eqref{assumptionineq} as $r_0 \leq 0$. 

Clearly $F$ is stable with respect to stability conditions just above $\cW(\cO_X((e+1)E), F_1)$. By Lemma \ref{lem-wall-crossing-negative}, all negative-slope walls for $F$ are of type ($b_1$). Since $\hom(F, \cO_X(uE)) = 0$ for all $u > e+1$, Proposition \ref{prop.general} ensures that $F$ is $\sigma$-stable, where: $\sigma$ lies just below the horizontal line passing through $\Pi_{\epsilon}(v)$ (if $r_0 \neq 0$), or $ \sigma$ lies in the Gieseker chamber (if $r_0 = 0$). Moreover, by construction, it is of stability type $\overline{e}$. Thus, $\cM_{\sigma}(v, \overline{e}) \neq \emptyset$ as required. Hence, it only remains to prove condition \eqref{condition} for a general $F_2\in\cM_{\sigma_e^+}(v_2)$.

\vspace{.2 cm}

If $e\geq 1$, then \eqref{vanishingF-2} implies that $\hom(\cO_X((e-1)E), F_2) = 0$ for a general $F_2$, and so condition \eqref{condition} clearly holds. Thus, we may assume $e=0$. In order to prove \eqref{condition}, we want to control the quantity $\hom(\cO_X(-E),F_2)=\hom(\cO_X,F_2(E))$. To handle this case, we first prove that $F_2$ is stable in the Gieseker chamber of $v_2$. Then we twist $F_2$ by $\cO_X(E)$, and move down to the wall that $\cO_X$ makes for $F_2(E)$. 

\vskip 3pt

We already know that a general $F_2\in\cM_{\sigma_e^+}(v_2)$ is at least semistable along the horizontal line passing through $\Pi_{\epsilon}(v_2)$. This horizontal line, as well as all the walls for the class $v_2$ above it up to the Gieseker chamber, are of type ($b_2$) by Lemma \ref{lem-wall-crossing-negative}. Hence if $F_2$ is not stable in the Gieseker chamber, then it has a destabilizing factor of class $(0, (q-a_0)E, \theta)$ with $q-a_0 >0$ and $\theta \geq 0$. Since $k > 2$ implies $-E\cdot H + 2 < 0$, we have the inequality
\[
-\bigl\langle v_2, (0,(q-a_0)E,\theta) \bigr\rangle + 2(q-a_0) 
= (q-a_0)(-E\cdot H + 2) - \theta(m_1 + m_2 - r_0) < 0.
\]
By Lemma~\ref{lem-dim}, this shows that a general $F_2 \in \cM_{\sigma_e^+}(v_2)$ is stable in the Gieseker chamber. 

By applying Theorem \ref{Giesekerchamber} and Lemma \ref{lem-slope-stability}, the object $F_2(E) \in \Coh^0(X)$ is also stable in the Gieseker chamber of $v_2(E):=v\bigl(F_2(E)\bigr)$\footnote{ Note that twist gives an isomorphism of moduli spaces in the Gieseker chamber, so $F_2(E)$ is also a general object in the moduli space $\cM_{H_\epsilon} (v_2(E))$.}. We move down again along the vertical line $b=0$, and investigate walls for the Mukai vector $v_2(E)$. 
We distinguish two cases:  

\vspace{.2 cm}
\textbf{Case 1.} First assume $\ch_2(F_2(E))\leq0$. 
Since  
$-\bigl\langle v_2(E), \ (0, (q-a_0)E, \theta) \bigr\rangle + 2(q-a_0) <0$
for $q-a_0>0$ and $\theta\geq 0$, and also
\[
\langle v_2(E) , v(\cO_X(mE)) \rangle=mk+2(m_1+m_2-r_0)-\ch_2\bigl(F_2(E)\bigr)>0
\]
for all $m \geq 0$, then the same argument as above shows that $F_2(E)$ is stable along the numerical wall $\cW(F_2(E),\cO_X)$. In particular $\hom(\cO_X,F_2(E)) = 0$, and thus condition \eqref{condition} is satisfied.

\vspace{.2 cm} 

\textbf{Case 2.} Now assume $\ch_2\bigl(F_2(E)\bigr)>0$. Then the walls for the Mukai vector $v_2(E)$ that we encounter when moving down to the origin can be either of type ($b_1$) for $m<0$, or of type ($b_2$). But the latter does not destabilize a general $F_2(E)\in\cM_{H_\epsilon}\bigl(v_2(E)\bigr)$ as in Case 1; hence we only need to consider walls made by $\cO_X(mE)[1]$ for $m<0$. Again, we need to distinguish two different subcases: 

\vspace{.2 cm}

\textbf{Case 2.1.} First assume $\bigl\langle v_2(E) , v(\cO_X(-E)) \bigr\rangle \leq 0$. Then $\bigl \langle v_2(E) , v(\cO_X(mE)) \bigr\rangle \leq 0$ for all $m \leq -1$, and hence a general $F_2(E)\in\cM_{H_\epsilon}(v_2(E))$ is $\sigma_{0, w}$-stable for all $w>0$ by Lemma \ref{lem-dim}. If $\chi(\cO_X,F_2(E))\leq0$, then $\hom(\cO_X,F_2(E))=0$ for a general $F_2$ and \eqref{condition} is clearly satisfied, so we may assume $\chi(\cO_X,F_2(E)) > 0$. Let $\ell$ be the line passing through $\Pi_{\epsilon}(F_2(E))$ and the origin $(0, 0)$, and define
\[
\ell_1 := \ell \cap \bigl\{(b,w) \in U_{\epsilon} : w > 0\bigr\}, \quad 
\ell_2 := \ell \cap \bigl\{(b,w) \in U_{\epsilon} : w < 0\bigr\}.
\]
Here, $\ell_1$ represents the numerical wall created by $\cO_X[1]$ for the class $v_2 \otimes \cO_X(E)$, while $\ell_2$ corresponds to the numerical wall created by $\cO_X$. 

Since $\chi(\cO_X,F_2(E)) > 0$, Lemma~\ref{lem-dim} implies that a general $F_2$ remains stable along $\ell_1$ with respect to $\cO_X[1]$, and thus $\hom(F_2(E), \cO_X[1]) = 0$. Consequently, we have
\[
\hom(\cO_X, F_2(E)) = \chi(\cO_X, F_2(E)),
\]
where we use the fact that $\Hom(F_2(E), \cO_X) = 0$. Finally, we compute
\begin{align*}
k - \hom(\cO_X, F_2(E)) &= k - \chi(\cO_X, F_2(E)) \\
&= -s_0 - 2r_0 + 2(m_1 + m_2) \geq m_2,
\end{align*}
which establishes the inequality~\eqref{condition}.

\vspace{.2 cm}

\textbf{Case 2.2.} Now assume $\bigl\langle v_2(E) , v(\cO_X(-E)) \bigr\rangle > 0$. Recall that $\ch_2\bigl(F_2(E)\bigr) >0$ by our assumption, hence for all $m \leq -2$ we have   
\begin{align}\label{2-e}
    \bigl\langle v_2(E) , v(\cO_X(mE)) \bigr\rangle = & \ mk -2r_0+ 2m_1 +2m_2-\ch_2\bigl(F_2(E)\bigr) \nonumber\\
    < & \ -2k -2r_0 +2(m_1+m_2) \overset{\eqref{assumptionineq}}{\leq} 0.
\end{align}
It follows that a general $F_2(E)\in\cM_{H_\epsilon}(v_2(E))$ remains stable up to the wall that $\cO_X(-E)[1]$ is making, and then gets destabilized via a short exact sequence
\begin{equation}\label{destF2E}
    F' \longrightarrow F_2(E) \longrightarrow \cO_X(-E)^{\oplus h}[1]
\end{equation}
in $\Coh^0(X)$, where $h=\chi\bigl(F_2(E),\cO_X(-E)[1]\bigr)$ and the object $F'$ is $\sigma_{b',w'}$-stable for $(b',w')$ in the wall $\cW\bigl(F_2(E),\cO_X(-E)[1]\bigr)$. 

Conversely, given an arbitrary $\sigma_{b', w'}$-stable object $G'$ of Mukai vector $v(F')$ and an element $V\in \Gr\bigl(h,\Hom(\cO_X(-E),G'\bigr)$, the object $G$ sitting in an extension
\[
G'\longrightarrow G \longrightarrow \cO_X(-E)^{\oplus h}[1]
\]
is stable just above $\cW(F_2(E),\cO_X(-E)[1])$ by Lemma \ref{lem-stability above the wall}. For this reason, we may assume that the object $F'$ appearing in \eqref{destF2E} is a general $\sigma_{b', w'}$-stable object. Since
 \begin{align*}
      \bigl\langle v(F'), v(\cO_X) \bigr\rangle = \bigl\langle v_2(E), v(\cO_X(-2E))[1] \bigr\rangle \overset{\eqref{2-e}}{\geq } 0,
 \end{align*}
we have $\hom(\cO_X,F') =0$ by the generality of $F'$. Since also $\hom\bigl(\cO_X,\cO_X(-E)[1]\bigr)=0$, then $\hom(\cO_X,F_2(E)) = 0$, hence \eqref{condition} holds and the proof is complete.

\end{proof}

\vskip 10pt

\section{Application to Brill-Noether loci on curves}\label{section-application}
Fix an elliptic $K3$ surface $X$ of degree $k\geq  2$ as described in Proposition \ref{thm-k3 surface}. In this section, we apply the results of Section \ref{stratification} to study the Brill--Noether loci
\[
V^r_d(C) :=W^r_d(C)\setminus W^{r+1}_d(C)= \bigl\{ L \in \mbox{Pic}^d(C) : h^0(C, L) = r+1 \bigr\}
\]
for smooth curves $C \in |H|$ and $d, r\in \mathbb N$ with $d \leq g - 1$. 

\vskip 4pt

To this end, consider the moduli space $\mathcal{M}_{H_\epsilon}(v)$ with Mukai vector $v = (0, H, 1 + d - g)$, where $\epsilon > 0$ is sufficiently small, and define the locus
\[
V^r_{H_\epsilon}(v) := \left\{ F \in \mathcal{M}_{H_\epsilon}(v) : h^0(X, F) = r+1 \right\}.
\]

For an integral curve $C \in |H|$, the locus $V^r_d(C)$ is identified with the fiber over $[C]\in |H|$ of the natural support map
\[
V^r_{H_\epsilon}(v) \hookrightarrow \mathcal{M}_{H_\epsilon}(v) \longrightarrow |H|.
\]
This identification is independent of $\epsilon$. Through our analysis of $V^r_{H_\epsilon}(v)$, we derive strong implications for the Brill--Noether theory of curves in $|H|$.

\vskip 4pt

The following is the main result of this section, building upon Section \ref{stratification}:

\vspace{0.8mm}

\begin{Thm}\label{relative:HBNthm}
Assume $d\leq g-1$. Then for $v := (0, H, 1+d-g)$ we have:
\begin{enumerate}
    \item\label{relative:HBNthm-a} There exists $\epsilon(v,r)>0$ such that for $0<\epsilon<\epsilon(v,r)$, one has
    \begin{equation*}
        V^r_{H_\epsilon} (v) \subseteq \bigcup_{\overline{e} \in I } \mathcal{M}_{H_\epsilon}(v, \overline{e} ),
    \end{equation*}
    where $I$ is the finite set of stability types $\overline{e} =\bigl((e_i, m_i)\bigr)_{i=1}^p$ with $p \geq 1$, such that $\sum_{i=1}^p m_i\leq r+1\leq \sum_{i=1}^pm_i(e_i+1)$ and $m_1(e_1+1)\leq r+1$.

    \vspace{2mm}

 \item\label{relative:HBNthm-b} For each $\overline{e} =\bigl((e_i, m_i)\bigr)_{i=1}^p$ in $I$, define $\ell_{\overline{e}}:=r+1-\sum_{i=1}^p m_i$. Then:
 \vspace{0.5mm}
    \begin{enumerate}
        \item\label{relative:HBNthm-b1} If $\mathcal{M}_{H_\epsilon}(v, \overline{e} )\neq \emptyset$, then  $\rho(g, r - \ell_{\overline{e}}, d) - \ell_{\overline{e}} k\geq0$. Moreover,  $\mathcal{M}_{H_\epsilon}(v, \overline{e} )$  is irreducible, smooth, and  of dimension at most $g + \rho(g, r - \ell_{\overline{e}}, d) - \ell_{\overline{e}} k$.

        \vspace{0.5mm}
        
        \item\label{relative:HBNthm-b2} If $\overline{e}$ is balanced, $\rho(g, r - \ell_{\overline{e}}, d) - \ell_{\overline{e}} k\geq0$ and $r+1-k\leq \ell_{\overline{e}}$, then $\mathcal{M}_{H_\epsilon}(v, \overline{e} )$ is non-empty of dimension $g + \rho(g, r - \ell_{\overline{e}}, d) - \ell_{\overline{e}} k$. 
    \end{enumerate}
    \end{enumerate}
    \end{Thm}
\begin{proof}
We first assume that $d < g - 1$; the case $d = g - 1$ will be treated separately in Subsection \ref{subsec: g-1}. In this case, assertion \eqref{relative:HBNthm-a} is a  consequence of Theorem \ref{thm.splitting}, applied to a stability condition $\sigma_{0,w_0}$ lying in the Gieseker chamber for the Mukai vector $v$.

We now prove \eqref{relative:HBNthm-b}. Fix $\overline{e} =\bigl((e_i, m_i)\bigr)_{i=1}^p$ in $I$, such that $\cM_{H_\epsilon}(v, \overline{e} )$ is non-empty. According to Remark \ref{obs stabtype}, the Mukai vector
    \[
    \bigl(-m_1-\cdots-m_p,\ H-(m_1e_1+\cdots+m_pe_p)E,\ 1+d-g-m_1-\cdots-m_p\bigr)
    \]
    has square $\geq -2$, which reads as
    \[
2g-2(m_1e_1+\cdots+m_pe_p)k-2(m_1+\cdots+m_p)(g-d-1+m_1+\cdots+m_p)\geq0.
\]
It follows that
\begin{align*}
    0 &\leq  g-(m_1e_1+\cdots+m_pe_p)k-(m_1+\cdots+m_p)(g-d-1+m_1+\cdots+m_p) \leq \\
    & \leq g-\bigl(r+1-(m_1+\cdots+m_p)\bigr)k-(m_1+\cdots+m_p)(g-d-1+m_1+\cdots+m_p)=\\
&= g-\ell_{\overline{e}}k-(r+1-\ell_{\overline{e}})(g-d+r-\ell_{\overline{e}})=\rho(g,r-\ell_{\overline{e}},d)-\ell_{\overline{e}}k,
\end{align*}
where the second inequality follows from the condition $r+1\leq \sum_{i=1}^pm_i(e_i+1)$. Hence the inequality $\rho(g,r-\ell_{\overline{e}},d)-\ell_{\overline{e}}k\geq0$ holds. Furthermore, by Theorem \ref{Prop-scheme structure}, $\cM_{H_\epsilon}(v,\overline{e})$ is irreducible, smooth and quasi-projective of dimension
\begin{align*}
    \left(v-\sum_{i=1}^p m_i(1,e_iE,1)\right)^2+2+\sum_{j=1}^p m_j\left(\langle v-\sum_{i=1}^j m_i(1,e_iE,1),(1,e_jE,1)\rangle-m_j\right)=\\
    =2g-(m_1e_1+\cdots+m_pe_p)k-(m_1+...+m_p)(g-d-1)-(m_1+\cdots+m_p)^2.
\end{align*}
Now using the inequality  $r+1\leq \sum_{i=1}^pm_i(e_i+1)$, this dimension is at most
\begin{align*}
g+g-(r+1-(m_1+\cdots+m_p))k-(m_1+\cdots+m_p)(g-d+m_1+\cdots+m_p-1)=\\
=g+g-(r-\ell_{\overline{e}}+1)(g-d+r-\ell_{\overline{e}})-\ell_{\overline{e}}k=g+\rho(g,r-\ell_{\overline{e}},d)-\ell_{\overline{e}}k,
\end{align*}
which proves the first part of \eqref{relative:HBNthm-b}. Finally, the second part  is an immediate application of Theorem \ref{thm-nonempty} for $\sigma_{0,w_0}$ in the Gieseker chamber of $v$. 
\end{proof}

\vskip 1pt

As a direct consequence, we obtain the following result:

\begin{Cor}\label{HBNupperbound}
For $d \leq g-1$, set
        $\rho_k(g,r,d):=\mathrm{max}_{\ell=0, \ldots,r} \bigl\{\rho(g,r-\ell,d)-\ell k\bigr\}$. Then:
        
    \begin{enumerate}
        \item\label{HBNupperbound:a} If $\rho_k(g,r,d)<0$, then $W^r_d(C)=\varnothing$ for every integral curve $C\in|H|$.
        \item\label{HBNupperbound:b} If $C\in|H|$ is a general curve, then $\dim W^r_d(C)\leq \rho_k(g,r,d)$.
    \end{enumerate}
\end{Cor}
\begin{proof}
Note that $\ell_{\overline{e}}\in\{0,...,r\}$ for every stability type $\overline{e}\in I$, since $\sum_{i=1}^p m_i\leq r+1$. If $\rho_k(g,r,d)<0$, then it follows from Theorem \ref{relative:HBNthm}.\eqref{relative:HBNthm-b} that $\cM_{H_\epsilon}(v,\overline{e})=\varnothing$ for all $\overline{e}\in I$. Therefore $V^r_{H_\epsilon}(v)=\varnothing$ by Theorem \ref{relative:HBNthm}.\eqref{relative:HBNthm-a}, which implies $V^r_{d}(C)=\varnothing$ for every integral curve $C\in |H|$. Since $\rho_k(g,r,d)$ is decreasing as a function of $r$, we also have $V^{r'}_d(C)=\varnothing$ for all $r'\geq r$, namely $W^r_d(C)=\emptyset$ which proves \eqref{HBNupperbound:a}.

\vskip 2pt

On the other hand, according to Theorem \ref{relative:HBNthm} we have $\dim V^r_{H_\epsilon}(v)\leq g+\rho_k(g,r,d)$, and hence by considering the support map
\[
V^r_{H_\epsilon}(C)\hookrightarrow \cM_{H_\epsilon}(v)\longrightarrow |H|
\]
we deduce $\dim V^r_d(C)\leq \rho_k(g,r,d)$ for a general $C\in|H|$. Since $\rho_k(g,r,d)$ is decreasing in $r$, it follows that $\dim W^r_d(C)\leq \rho_k(g,r,d)$ as well, which proves \eqref{HBNupperbound:b}.
\end{proof}

\vspace{1mm}

\begin{Rem}\label{general-linearsystem}
Take $a\in\bZ_{\geq0}$ such that $g':=g-ak\geq1$. Thanks to the results of Section \ref{stratification}, 
Theorem \ref{relative:HBNthm} can be similarly proven for the Mukai vector $v=(0,H-aE,1+d-g')$ if $d\leq g'-1$.  As a consequence, Corollary \ref{HBNupperbound} is also valid for the genus $g'=g-ak$ curves in the linear system $|H-aE|$. Whereas here we take $a=0$ to simplify the statements, this more general version will be required in Section \ref{sec:dominance}.
\end{Rem}

\vspace{1mm}

\begin{Ex}\label{Example:LineBundles}
Set $v:=(0,H,1+d-g)$ for $d\leq g-1$ and $r\geq0$, and let $\ell\in\bZ$ satisfy $\max\{0,r+1-k\}\leq \ell \leq r$. Consider the integers $e,m_1\geq0$ and $m_2>0$ satisfying
\[
r+1=m_1(e+2)+m_2(e+1),\;\;\;r+1-\ell=m_1+m_2.
\]If $\rho(g,r-\ell,d)-\ell k\geq0$, then Theorem \ref{relative:HBNthm}.\eqref{relative:HBNthm-b} asserts the non-emptiness of the moduli space 
$\cM_{H_\epsilon}(v,\overline{e})$ for the balanced stability type $\overline{e}:=\bigl((e+1,m_1),(e,m_2)\bigr)$. We spell out in the geometric language of kernel vector bundles the meaning of a sheaf $L\in\cM_{H_\epsilon}(v)$ having stability type $\overline{e}$.

\vskip 2pt

Since $\Ext^1\bigl(\cO_X(eE),\cO_X((e+1)E)\bigr)=0$, any $L\in\cM_{H_\epsilon}(v,\overline{e})$ fits in an exact sequence
\[
0\longrightarrow \cO_X((e+1)E)^{\oplus m_1}\oplus\cO_X(eE)^{\oplus m_2}\longrightarrow L\longrightarrow F_2\longrightarrow 0
\]
in $\Coh^b(X)$ ($-1\ll b<0$), where $F_2$ is stable along the numerical wall $\cW(F_2,\cO_X)$. More precisely, $F_2$ is a complex with two cohomologies: $\cH^{-1}(F_2)$ (resp.~$\cH^{0}(F_2)$) arises as the kernel (resp.~the cokernel) of the evaluation map
\begin{small}
\begin{equation}\label{evaluationmap}
H^0\bigl(L(-(e+1)E)\bigr)\otimes \cO_X((e+1)E) \bigoplus \frac{H^0(L(-eE))}{H^0(\cO_X(E))\otimes H^0(L(-(e+1)E))}\otimes \cO_X(eE)\overset{\mathrm{ev}}{\longrightarrow} L.  
\end{equation}
\end{small}

If $L \in \cM_{H_\epsilon}(v,\overline{e})$ is general, then by Theorem \ref{Giesekerchamber}.\eqref{GiesekerchamberB} we can write 
\[
F_2=R\mathcal{H}om(G,\cO_X)[1],
\]
where $G$ is general in the moduli space $\cM_{H_\epsilon}\bigl(r-\ell+1,H-\ell E,g-d+r-\ell\bigr)$. If moreover $\ell<r$, then such a general $G$ is locally free: this implies $\cH^0(F_2)=0$ (in other words, the evaluation \eqref{evaluationmap} is surjective). When $\ell=0$ (so that $m_1=e=0$ and $m_2=r+1$) and $L$ is supported on a smooth curve, $G$ is simply the Lazarsfeld-Mukai bundle \cite{lazarsfeld-BNP} of $L$.
\end{Ex}

\vspace{2mm}

\subsection{The case of degree $g-1$}\label{subsec: g-1}
Throughout this subsection we fix $a\in\bZ_{\geq0}$ such that $g':=g-ak\geq1$. We now introduce a notion of stability type for stable objects of Mukai vector \( v = (0, H-aE, 0) \). This case requires an  \emph{ad hoc} treatment (it was skipped in Definition \ref{Def.splitting}), as all the numerical walls $\cW(v,\cO_X(eE))$, where $e\geq0$, coalesce into the horizontal line $w=0$\footnote{This also justifies the need to incorporate a small ball around the origin in the definition \eqref{U} of the region $U_\epsilon$ of Bridgeland stability conditions.}. Nevertheless, we find analogues of the results in Section \ref{stratification} for this modified definition; using these analogues, the proof of Theorem \ref{relative:HBNthm} for $d<g-1$ extends naturally to the case $d=g-1$.

\begin{Def}\label{def:g-1 stabtype}
   A stable object $F_0 \in \mathcal{M}_{H_\epsilon}^{\mathrm{st}}(0, H-aE, 0)$ is said to be \emph{of stability type} $\overline{e} = \bigl((e_i, m_i)\bigr)_{i=1}^p\subset \bZ_{\geq0}\times\bZ_{>0}$, where $p\geq0$ and $e_1 > e_2 > \dots > e_p \geq 0$, if for all $b\in(-\delta_\epsilon,0)$ there exist short exact sequences in $\mathrm{Coh}^b(X)$
    \begin{align*}
        0 &\longrightarrow \Hom(\mathcal{O}_X(e_1E),F_0)\otimes\mathcal{O}_X(e_1E) \xrightarrow{\mathrm{ev}_1} F_0 \longrightarrow F_1\longrightarrow 0, \\
        0 &\longrightarrow \Hom(\mathcal{O}_X(e_2E),F_1)\otimes\mathcal{O}_X(e_2E) \xrightarrow{\mathrm{ev}_2} F_1 \longrightarrow F_2\longrightarrow 0, \\
        &\vdots \\
        0 &\longrightarrow \Hom(\mathcal{O}_X(e_pE),F_{p-1})\otimes\mathcal{O}_X(e_pE)\xrightarrow{\mathrm{ev}_p} F_{p-1} \longrightarrow F_p\longrightarrow 0,
    \end{align*}
    such that for all $i=0, \ldots, p$:
    \begin{enumerate}
        \item\label{def:g-1 stabtype a} The object $F_i$ is $\sigma_{b,0}$-semistable.
        \item\label{def:g-1 stabtype b} $\Hom(\mathcal{O}_X(-uE)[1],F_i)=0$ for all $u\geq1$, and $m_i:= \hom(\mathcal{O}_X(e_iE), F_{i-1})$.
        \item\label{def:g-1 stabtype c} If $i \neq 0$, then $\Hom(\mathcal{O}_X(tE),F_i)=0$ for every $t\geq e_i$. Furthermore, the vanishing $\Hom(\mathcal{O}_X(tE),F_p)=0$ holds for every $t\geq0$.
    \end{enumerate} 
\end{Def}

Note that $\mathcal{M}_{H_\epsilon}(0, H - aE, 0)$ may contain strictly Gieseker-semistable sheaves for any $\epsilon > 0$; however, we restrict our attention to stable sheaves only. We begin by showing the following:

\begin{Lem}\label{stability of Qp}
If $\epsilon>0$ is small enough (depending on $a$ and the list $\bigl((e_i, m_i)\bigr)_{i=1}^p$), then all quotients $F_i$ in Definition \ref{def:g-1 stabtype} are $\sigma_{b, w}$-stable for every $b \in (-\delta, 0)$ and $0<w \ll 1$.
\end{Lem}
\begin{proof}
For $i=0$ the statement is clear, as $\sigma_{b,w}$ lies in the Gieseker chamber of $v$ by Proposition \ref{prop-main-wallcrossing}.
For $i\geq1$, assume $Q_1 \rightarrow F_i \rightarrow Q_2$ is a short exact sequence destabilizing $F_i$ just above the horizontal wall $w=0$; we may also assume that $Q_1$ is stable there.  
Write $\ch(Q_1) = (-r, tH-qE, 0)$. Since $\nu_{b, w}(Q_1) > \nu_{b,w}(F_i)>0$ for $w\gg0$, we have $r > 0$. Moreover, if we pick $\epsilon < \epsilon_{v(F_i)}$, then by Proposition \ref{prop-main-wallcrossing}, $t\in\{0,1\}$.

\vskip 2pt

If $t=0$, then stability of $Q_1$ implies $Q_1 \cong \cO_X(-uE)[1]$ for some $u\geq1$, according to Lemma \ref{lem-eE}. This is not possible due to the vanishing $\hom(\cO_X(-uE)[1],F_i)=0$.  

\vskip 2pt

If $t=1$, then $\ch(Q_2) = \ch(F_i)-\ch(Q_1)
    = \big(r-\sum_{j=1}^i m_j,  (q-a- \sum_{j=1}^i m_je_j)E , \ 0 \big)$.
Since we have surjections $F_0 \twoheadrightarrow F_i \twoheadrightarrow Q_2$ in $\Coh^b(X)$ and $F_0$ is $\sigma_{b,w}$-stable, we find $0=\nu_{b,w}(F_0)<\nu_{b,w}(Q_2)$ for all $w>0$. This implies $\ch_0(Q_2) <0$, i.e. $0 < r<\sum_{j=1}^i m_j$. But then if we choose $\epsilon$ small enough, we get  
\begin{equation*}
     \frac{-\ch_0(F_i)}{H_\epsilon\cdot\ch_1(F_i)}=\frac{\sum m_j}{\left(H-(a+\sum m_je_j) E\right)\cdot H_\epsilon}  < \frac{\sum m_j -r}{ (q -a-\sum m_je_j)E\cdot H_\epsilon}=\frac{-\ch_0(Q_2)}{H_\epsilon\cdot\ch_1(Q_2)}
\end{equation*}
which yields $\nu_{0, w}(F_i) < \nu_{0, w}(Q_2)$ for all $w\gg0$, a contradiction. 
\end{proof}

Lemma \ref{stability of Qp} is the key step in establishing the analogue of Theorem \ref{Prop-scheme structure}.

\begin{Thm}\label{g-1:scheme-structure}
Fix $v=(0,H-aE,0)$ and any stability type $\overline{e}=\bigl(e_i, m_i)\bigr)_{i=1}^{p}$ where $p\geq0$. If $\epsilon$ is sufficiently small, the subset
    \[
    \cM_{H_\epsilon}(v,\overline{e}):=\bigl\{F_0\in\cM_{H_\epsilon}(v): \;\text{$F_0$ is of stability type $\overline{e}$}\bigr\}
    \]
admits a natural scheme structure as a locally closed subscheme of $\cM_{H_\epsilon}(v)$. 
Moreover, if $\cM_{H_\epsilon}(v,\overline{e})$ is non-empty, then it is smooth and irreducible of dimension
    \[
    \left(v-\sum_{i=1}^p m_i(1,e_iE,1)\right)^2+2+\sum_{j=1}^p m_j\left(\bigl\langle v-\sum_{i=1}^j m_i(1,e_iE,1),(1,e_jE,1)\bigr\rangle-m_j\right).
    \]  
\end{Thm}

\vspace{3mm}

More precisely, set $v_p:=v-\sum_{i=1}^p m_i(1,e_iE,1)$ and $\sigma:=\sigma_{b,w}$ for $b\in(-\delta_\epsilon,0)$, $0<w\ll1$. If $U\subseteq \cM_{\sigma}(v_p)$ is the (possibly empty) open subset of objects $F_p$ with $\hom(\cO_X(-E)[1],F_p)=0$ and $\hom(\cO_X,F_p)=0$, then $\cM_{H_\epsilon}(v,\overline{e})$ is an open subset of an iterated Grassmannian bundle over $U$.

\vskip 2pt

Our next result is an analogue of Theorem \ref{thm.splitting}. 

\vspace{1mm}

\begin{Thm}\label{Thm-splitting-g-1}
Let $v=(0,H-aE,0)$. There exists $\epsilon(a)>0$  such that, if $0<\epsilon<\epsilon(a)$, then every $F_0\in\cM_{H_\epsilon}(v)$ has a stability type $\overline{e} = \bigl((e_i, m_i)\bigr)_{i=1}^p$. If $h^0(X, F_0)=r+1\geq1$, then moreover $p\geq1$ and the following inequalities hold:
\[
\sum_{i=1}^p m_i\leq r+1\leq \sum_{i=1}^pm_i(e_i+1),\;\;\; m_1(e_1+1)\leq r+1.
\]
\end{Thm}

\vspace{1mm}

The key step to prove Theorem \ref{Thm-splitting-g-1} is the following injectivity result:

\vspace{1mm}

\begin{Prop}\label{prop.composition}
Let $b\in\bR$ and $n\in\mathbb N$. If $Q\in\Coh^b(X)$ is an object satisfying $\Hom(\cO_X(pE), Q) = 0$ for all $p > n$, then the map
    \begin{align*}
        \varphi \colon \Hom(\cO_X(nE),Q)\otimes\Hom(\cO_X(mE),\cO_X(nE)) & \longrightarrow \Hom(\cO_X(mE),Q)\\
        f\otimes g & \longmapsto f\circ g
    \end{align*}
is injective for all $m\leq n$.
\end{Prop}
\begin{proof}
     We may assume that $Q$ is not of the form $\cO_X(qE)$ for some $q\in\bZ$; otherwise, $q\leq n$ and the claim is trivial. Consider the identification
    \[
    \Hom(\cO_X(mE),\cO_X(nE))\cong H^0(X, \cO_X((n-m)E))\cong \mathrm{Sym}^{n_0}H^0(X, \cO_X(E))
    \]
    provided by Lemma \ref{multpencil}, where $n_0:= n-m$. Let $s,t$ be a basis of $\Hom(\cO_X(-E),\cO_X)$, and let $J, F\in|E|$ denote the corresponding (disjoint) curves in the elliptic pencil. For any $i>0$, by considering for $\ell \gg 0$ the exact triangles 
\[
\cO_X(\ell E) \longrightarrow \cO_{iJ} \longrightarrow \cO_X((\ell-i)E)[1], \qquad 
\cO_X(\ell E) \longrightarrow \cO_{iF} \longrightarrow \cO_X((\ell-i)E)[1]
\] 
the vanishings $\Hom\bigl(\cO_{X}((\ell-i)E)[1], Q\bigr) = 0 = \Hom\bigl(\cO_X(\ell E), Q\bigr)$ imply that  
      \begin{equation}\label{vanish}
        \Hom(\cO_{iJ}, Q) = 0, \qquad \Hom(\cO_{iJ}, Q) = 0. 
    \end{equation}
Hence the maps 
   \begin{align}\label{injective}
       \Hom(\cO_X((i-1)E),Q) \xrightarrow{\circ s^i} \Hom(\cO_X(-E),Q),\\
       \Hom(\cO_X((i-1)E),Q) \xrightarrow{\circ t^i} \Hom(\cO_X(-E),Q)\nonumber
   \end{align}
are injective for any $i>0$.

   \bigskip
    The compositions $s^{n_0},s^{n_0-1}t,\ldots, st^{n_0-1},t^{n_0}$ form a basis of $\Hom\bigl(\cO_X(mE),\cO_X(nE)\bigr)$. Assume for the sake of a contradiction that $\ker(\varphi)\neq0$, namely we have a nonzero tensor
    \[
    \delta_0\otimes t^{n_0}+\delta_1\otimes t^{n_0-1}s+\cdots+\delta_{n_0}\otimes s^{n_0}\;\;(\delta_j\in\Hom(\cO_X(nE),Q))
    \]
    in the kernel. If $i\in \{0,\ldots,n_0-1\}$ is minimal such that $\delta_i \neq 0$, then we have 
    $$
    (\delta_i\circ t^{n_0-i}+\cdots+\delta_{n_0}\circ s^{n_0-i})\circ s^i =0
    $$
    and hence $\delta_i\circ t^{n_0-i}+\cdots+\delta_{n_0}\circ s^{n_0-i}=0$, by the injectivity of the maps in \eqref{injective}. Thus
    $$\delta_i\circ t^{n_0-i}=(-\delta_{i+1}\circ t^{n_0-i-1}-\cdots-\delta_{n_0}\circ s^{n_0-i})\circ s. $$
    Note that since $\delta_i \neq 0$, we have $\delta_i\circ t^{n_0-i} \neq 0$ as well.

\vskip 4pt
  The goal is to show 
  \begin{equation}\label{goal}
      \delta_i=\delta_i'\circ s
  \end{equation}
  for some nonzero $\delta_i'\in\Hom(\cO_X((n+1)E),Q)$, which contradicts our assumption. Consider the two identical morphisms
    \begin{align*}
        \cO_X((m+i)E)\overset{t^{n_0-i}}{\longrightarrow}\cO_X(nE)\overset{\delta_i}{\longrightarrow}Q\\
        \cO_X((m+i)E)\overset{s}{\longrightarrow}\cO_X((m+i+1)E)\overset{-\delta_{i+1}\circ t^{n_0-i-1}-\cdots-\delta_{n_0}\circ s^{n_0-i}}{\longrightarrow}Q
    \end{align*}
    By the octahedral axiom, we have distinguished triangles
    \begin{align*}
    \cO_{(n_0-i)F}\longrightarrow \mathrm{cone}(\delta_i\circ t^{n_0-i})\longrightarrow \mathrm{cone}(\delta_i)\\
    \cO_{J}\longrightarrow \mathrm{cone}(\delta_i\circ t^{{n_0-i}})\longrightarrow \mathrm{cone}(-\delta_{i+1}\circ t^{n_0-i-1}-\cdots-\delta_{n_0}\circ s^{n_0-i}). 
    \end{align*}
    Note that none of the above cones is zero, as we initially assumed that $Q \neq \cO_X(qE)$ for all $q \in \mathbb{Z}$. Since $\Hom(\cO_J,\cO_{(n_0-i)F})=0$ (they are sheaves with disjoint supports), we obtain a nonzero map
$
  \psi \colon  \cO_G\longrightarrow \mathrm{cone}(\delta_i\circ t^{n_0-i})\longrightarrow \mathrm{cone}(\delta_i).
    $
 Finally, consider the following diagram whose rows are distinguished triangles:
\[
 \begin{xy}
\xymatrix{
\cO_J[-1]
\ar[rr]\ar[d]^{\psi} &&  \cO_X(nE)\ar[rr]^s\ar[d]^= &&  \cO_X((n+1)E)
\\
\mathrm{cone}(\delta_i)[-1]\ar[rr] &&\cO_X(nE)\ar[rr]^{\delta_i} && Q
}
\end{xy}
\]

    Note that the square on the left-hand side commutes (up to constant). Indeed:

    \begin{itemize}
        \item $\Hom\bigl(\cO_J[-1],\cO_X(nE)\bigr)\cong H^1(\cO_J(-nE))\cong H^1(J, \cO_J)\cong\mathbb{C}$
        \item The map $\cO_J[-1]\rightarrow \cO_X(nE)$ in the first row is nonzero.
        \item The composition $\cO_J[-1]\rightarrow \mathrm{cone}(\delta_i)[-1] \rightarrow \cO_X(nE)$ is nonzero. Otherwise, we would obtain a nonzero map $\cO_J\rightarrow Q$, which is not possible by \eqref{vanish}. 
    \end{itemize}
    Therefore, by the axiom TR3 of triangulated categories, we have $\delta_i=\delta_i'\circ s$ for some $\delta_i'\in\Hom(\cO_X((n+1)E),Q)$ as required in \eqref{goal}. This concludes the proof.
\end{proof}

\vskip 2pt

\begin{proof}[Proof of Theorem \ref{Thm-splitting-g-1}]
We only prove the existence of a stability type for $F_0$; the second statement (constraints on the stability type when $h^0(F_0)=r+1$) can be argued \emph{mutatis mutandis} as in the proof of Theorem \ref{thm.splitting}. 
\vskip 3pt

We construct the $F_i$'s inductively. Assuming $\epsilon<\epsilon_{v}$, Proposition \ref{prop-main-wallcrossing} guarantees that there is no actual wall for $F_0$ passing through the positive part of the vertical line $b=0$. Hence $F_0$ is $\sigma_{b,0}$-semistable and $\nu_{b,0}(F_0)=\nu_{b,0}(\cO_X)$ for $b \in (-\delta, 0)$. Furthermore, we have $\Hom(\cO_X(-uE)[1],F_0)=0$ for all $u\geq1$, since $F_0$ is a sheaf. 

For the induction step, let $i\geq0$ and assume  we have constructed $\sigma_{b,0}$-semistable objects $F_0,\ldots,F_i$ and integers $e_1>\cdots >e_i\geq0$ satisfying properties \eqref{def:g-1 stabtype b} and \eqref{def:g-1 stabtype c}. 
If $\Hom(\cO_X(tE),F_i)=0$ for all $t\geq0$, then we set $p:=i$ and finish the process (the stronger vanishing for $F_p$ in \eqref{def:g-1 stabtype c} will be satisfied). Otherwise, we define
\[
e_{i+1}:=\max\bigl\{t\geq0 : \Hom(\cO_X(tE),F_i)\neq0\bigr\}.
\]
This maximum is indeed attained: $e_{i+1}<e_i$ if $i\geq1$ thanks to \eqref{def:g-1 stabtype c}, whereas for $i=0$ we have $e_1\leq h^0(F_0)-1$. We claim that the evaluation map
\[
\Hom(\cO_X(e_{i+1}E),F_i)\otimes\cO_X(e_{i+1}E)\xrightarrow{\ev_{i+1}}F_i
\]
is injective in $\Coh^b(X)$ for $b\in(-\delta,0)$. Indeed, we have a commutative diagram
\[
 \begin{xy}
\xymatrix{
\Hom(\cO_X(e_{i+1}E), F_i)\otimes\Hom(\cO_X, \cO_X(e_{i+1}E)) \otimes \cO_X
\ar[rr]\ar[d]^\psi &&  \Hom(\cO_X,F_i)\otimes\cO_X\ar[d]
\\
\Hom(\cO_X(e_{i+1}E),F_i)\otimes\cO_X(e_{i+1}E)\ar[rr]^{\ev_{i+1}} && F_i
}
\end{xy}
\]
where the top horizontal arrow is injective in $\Coh^b(X)$ (by Proposition \ref{prop.composition}), and the two vertical arrows are injective as well (since $\cO_X$  is $\sigma_{b,0}$-stable by Lemma \ref{lem-ox}). Then 
\[
\ker(\ev_{i+1})\subseteq \mathrm{coker}(\psi)=\Hom\bigl(\cO_X(e_{i+1}E),F_i\bigr)\otimes\cO_X(-E)^{\oplus e_{i+1}}[1], 
\]
and thus $\ker(\ev_{i+1})=\cO_X(-E)^{\oplus t}[1]$ for some $t\geq0$, as $\cO_X(-E)[1]$ is $\sigma_{b,0}$-stable by Lemma \ref{lem-ox}. Since $\hom\bigl(\cO_X(-E)[1],\cO_X(e_{i+1}E)\bigr)=0$, it follows that $\ker(\ev_{i+1})=0$.

\vskip 1pt

We define $F_{i+1}:=\mathrm{coker}(\ev_{i+1})$. Clearly $F_{i+1}$ is $\sigma_{b,0}$-semistable, as it is the quotient of two $\sigma_{b,0}$-semistable objects of the same slope. This proves \eqref{def:g-1 stabtype a}. Moreover, since $\hom(\cO_X(-uE)[1],F_i)=0$ for any $u\geq1$, we know that $\Hom(\cO_X(-uE)[1],F_{i+1})$ equals the kernel of the natural map
\[
\Hom(\cO_X(e_{i+1}E),F_i)\otimes\Hom(\cO_X(-uE),\cO_X(e_{i+1}E))\longrightarrow\Hom(\cO_X(-uE),F_i).
\]
Therefore $\Hom(\cO_X(-uE)[1],F_{i+1})=0$ by Proposition \ref{prop.composition}, which proves \eqref{def:g-1 stabtype b}. Finally, by construction  $\Hom(\cO_X(e_{i+1}E),F_{i+1})=0$, which in virtue of Lemma \ref{lem-plus one} implies $\Hom(\cO_X(tE),F_{i+1})=0$ for all $t\geq e_{i+1}$. This proves \eqref{def:g-1 stabtype c} and concludes the proof.
\end{proof}

\vskip 3pt

\begin{Lem}\label{lem-plus one}
Let $b\in(-\delta,0)$, and let $Q \in \Coh^b(X)$ satisfy $\hom(\cO_X(-E)[1],Q)=0$ and $\hom(\cO_X(tE), Q) = 0$ for some $t\geq0$. Then 
$\hom(\cO_X((t+1)E), Q) = 0$.
\end{Lem}

\begin{proof}
Suppose by contradiction that $\hom\bigl(\cO_X((t+1)E), Q\bigr) \neq 0$. Then the exact triangle
\[
\cO_X(tE) \longrightarrow \cO_X((t+1)E) \longrightarrow \cO_J((t+1)E)\cong \cO_J
\]
for $J \in |E|$ implies that $\hom(\cO_J, Q) \neq 0$. Therefore, from the exact triangle 
\[
\cO_X(tE) \longrightarrow \cO_J \longrightarrow \cO_X((t-1)E)[1]
\]
we have $\hom(\cO_X((t-1)E)[1], Q) \neq 0$. This contradicts our hypothesis for $t=0$, whereas for $t \geq 1$ we know that $\cO_X((t-1)E) \in \Coh^b(X)$, hence $\hom(\cO_X((t-1)E)[1], Q)=0$ which, again, is a contradiction.  
\end{proof}


\vspace{2mm}

Finally, we establish the following non-emptiness result, paralleling  Theorem \ref{thm-nonempty}. 

\vspace{2mm}

\begin{Thm}\label{thm-existence-g-1} 
Let  $v = (0, H - aE, 0)$ and consider $\overline{e} := \bigl((e+1, m_1), (e, m_2)\bigr)$  
for \( e, m_1 \geq 0 \) and \( m_2 > 0 \) such that  
\[
    m_1 + m_2 < k, \mbox{ and }
\]
\[
    \Bigl(v-m_1\cdot v\bigl(\cO_X((e+1)E)\bigr)-m_2\cdot v\bigl(\cO_X(eE)\bigr)\Bigr)^2\geq -2.
\]
Then there exists \(\epsilon(a, \overline{e}) > 0\)  such that, if \(\epsilon < \epsilon(a, \overline{e})\), then   
$\mathcal{M}_{H_\epsilon}(v, \overline{e})$ is non-empty.
\end{Thm}
\begin{proof}
The proof goes along the same lines of Theorem \ref{thm-nonempty}; we explain which parts must be adapted. We fix the Mukai vector
\[
v_2 := v - m_1\cdot  v\bigl(\cO_X((e+1)E)\bigr) - m_2\cdot v\bigl(\cO_X(eE)\bigr)
\]
and a stability condition $\sigma_+:=\sigma_{b,w}$ for fixed $b\in(-\delta_\epsilon,0)$ and $0<w\ll1$. For every $F_2\in\cM_{\sigma_+}(v_2)$, comparison of $\sigma_+$-slopes immmediately yields $\Hom(F_2,\cO_X(uE))=0$ for all $u\geq0$, $\Hom(F_2,\cO_J)=0$ for every $J\in|E|$, and $\Hom(\cO_X(-E)[1],F_2)=0$.

\vskip 3pt

Furthermore, for a general $F_2\in \cM_{\sigma_+}(v_2)$, we have $\Hom(\cO_X,F_2)=0$ by a dimension count similar to that of Lemma \ref{lem-dim}, together with the $\sigma_{b,0}$-stability of $\cO_X$). Therefore, Lemma \ref{lem-plus one} implies that $\Hom(\cO_X(uE),F_2)=0$ for all $u \geq 1$. Also the short exact sequence in $\Coh^b(X)$
\[
    0\longrightarrow \cO_X\longrightarrow\cO_X(-(u+1)E)^{\oplus u}[1]\longrightarrow\cO_X(-uE)^{\oplus u+1}[1]\longrightarrow0,
\]
obtained by shifts and twists of \eqref{evalutationseq}, implies $\Hom\bigl(\cO_X(-(u+1)E)[1],F_2\bigr)=0$ for all $u\geq1$.

It follows from Proposition \ref{prop.general} that, if 
\begin{equation}\label{condition:g-1}
    0 < m_2 < k - \hom(\cO_X((e-1)E), F_2),
\end{equation}
then for a general $V_2\in \Gr\left(m_2,\Ext^1(F_2,\cO_X(eE))\right)$ the extension 
\[
0 \longrightarrow V_2^\vee\otimes\cO_X(eE) \longrightarrow F_1 \longrightarrow F_2 \longrightarrow 0 
\]
satisfies $\Hom(F_1, \cO_X(uE)) = 0$ for all $u \geq 0$. We have $\Hom(\cO_X(-uE)[1],F_1)=0$ for every $u\geq1$ as well.
The hard part is to prove $\sigma_+$-stability of $F_1$ under the genericity assumption on $V_2$ (note that Lemma \ref{lem-stability above the wall} cannot be applied):

\begin{Claim}\label{maps O_D}
$F_1$ is $\sigma_+$-stable.
\end{Claim}
\begin{proof}[Proof of the claim]
Assume $F_1$ is not $\sigma_+$-stable, and let  $Q_1 \rightarrow F_1 \rightarrow Q_2$ be a short exact sequence destabilizing $F_1$ with respect to $\sigma_+$, such that $Q_1$ is $\sigma_+$-stable. Arguing as in the proof of Lemma \ref{stability of Qp}, we obtain $\ch(Q_1)=(-r,H-qE,0)$ with $r>0$.

    \vspace{1mm}

    Therefore $\ch_1(Q_2)=(q-a)E$. As an application of Lemma \ref{lem-key}, any $\sigma_+$-stable factor of a HN factor of $Q_2$ has $\ch_1$ equal to a multiple of $E$ and $\ch_2=0$. Let $Q'$ be such a stable factor. If $\ch_0(Q')\neq0$, then by Lemma \ref{lem-eE} $Q'$ is (up to shift) a line bundle $\cO_X(uE)$. This implies $\ch_0(Q')\leq0$ for every stable factor; otherwise, the last stable factor defines a quotient $F_1\twoheadrightarrow Q_2 \twoheadrightarrow \cO_X(uE)$ in $\Coh^b(X)$ for some $u\geq0$, contradicting the vanishing $\hom(F_1,\cO_X(uE))=0$.

    \vspace{1mm}

Hence, \( \ch_0(Q_2) \leq 0 \). If \( \ch_0(Q_2) < 0 \),  by arguing as in the proof of Lemma~\ref{stability of Qp}, we obtain \( \nu_{0,w}(F_1) < \nu_{0,w}(Q_2) \) for  \( w \gg 0 \), which is a contradiction. Therefore, every \( \sigma_+ \)-stable factor \( Q' \) of \( Q_2 \) has Chern character \( (0, q'E, 0) \) for some \( q' > 0 \). In particular, by Corollary~\ref{cor:nowall}, \( Q_2 \) is a torsion sheaf. We explain how this leads to a contradiction.

        \vspace{1mm}

    If $H^0(X,Q_2)=0$, then $\Hom(\cO_X(eE),Q_2)=0$ as $\cO_X(eE)$ restricts trivially to the (scheme-theoretic) support of $Q_2$. Since $\Hom(F_2,Q_2)=0$ as well ($F_2$ is $\sigma_+$-stable with $\nu_{\sigma_+}(F_2)>0=\nu_{\sigma_+}(Q_2)$), it follows that $\Hom(F_1,Q_2)=0$, a contradiction. Suppose \( H^0(X, Q_2) \neq 0 \). Then, by Lemma \ref{multpencil}, all irreducible components of the (set-theoretic) support of \( Q_2 \) lie in the linear system \( |E| \). Pick one such component \( J \in |E| \) with \( H^0(Q_2|_J) \neq 0 \). Since \( \chi(Q_2|_J) = 0 \) (as \( Q_2 \) is Gieseker semistable), Serre duality on \( J \) implies
\[
0 \neq h^1(Q_2|_J) = \operatorname{Hom}_{\mathcal{O}_J}(Q_2|_J, \mathcal{O}_J) = \operatorname{Hom}_{\mathcal{O}_X}(Q_2, \mathcal{O}_J).
\]

    Therefore we have a surjection $F_1\twoheadrightarrow Q_2\twoheadrightarrow\cO_J$. Note that, as $\Hom(F_2,\cO_J)=0$, $\Hom(F_1,\cO_J)$ equals the kernel of the induced map
    \[
    V_2\otimes \Hom(\cO_X(eE),\cO_J)\longrightarrow\Ext^1(F_2,\cO_J).
    \]
    Since $\Hom(\cO_X(eE),\cO_J)\cong\bC$, this map is the restriction to $V_2$ of the natural map
    \[
    \psi_J:\Ext^1(F_2,\cO_X(eE))\longrightarrow\Ext^1(F_2,\cO_J)
    \]
    obtained by applying $\Ext^1(F_2,-)$ to the short exact sequence
    \[
    0\longrightarrow \cO_X((e-1)E)\longrightarrow \cO_X(eE)\longrightarrow \cO_J\longrightarrow 0.
    \]
    As $\hom(F_2,\cO_J)=0$ and $\ext^2(F_2,\cO_X(eE))=\hom(\cO_X(eE),F_2)=0$, we get the long exact sequence 
    \begin{align*}
0 &\to \Ext^1(F_2, \mathcal{O}_X((e-1)E)) 
   \to \Ext^1(F_2, \mathcal{O}_X((e-1)E)) \\
  &\xrightarrow{\psi_J} \Ext^1(F_2, \mathcal{O}_J) 
   \to \Ext^2(F_2, \mathcal{O}_X((e-1)E)) 
   \to 0.
\end{align*}
    Hence $\ker(\psi_J)$ has codimension $\ext^1(F_2,\cO_J)-\hom(\cO_X((e-1)E,F_2)$ in $\Ext^1(F_2,\cO_X(eE))$. A lower bound for this codimension (using again $\hom(F_2,\cO_J)=0$) is
    \[
    -\chi(F_2,\cO_J)-\hom(\cO_X((e-1)E,F_2)=k-\hom(\cO_X((e-1)E,F_2).
    \]
It follows from \eqref{condition:g-1} that a general choice of $V_2$ satisfies $V_2\cap \ker(\psi_J)=0$ for every $J\in|E|$; note that we require strict inequality in \eqref{condition:g-1}, as $J$ varies in a 1-dimensional family. Therefore $\Hom(F_1,\cO_J)=0$ for every $J\in|E|$, which is a contradiction and proves the Claim \ref{maps O_D}.
\end{proof}

Coming back to the proof of Theorem \ref{thm-existence-g-1}, assume \eqref{condition:g-1} holds. Applying Proposition \ref{prop.general} again, we find that for a general $V_1\in\Gr\left(m_1,\Ext^1(F_1,\cO_X(eE))\right)$ the extension 
\[
0 \longrightarrow V_1^\vee\otimes\cO_X((e+1)E)\longrightarrow F_0 \longrightarrow F_1 \longrightarrow 0 
\]
satisfies $\hom(F, \cO_X(uE)) = 0 = \hom(\cO_X(-uE)[1],F_0)$ for every $u \geq0$. Note that $F_0$ is $\sigma_{b,0}$-semistable (as an extension of $\sigma_{b,0}$-semistable objects with equal $\nu_{b,0}$-slope). It suffices to check that $F_0$ is $\sigma_+$-stable: since $\sigma_+$ lies in the Gieseker chamber for $v$, then we obtain $F_0\in\cM_{H_\epsilon}(v,\overline{e})$ as required.

\vskip 3pt

We first show that $F_0$ is $\sigma_+$-semistable. Indeed, if $Q_1 \rightarrow F_0 \rightarrow Q_2$ is a $\sigma_+$-destabilizing sequence with $Q_1$ $\sigma_+$-stable, arguing as in the proof of Lemma \ref{stability of Qp} again gives $\ch(Q_1)=(-r,H-qE,0)$ with $r>0$. But then $\ch_0(Q_2)=r>0$, which contradicts the vanishing $\hom(F_0,\cO_X(uE))=0$ for some $u\geq0$ as in the proof of Claim \ref{maps O_D}. Since $F_0$ is $\sigma_+$-semistable, we know that it is a Gieseker semistable sheaf of Chern character $(0,H-aE,0)$. If not $\sigma_+$-stable, $F_0$ must be supported on a reducible curve, and admits a semistable sheaf $Q_2$ of Chern character $(0,q'E,0)$ ($q'>0$) as a quotient. But arguing as in the proof of Claim \ref{maps O_D} we find $\Hom(F_0,Q_2)=0$ for the general choice of $V_1$, which shows $\sigma_+$-stability of $F_0$.

\vspace{1mm}

Therefore, the proof of Theorem \ref{thm-existence-g-1} is complete as long as the inequality \eqref{condition:g-1} holds. If $e\geq1$, this is trivial since $\hom(\cO_X\bigl((e-1)E),F_2\bigr)=0$. For $e=0$, this can be checked exactly as in the proof of Theorem \ref{thm-nonempty}, where one is led to Case 2.1 (since $\ch_2(F_2(E))=k>0$ and $\bigl\langle v_2(E),v(\cO_X(-E))\bigr\rangle=2(m_1+m_2-k)\leq0$).
\end{proof}

\begin{proof}[Proof of Theorem \ref{relative:HBNthm} when $d = g - 1$]
Combining Theorem \ref{g-1:scheme-structure}, Theorem \ref{Thm-splitting-g-1}, and Theorem \ref{thm-existence-g-1} via the same argument as in the case $d < g - 1$ implies the claim.
\end{proof}

\vspace{4mm}

\subsection{Stability type versus splitting type}\label{subsec:comparison}
The main feature of Theorem \ref{relative:HBNthm} is the stratification of $V^r_{H_\epsilon}(v)$ in terms of Bridgeland stability types. It is interesting to compare this invariant with the  \emph{splitting type} recalled in (\ref{eq:splitting}) introduced by H. Larson  \cite{larson:inv}.

\vskip 3pt

Take $a\geq 0$ and $g':=g-ak\geq1$. Recall that for any integral curve $C\in |H-aE|$, if $\pi\colon C\stackrel{i}\rightarrow X\rightarrow \mathbf{P}^1$ is the induced degree $k$ cover, then the collection $\overline{f}_L=\bigl((f_i,n_i)\bigr)_{i=1}^q$ is via (\ref{eq:splitting}) the splitting type of $L\in W^r_d(C)$.
Note that $f_1$ is the largest integer with $\Hom\bigl(\cO_{\mathbf{P}^1}(f_1),\pi_*L\bigr)=\Hom\bigl(\cO_X(f_1E),i_*L\bigr)\neq 0$, and $n_1=\hom(\cO_X(f_1E),i_*L)$. 

\vskip 3pt

A first expectation could be that the stability type of $i_*L$ coincides with the non-negative part $(\overline{f}_L)^{\geq0}\subseteq \overline{f}_L$ of the splitting type of $L$. This might be too coarse to hold in full generality, but it is almost true for balanced stability types.
More precisely, consider the stability type $\overline{e}=\bigl((e+1,m_1),(e,m_2)\bigr)$ with $m_2>0$. As usual, write $r:=m_1(e+2)+m_2(e+1)-1$ and $\ell:=r+1-m_1-m_2$. 

\vskip 3pt

\begin{Prop}\label{comparisontypes}
Let $C\in|H-aE|$ be integral and $L\in\overline{\Pic}^d(C)$ for $d\leq g'-1$. Then:
\begin{enumerate}
    \item\label{comparisontypes:a} If $i_*L$ has stability type $\overline{e}$, then $(\overline{f}_L)^{\geq0}=\overline{e}$.

    \item\label{comparisontypes:b} If \( (\overline{f}_L)^{\geq 0} = \overline{e} \), then \( i_*L \in \widetilde{\mathcal{M}}_{H_\epsilon}(v, \overline{e}) \), as defined in \eqref{union}. Equivalently, \( i_*L \) has stability type \( \overline{e} \cup \overline{e'} \) for some \( \overline{e'} \subset \mathbb{Z}_{\leq e-1} \times \mathbb{Z}_{>0} \).

    \item\label{comparisontypes:c} The loci $\bigl\{L\in \overline{\Pic}^d(C):\text{$(\overline{f}_L)^{\geq0}=\overline{e}$}\bigr\}$ and $\bigl\{L\in \overline{\Pic}^d(C):\text{$i_*L$ has stability type $\overline{e}$}\bigr\}$ have the same closure in $\overline{\mathrm{Pic}}^d(C)$.
\end{enumerate}
\end{Prop}
\begin{proof}
Write $v=(0,H-aE,1+d-g')$. If $i_*L\in\cM_{H_\epsilon}(v,\overline{e})$, then using the vanishing $\ext^1(\cO_X(eE),\cO_X((e+1)E))=0$ we have that $i_*L$ sits in a short exact sequence
    \[
        0\longrightarrow \cO_X(eE)^{ \oplus m_2}\oplus\cO_X((e+1)E)^{\oplus m_1} \longrightarrow i_* L\longrightarrow Q_2\longrightarrow 0
    \]
    in $\Coh^b(X)$ (for $-1\ll b<0$), where $Q_2$ is stable along the walls $\cW(Q_2,\cO_X)$ and $\cW(Q_2,\cO_X(eE))$. In particular $\hom(\cO_X,Q_2)=0=\hom(\cO_X(eE),Q_2)$ and therefore
    \begin{align}
    \hom(\cO_X,i_*L)=m_1(e+2)+m_2(e+1),\label{homOX}\\
    \hom(\cO_X(eE),i_*L)=m_2+2m_1.\label{homeE}
    \end{align}
    We already know that the first entry of $(\overline{f}_L)^{\geq0}$ is $(e+1,m_1)$; combining this with $\ext^1(\cO_X(eE),\cO_X((e+1)E))=0$, \eqref{homeE} results in the inclusion $\overline{e}\subset (\overline{f}_L)^{\geq0}$. Then it must be $\overline{e}= (\overline{f}_L)^{\geq0}$ thanks to \eqref{homOX}, which proves \eqref{comparisontypes:a}.

    \vskip 1pt

    If $L$ satisfies $(\overline{f}_L)^{\geq0}=\overline{e}$, then \eqref{homeE} holds and $i_* L$ sits in an extension
        \[
        0\longrightarrow \cO_X((e+1)E)^{\oplus m_1}\longrightarrow i_* L\longrightarrow Q_1\longrightarrow 0
        \]
        in $\Coh^0(X)$, with $Q_1$ stable along $\cW\bigl(\cO_X((e+1)E),i_*L\bigr)$. Now it follows from the vanishing $\ext^1(\cO_X(eE),\cO_X((e+1)E)=0$ that
        \[
        \hom(\cO_X(eE),Q_1)=\hom(\cO_X(eE),i_*L)-2m_1\overset{\eqref{homeE}}{=}m_2,
        \]
        and hence $Q_1$ sits in a sequence
        $
        0\rightarrow \cO_X(eE)^{\oplus m_2}\rightarrow Q_1\rightarrow Q_2\rightarrow 0
        $
        with $Q_2$ stable along $\cW(\cO_X(eE),Q_2)$. This proves \eqref{comparisontypes:b}. Finally, assertion \eqref{comparisontypes:c} follows from the inclusion $\widetilde{\cM}_{H_\epsilon}(v,\overline{e})\subseteq \overline{\cM_{H_\epsilon}(v,\overline{e})}$
    inside the moduli space $\cM_{H_\epsilon}(v)$, provided by Remark \ref{partial compact}.
\end{proof}

\vskip 2pt

\begin{Rem}
Regarding the comparison in Proposition \ref{comparisontypes}, we observe the following:
\vskip 3pt
 
\noindent  (a) It follows from Proposition \ref{comparisontypes} that, if $e=0$, that is, $\ell <\frac{r+1}{2}$, then the equality $\bigl\{L\in \overline{\Pic}^d(C):\;\text{$i_*L$ has stability type $\overline{e}$}\bigr\}=V^r_{d,\ell}(C)$ holds.

    \vspace{0.5mm}

\noindent (b) It is shown in Section \ref{sec:dominance}  that if $\max\{0,r+2-k\}\leq \ell\leq r$ and $\rho(g',r-\ell,d)-\ell k\geq0$, the support morphism  $\widetilde{\cM}_{H_\epsilon}(v,\overline{e})\rightarrow |H-aE|$
    is dominant. By generic smoothness and Proposition \ref{comparisontypes}.\eqref{comparisontypes:b}, the locus $V^r_{d,\ell}(C)$ is then smooth of dimension $\rho(g',r-\ell,d)-\ell k$ for a general $C\in|H-aE|$.

\vspace{0.5mm}

\noindent (c) In contrast, if $r+1-k\geq0$ and $\ell:=r+1-k$ satisfies $\rho(g',r-\ell,d)-\ell k\geq0$, then the image of the support morphism must be contained in the locus of reducible curves in $|H-aE|$. Otherwise, by Proposition \ref{comparisontypes} we would have a rank 1, torsion-free sheaf $L$ on an integral curve, with $\chi(L)\leq0$ and splitting type with $r+1-\ell=k$ non-negative entries, which is impossible.
\end{Rem}

\section{Non-emptiness of Hurwitz-Brill-Noether loci via $K3$ surfaces}\label{sec:dominance}

In this section we provide a new approach, using stability conditions on $K3$ surfaces, to the non-emptiness of the loci $V^r_{d, \ell}(C, A)$  (recall Definition \ref{HBNdegloci}), for a general element $[C, A]\in \mathcal{H}_{g,k}$.  Precisely, for every $\ell$ with $\mbox{max}\{r+2-k, 0\}\leq \ell \leq r$ satisfying inequality (\ref{eq:ineq11}), that is,  $$
\rho(g, r-\ell,d)-\ell k\geq 0,$$
we construct a component of $V^r_{d, \ell}(C, A)$ (and by passing to the closure a component of $W^r_{d, \ell}(C,A)$) having precisely this dimension. An immediate consequence is the existence theorem in Hurwitz-Brill-Noether theory.

\vskip 3pt

This is achieved by proving the result for curves on elliptic $K3$ surfaces. As usual, take the integers $e,m_1\geq0$, $m_2>0$ such that
\[
r+1=m_1(e+2)+m_2(e+1),\;\;r+1-\ell=m_1+m_2.
\]
For any $K3$ surface $X$ as in Proposition \ref{thm-k3 surface} and a general $C\in|H|$, we show that
\[
\Bigl\{L\in\Pic^d(C): \text{$i_*L$ has stability type $\overline{e}:=\{(e+1,m_1),(e,m_2)\}$}\Bigr\}
\]
is smooth of pure dimension $\rho(g,r-\ell,d)-\ell k$. 
In view of Proposition \ref{comparisontypes} we obtain that, for the pair $(C,A)=(C,\cO_C(E))\in \mathcal{H}_{g,k}$, the locus $V^r_{d, \ell}(C, A)$ has a component of dimension $\rho(g,r-\ell,d)-\ell k$, which as explained in Proposition \ref{prop:semicont}, is enough to derive the result  for a general $[C,A]\in \mathcal{H}_{g,k}$.

\vskip 6mm

\subsection{Curves on elliptic $K3$ surfaces}\label{proof-dominanceK3}
We fix, as throughout the paper, a degree $k$ elliptic $K3$ surface $X$ as in Proposition \ref{thm-k3 surface}. As explained above, we want to prove that the locus
\[
\Bigl\{L\in\Pic^d(C): \text{$i_*L$ has stability type $\overline{e}$}\Bigr\}
\]
is smooth of pure dimension $\rho(g,r-\ell,d)-\ell k$, for a general curve $C\in|H|$.
This amounts to prove that the natural support map
\[
\pi_{(g,d,r,\ell)}:\cM_{H_\epsilon}\bigl((0,H,1+d-g),\overline{e}\bigr)\longrightarrow |H|
\]
is dominant, since, as proven in Theorem \ref{relative:HBNthm}, the moduli space $\cM_{H_\epsilon}\bigl((0,H,1+d-g),\overline{e}\bigr)$ is smooth and irreducible of dimension $g+\rho(g,r-\ell,d)-\ell k$.

\vskip 3pt

Our argument is inductive in nature and relies on a reduction to the case $\ell=0$. Indeed, it suffices to exhibit a curve $Y\in|H|$ such that the fiber $\pi_{(g,d,r,\ell)}^{-1}\bigl([Y]\bigr)$ has an irreducible component of the correct dimension $\rho(g,r-\ell,d)-\ell k$. We will do this by picking $Y$ to be reducible, so that the information is extracted from a curve of lower genus on $X$.

In order to perform induction on the tuple $(g,d,r,\ell)$, let us spell out the content of Theorem \ref{relative:HBNthm} and Remark \ref{general-linearsystem} for the linear systems $|H-aE|$, $a\geq0$. 
Let $(g',d',r',\ell')\in\mathbb N^4$ be a tuple satisfying 
\begin{equation}\label{cond: tuple}
    \begin{aligned}
g':=g-ak\geq 1\text{ for some $a\geq0$, $\;\;d'\leq g'-1$,}\\
\mathrm{max}\{r'+2-k, 0\}\leq \ell' \leq r',\;\;\rho(g',r'-\ell',d')-\ell'k\geq0.
\end{aligned}
\end{equation}
Taking the unique $e',m_1'\geq0$ and $m_2'>0$ such that $r'+1=m_1'(e'+2)+m_2'(e'+1)$ and $r'+1-\ell'=m_1'+m_2'$, 
then for $\overline{e'}:=\{(e'+1,m_1),(e',m_2')\}$ the moduli space 
\[
\cM_{H_\epsilon}\left((0,H-aE,1+d'-g'),\overline{e'}\right)
\]
is smooth and irreducible of dimension $g'+\rho(g',r'-\ell',d')-\ell' k$. We denote by 
\[
\pi_{(g',d',r',\ell')}\colon \cM_{H_\epsilon}\left((0,H-aE,1+d'-g'),\overline{e'}\right)\longrightarrow |H-aE|
\]
the corresponding support map. Our main reduction is the following:

\vskip 5mm

\begin{Prop}\label{reduction:l=0}
    If $\ell\geq1$ and the map $\pi_{(g-k,d-k,r-1,\ell-1)}$ is dominant, then the map $\pi_{(g,d,r,\ell)}$ is dominant.
\end{Prop}
\begin{proof}
First note that if $\ell\geq1$, then the tuple $(g',d',r',\ell'):=(g-k,d-k,r-1,\ell-1)$ satisfies \eqref{cond: tuple}. Indeed, by Proposition \ref{ineqBN} the assumption $\rho(g,r-\ell,d)-\ell k\geq0$ implies $d\geq k$ (hence $g'\geq1$ and $d'\geq0$), and since
\[
\rho(g-k,r-\ell,d-k)-(\ell-1)k=\rho(g,r-\ell,d)-\ell k\geq0
\]
the map $\pi_{(g-k,d-k,r-1,\ell-1)}$ is indeed well defined. Note also that:
\begin{itemize}
    \item If $m_1\neq0$, then $m_1'=m_1-1$, $m_2'=m_2+1$ and $e'=e$.
    \item If $m_1=0$, then $m_1'=m_2-1$, $m_2'=1$ and $e'=e-1$.
\end{itemize}

\vspace{1mm}

Assume the map $\pi_{(g-k,d-k,r-1,\ell-1)}$ is dominant. Then for a general curve $C\in|H-E|$, the fiber $\pi_{(g-k,d-k,r-1,\ell-1)}^{-1}(Y)\subseteq \Pic^{d-k}(C)$ is smooth of pure dimension
\[
\rho(g-k,r-\ell,d-k)-(\ell-1)k=\rho(g,r-\ell,d)-\ell k.
\]
Having chosen general curves $C\in |H-E|$ and $J\in |E|$, we consider the nodal curve
\[
Y:=C+J\in |H|
\]
and we denote by $x_1, \ldots, x_k$ the points of the (transverse) intersection $C\cdot J$.

\vspace{1mm}

Consider a general point $L_C\in\pi_{(g-k,d-k,r-1,\ell-1)}^{-1}\bigl([C]\bigr)$. If $m_1\neq0$, then $L_C$ fits in a short exact sequence
\[
0\longrightarrow 
\cO_X((e+1)E)^{\oplus m_1-1}\oplus \cO_X(eE)^{\oplus m_2+1}\longrightarrow L_C\longrightarrow W \longrightarrow 0
\]
in $\Coh^b(X)$ ($-1\ll b<0$), where $W\in\cM_{H_\epsilon}\bigl(-(r-\ell+1),H-\ell E,-(g-d+r-\ell)\bigr)$ is a general element. More precisely, as pointed out in Example \ref{Example:LineBundles}, we have $W=G[1]$ for a vector bundle $G$ arising as the kernel of a certain evaluation map. 
A similar statement holds true if $m_1=0$, by considering a distinguished triangle
\[
0\longrightarrow 
\cO_X(eE)^{\oplus m_2-1}\oplus \cO_X((e-1)E)\longrightarrow L_C\longrightarrow W \longrightarrow 0.
\]
We will assume $m_1\neq0$ in the rest of the proof; for $m_1=0$ one applies, \emph{mutatis mutandis}, the same argument.

\vskip 3pt

Consider the short exact sequence 
\[
0\longrightarrow \cO_X(eE)\longrightarrow \cO_X((e+1)E)\longrightarrow \cO_J\longrightarrow 0
\]
obtained by multiplication by the section defining the curve $J\in |E|$. We can construct the following commutative diagram with exact rows and columns in $\Coh^b(X)$:

\begin{equation}\label{diag:YC}
\begin{xy}
\xymatrix{
0 \ar[r]  & \cO_X((e+1)E)^{\oplus m_1-1} \oplus \cO_X(eE)^{\oplus m_2+1}
\ar[r] \ar[d] &  L_C \ar[r]\ar[d] & W  \ar[r]\ar[d]^{=} &  0
\\
0 \ar[r]  & \cO_X((e+1)E)^{\oplus m_1} \oplus \cO_X(eE)^{\oplus m_2} \ar[r] \ar[d] & L \ar[d]\ar[r]\ar[d] & W\ar[r] & 0\\
&    \cO_J \ar[r]^{=}& \cO_J  & &
}
\end{xy}
\end{equation}

\noindent It follows from the middle column that $L$ is a line bundle on the curve $Y=C+ J$, such that $L_{|J}\cong\cO_J$ and $L_{|C}\cong L_C(x_1+\cdots+x_k)$. In particular, $\deg(L)=d$. Furthermore, we can read from the second row that the stability type of $L$ equals $\overline{e}=\{(e+1,m_1),(e,m_2)\}$. Therefore $h^0(Y, L)=r+1$ and $L\in\pi^{-1}_{(g,d,r,\ell)}\bigl([Y]\bigr)$.

We are going to prove that the irreducible component of $\pi^{-1}_{(g,d,r,\ell)}\bigl([Y]\bigr)$ containing $L$ has the right dimension $\rho(g,r-\ell,d)-\ell k$, or equivalently, that this construction fills up (locally) a component of $\pi^{-1}_{(g,d,r,\ell)}\bigl([Y]\bigr)$. To this end, observe that since $L$ is locally free, the locally around $L$ the fibre $\pi^{-1}_{(g,d,r,\ell)}\bigl([Y]\bigr)$ is isomorphic to the locus, see also (\ref{HBNdegloci}) 
\begin{align*}
V^r_{d, \ell}\bigl(Y, \cO_Y(E)\bigr):=\Bigl\{N\in \mbox{Pic}^d(Y): h^0(Y, N)= r+1, \ h^0(Y, N(-(e+1)E))= m_1, \\
h^0(Y, N(-eE))= 2m_1+m_2\Bigr\},
\end{align*}
where $\mbox{Pic}^d(Y)$ consists of line bundles of bidegree $(d,0)$ on $Y=C+J$. 

\vskip 3pt

One has the following exact sequence 
\begin{equation}\label{eq:multipliers}
0\longrightarrow (\mathbb C^*)^{k-1}\longrightarrow \mbox{Pic}^d(Y)\longrightarrow \mbox{Pic}^d(C)\times \mbox{Pic}^0(J)\longrightarrow 0.
\end{equation}
Note that if $N\in \mbox{Pic}^d(Y)$ has a global section $s=(s_C, s_J)\in H^0(C, N_{|Y})\times H^0(J, N_{|J})$ such that $s(x_i)\neq 0$ for $i=1, \ldots, k$,
then $N$ is uniquely determined by its restrictions $N_{|C}\in \mbox{Pic}^d(C)$ and $N_{|J}\in \mbox{Pic}^0(J)$, because the multipliers corresponding to the fibre of
(\ref{eq:multipliers}) over the point $(N_{|C}, N_{|J})$ are given by the quotients $s_Y(x_i)/s_C(x_i)\in \mathbb C^*$.

We now observe that the line bundle $L$ constructed via (\ref{diag:YC}) enjoys this property. Indeed, since $h^0(C, L_C)=r$, $h^0(Y, L)=r+1$ and $h^0(J, \cO_J)=1$, we have an exact sequence
\[
0\longrightarrow H^0(C, L_C)\longrightarrow H^0(Y, L)\longrightarrow H^0(J, \cO_J)\longrightarrow 0.
\]
But a section $s_0\in H^0(Y, L)$ projecting onto the section $1\in H^0(J, \cO_J)$  does not vanish at any of the points $C\cdot J$.

\vskip 4pt

Assume now that $N\in V^r_{d,\ell}\bigl(Y, \cO_Y(E)\bigr)$, where $\mbox{deg}(N_{|C})=d$ and $\mbox{deg}(N_{|J})=0$. Since we are working locally around $L$, we may assume that
\begin{equation}\label{ineqsectionsY}
\begin{aligned}
  h^0\bigl(C, N_{|C}(-x_1-\cdots-x_k)\bigr)&\leq r,\\
  h^0\bigl(C, N_{|C}(-x_1-\cdots-x_k)(-eE)\bigr)&\leq 2m_1+m_2-1,\\
  h^0\bigl(C, N_{|C}(-x_1-\cdots-x_k)(-(e+1)E)\bigr)&\leq m_1-1.
\end{aligned}
\end{equation}
On the other hand one has an exact sequence
\[
0\longrightarrow N_{|C}(-x_1-\cdots-x_k)\longrightarrow N\longrightarrow N_{|J}\longrightarrow 0,
\]
from which it follows that $h^0(C, N_{|C}(-x_1-\cdots-x_k))\geq h^0(Y, N)-h^0(J, N_{|J})\geq r$, with equality only if $N_{|J}\cong \cO_J$. Therefore, $N_{|J}\cong \cO_J$ and $N$ is uniquely determined by its restriction $N_{|C}$. Furthermore, 
\[
h^0\bigl(Y, N_{|Y}(-x_1-\cdots-x_k)(-aE)\bigr)\geq h^0(C,N(-aE))-1
\]
for $a=e, e+1$. Therefore, all inequalities in \eqref{ineqsectionsY} are equalities, which means that $N_{|C}(-x_1-\cdots-x_k)$ has $\{(e+1,m_1-1),(e,m_2+1)\}$ as the nonnegative part of its splitting type. Thus, by Proposition \ref{comparisontypes}, we find that $N_{|C}(-x_1-\cdots-x_k)$ lies in $\pi_{(g-k,d-k,r-1,\ell-1)}^{-1}\bigl([C]\bigr)$ (since we work locally). Conversely, as already shown, a general line bundle $L_C\in\pi_{(g-k,d-k,r-1,\ell-1)}^{-1}\bigl([C]\bigr)$ gives rise via \eqref{diag:YC} to an element in the fibre $\pi_{(g,d,r,\ell)}^{-1}\bigl([Y]\bigr)$. 
This shows that 
\[
\mbox{dim}_{L}\Bigl(\pi_{(g,d,r,\ell)}^{-1}\bigl([Y]\bigr)\Bigr)=\mbox{dim}_{L_C} \Bigl(\pi_{(g-k,d-k,r-1,\ell-1)}^{-1}\bigl([C]\bigr)\Bigr)=\rho(g, r-\ell, d)-\ell k,
\]
which finishes the proof.
\end{proof}

\vskip 9pt

As a consequence of Proposition \ref{reduction:l=0}, we can establish the dominance of $\pi_{(g,d,r,\ell)}$ for arbitrary $\ell$ in the following cases:

\vskip 9pt

\begin{Cor}\label{cor:dominance_restricted}
Let $X$ be a degree $k$ elliptic  $K3$ surface with $\mathrm{Pic}(X)\cong \mathbb Z\cdot H\oplus \mathbb Z\cdot E$.  If $$\ell\geq\frac{1}{2}\bigl(g-d+2r+1-k\bigr),$$ then the map $\pi_{(g,d,r,\ell)}$ is dominant. In particular, $V^r_{d, \ell}\bigl(C, \cO_C(E)\bigr)$ is non-empty and of dimension $\rho(g, r-\ell,d)-\ell k$ for a general curve $C\in |H|$.
\end{Cor}

\begin{proof}
In view of Proposition \ref{reduction:l=0}, we may assume $\ell=0$. Thus we aim to prove that
    \[
    \cM_{H_\epsilon}\bigl((0,H,1+d-g),\overline{e}\bigr)\longrightarrow |H|
    \]
    is a dominant map for $\overline{e}=\{(0,r+1)\}$. This amounts to prove that for a general $C\in|H|$ there exists a line bundle  $L\in\Pic^d(C)$ with $h^0(C, L)=r+1$ and $h^0(C, L(-E))=0$.

\vskip 3pt
Since $k\geq g-d+2r+1-k$, we have
    \begin{equation}\label{maxl=0}
        \rho(g,r-\ell,d)-\ell k < \rho(g,r,d)\ \text{ for every
        }\ell\geq 1
    \end{equation}
By Corollary \ref{HBNupperbound}, \eqref{maxl=0} implies 
$\dim W^r_d(C)=\rho(g,r,d)$ for a general curve $C\in|H|$
    and also $\dim W^r_d(C)>\dim W^{r+1}_d(C)$, so a general $L\in W^r_d(C)$ has $h^0(C, L)=r+1$.

    \vskip 2pt
    
    On the other hand, any $L\in\Pic^d(C)$ with $h^0(C, L)=r+1$ and $h^0(C, L(-E))\neq0$ has stability type $\overline{e'}=\{(e_i',m_i')\}_{i=1}^p$, where $e_1'\geq 1$ and
    \begin{equation}\label{refined-r+1}
    r+1\geq\left\{
    \begin{array}{c l}
     2m_1 & \text{if $p=1$}\\
     2(m_1+\cdots+m_{p-1})+m_p & \text{if $p\geq 2$}\\
    \end{array}
    \right.
    \end{equation}
    (this follows from Theorem \ref{thm.splitting} and its proof). Furthermore, Theorem \ref{relative:HBNthm} yields
    \[
    \dim\cM_{H_\epsilon}\bigl((0,H,1+d-g),\overline{e'}\bigr)\leq g+\rho(g,r-\ell,d)-\ell k\overset{\eqref{maxl=0}}{<} g+\rho(g,r,d)
    \]
    where $\ell=r+1-(m_1+\cdots+m_p)\overset{\eqref{refined-r+1}}{\geq}1$. It follows that, for a general $C\in|H|$, the locus
    \[
    \bigl\{ L\in\Pic^d(C):h^0(C, L)=r+1,\; h^0(C, L(-E))\neq0\bigr\}
    \]
    has dimension $<\rho(g,r,d)$, which concludes the proof.
\end{proof}

\vskip 7pt

\begin{Cor}\label{cor_dominance_maximal}
For a degree $k$ elliptic $K3$ surface with $\mathrm{Pic}(X)\cong \mathbb Z\cdot H\oplus \mathbb Z\cdot E$, for a general curve $C\in |H|$, we have that 
$$\mathrm{dim}\ W^r_d(C)=\rho_k(g,r,d).$$
\end{Cor}
\begin{proof}This follows from Corollary \ref{cor:dominance_restricted} recalling that $\ell_{\mathrm{max}}:=\frac{1}{2}\bigl(g-d+2r+1-k\bigr)$ is the quantity where the maximum of the quadratic function $\rho(g,r-\ell,d)-\ell k$ is attained. Hence, having proven existence for 
$$\ell\geq\min\Bigl\{\frac{1}{2}(g-d+2r+1-k), r\Bigr\},$$we have shown the existence of a component of $W^r_d(C)$ of the maximal dimension. 
\end{proof}

Next we observe how our results provide a full answer to the existence problem of the loci $V^r_{d, \ell}(C, A)$ in the extremal case when $\ell=r+2-k$.

\begin{Thm}\label{thm:existencel+2}
For a degree $k$ elliptic $K3$ surface with $\mathrm{Pic}(X)\cong \mathbb Z\cdot H\oplus \mathbb Z\cdot E$, for a general curve $C\in |H|$ we have 
$$\mathrm{dim} \ V^r_{d, r+2-k}\bigl(C, \cO_C(E)\bigr)=\rho(g, k-2,d)-(r+2-k)k.$$
\end{Thm}
\begin{proof}
Follows immediately by combining Proposition \ref{prop:coppens-martens} with  Proposition \ref{reduction:l=0}.  
\end{proof}

\subsection{Halphen surfaces} To establish an  existence result analogous to Theorem \ref{thm:existencel+2} for all values of $\ell$, we need to relax the assumption of working with an arbitrary degree $k$ elliptic $K3$ surface with $\mbox{Pic}(X)\cong \mathbb Z\cdot H\oplus \mathbb Z\cdot E$. We recall that such $K3$ surfaces form an irreducible  Noether-Lefschetz divisor in the $19$-dimensional moduli space $\cF_g$ of polarized $K3$ surfaces of genus $g$ and we will use a degeneration of elliptic $K3$ surfaces to \emph{Halphen surfaces}  of index $k$ following \cite[2.1]{cantat-dolgachev}. Let $p_1, \ldots, p_9\in \mathbf{P}^2$ be distinct points and denote by $$\epsilon \colon X:=\mathrm{Bl}_{\{p_1, \ldots, p_9\}}\bigl(\mathbf{P}^2\bigr)\longrightarrow \mathbf{P}^2$$ the blow-up of $\mathbf{P}^2$ at these points, by $E_1, \ldots, E_9$ the corresponding exceptional divisors on $X$ and set 
$\cO_X(h):=\epsilon^*(\cO_{\mathbb P^2}(1))$. 
\begin{Def}
The surface 
$X$ is said to be a Halphen surface if there exists an integer $k\geq 2$ such that $\mbox{dim} \ \bigl|-kK_X\bigr|=1$ and $\bigl|-kK_X \bigr|$ is base point free. The smallest $k$ with these properties is the \emph{index} of $X$. 
\end{Def}
From the definition it follows there exists a unique cubic curve 
$J\in \bigl|3h-E_1-\cdots-E_9\bigr|$, therefore $J\equiv -K_X$ and there exists an irreducible plane curve $Z$ of degree $3k$ having points of multiplicity $k$ at $p_1, \ldots, p_9$. Accordingly, then
$$\cO_J\bigl(3h-E_1-\cdots-E_9\bigr)=\cO_J(J) \in \mbox{Pic}^0(J)[k]$$
is a torsion point of order $k$. Moreover $Z\cdot J=0$ and there exists an elliptic pencil $f\colon X\rightarrow \mathbf{P}^1$ having $kJ$ as a non-reduced fibre.

\vskip 4pt

We fix a Halphen surface $X$ of index $k$ and for $g\geq 1$ define the \emph{du Val} linear system 
\begin{equation}\label{eq:duVal-system}
\Lambda_g:=\bigl|3gh-gE_1-\cdots -gE_8-(g-1)E_9\bigr|.
\end{equation}
It follows from \cite[Lemma 2.4]{arbarello-explicit} that the general element $C\in \Lambda_g$ is a smooth curve of genus $g$. Observe that $C\cdot J=1$, therefore the point of intersection $p_{10}^{(g)}=C\cdot J$ is a base point of the linear system $\Lambda_g$. The point $p_{10}^{(g)}\in J$ is determined by the relation 
\begin{equation}\label{eq:p10g}
p_{10}^{(g)}=-gp_1-\cdots -gp_8-(g-1)p_9\in J
\end{equation}
with respect to the group law on $J$.
 Blowing up the base point $p_{10}^{(g)}$ and denoting by $|C'|$ the strict transform of the linear system $\Lambda_g$, it is shown in \cite{arbarello-bruno-sernesi} that the linear system $|C'|$ is base point free on $X':=\mathrm{Bl}_{p_{10}^{(g)}}(X)$ and maps $X'$ onto a surface $\overline{X}\subseteq \mathbf{P}^g$ that is a limit of polarized $K3$ surfaces of degree $2g-2$. Note that the map $X'\rightarrow \overline{X}$ contracts the proper transform of $Z$ to an elliptic singularity $q$ on $\overline{X}$. The main result of \cite{arbarello-explicit} asserts that if the points $p_1, \ldots, p_9\in \mathbf{P}^2$ are not associated with a Halphen surface of index at most $g$, then the general curve $C\in \Lambda_g$ satisfies the Petri theorem.
This is the case if $p_1, \ldots, p_9$ are chosen generically in $\mathbf{P}^2$.
 \vskip 3pt

 Assume now that $X$ is a Halphen surface of index $k$. Then $A:=\cO_C(kJ)\in W^1_k(C)$  and it is proved in \cite[Theorem 3.1]{arbarello24} that if $k\leq \frac{g+3}{2}$, then $A$ computes the gonality of $C$, therefore $\mbox{gon}(C)=k$. Furthermore, using e.g. \cite{arbarello-bruno-sernesi} it is easy to see that Halphen surfaces appear as limits of degree $k$ elliptic $K3$ surface. Indeed, it follows from \cite[Theorem 24]{arbarello-bruno-sernesi} that the map between the deformation functors 
 $$\mu\colon \mbox{Def}\bigl(\overline{X}, \cO_{\overline{X}}(1)\bigr)\rightarrow \mbox{Def}\bigl(\overline{X},q\bigr)$$ is smooth. Here $\mbox{Def}\bigl(\overline{X},q\bigr)$ is parametrized by $H^0(\overline{X}, T^1_{\overline{X}})$, where $T^1_{\overline{X}}=\mathcal{E} xt^1\bigl(\Omega_{\overline{X}}^1, \cO_{\overline{X}}\bigr)$, whereas $\mbox{Def}\bigl(\overline{X}, \cO_{\overline{X}}(1)\bigr)$ is parametrized by $\mbox{Ext}^1\bigl(Q_{\cO_{\overline{X}}(1)}, \cO_{\overline{X}}\bigr)$, where $Q_{\cO_{\overline{X}}(1)}$ is the corresponding Atiyah class sitting in the exact sequence
 $$0\longrightarrow \Omega_{\overline{X}}\longrightarrow Q_{\cO_{\overline{X}}(1)}\longrightarrow \cO_{\overline{X}} \longrightarrow 0.$$
 Since $\mbox{dim } \mbox{Def}\bigl(\overline{X}, q\bigr)>1$, it follows that the codimension one subvariety of $\mbox{Def}\bigl(\overline{X}, \cO_{\overline{X}(1)}\bigr)$ parametrizing those deformations that also preserve the line bundle $\cO_{\overline{X}}(Z)$ cannot be contained in the space of topologically trivial deformations of $(\overline{X}, \cO_{\overline{X}}(1))$. Since $|\cO_{\overline{X}}(1)(-Z)|$ is very ample, we also know that each formal deformation of $\bigl(\overline{X}, \cO_{\overline{X}}(1)\bigr)$ is also effective. Therefore we obtain  that the  Halphen surface $\overline{X}\subseteq \mathbf{P}^g$ smooths to an elliptic $K3$ surface like in Proposition \ref{thm-k3 surface}. 

\vskip 4pt

Before stating our next result, we recall that for a point $p$ on a smooth curve $C$ we introduce the \emph{vanishing sequence} at $p$
$$a^{\ell}(p)=\bigl(0\leq a_0^{\ell}(p) < \ldots < a^{\ell}_r(p) \leq d\bigr)$$ of a linear system 
$\ell=(L, V)\in G^r_d(C)$ by ordering the vanishing orders at $p$ of the sections from $V$. The \emph{ramification sequence} of $\ell$ at $p$ is the non-decreasing sequence $\alpha^{\ell}(p)=\bigl(\alpha_0^{\ell}(p)\leq \cdots \leq \alpha^{\ell}(p)\bigr)$ 
obtained by setting $\alpha^{\ell}_i(p):=a^{\ell}_i(p)-i$.  Having fixed integers $0<r \leq d$, a Schubert index of type $(r, d)$ is a non-decreasing sequence of integers $\bar{\alpha}=\bigl(0\leq \alpha_0 \leq \ldots \leq \alpha_r \leq d-r\bigr)$. Its \emph{weight} is defined as $|\bar{\alpha}|:=\alpha_0+\cdots+\alpha_r.$

For a curve $Y$ of compact type, for points $p_1, \ldots, p_s\in Y_{\mathrm{reg}}$ and for Schubert indices $\bar{\alpha}^1, \ldots, \bar{\alpha}^s$ of type $(r,d)$, we denote by $\overline{G}^r_d\bigl(Y,(p_1, \bar{\alpha}^1), \ldots, (p_s, \bar{\alpha}^s)\bigr)$ the variety of limit linear series $\ell$ of type $g^r_d$ on $Y$ satisfying the inequalities 
$\alpha^{\ell}(p_i)\geq \bar{\alpha}^i$, for $i=1, \ldots, s$. It is a determinantal variety of expected dimension 
\begin{equation}\label{eq:BNnumb_pointed}
\rho\bigl(g, r, d, \bar{\alpha}^1, \ldots, \bar{\alpha}^s\bigr)=\rho(g,r,d)-\sum_{i=1}^s |\bar{\alpha}^i|.
\end{equation}

We shall use that if $[J, p_1, p_2]\in \cM_{1,2}$ is a $2$-pointed elliptic curve such that $\cO_J(p_1-p_2)$ is not a torsion bundle, then $G^r_d\bigl(J, (p_1, \bar{\alpha}^1), (p_2, \bar{\alpha}^2)\bigr)$ has the expected dimension (\ref{eq:BNnumb_pointed}) for any choice of the Schubert indices $\bar{\alpha}^1$ and $\bar{\alpha}^2$. We refer to \cite{EH1} for basics on the theory of limit linear series.

\begin{Thm}\label{thm:existence_limit}
Fix a general degree $k$ elliptic $K3$ surface with $\mathrm{Pic}(X)\cong \mathbb Z\cdot H\oplus \mathbb Z\cdot E$. Then for $d\leq g-1$ and $k\geq r+2$ such that $\rho(g,r,d)\geq 0$, for a general curve $C\in |H|$, there exists a component $Z$ of $W^r_d(C)$ of dimension $\rho(g,r,d)$, whose general point corresponds to a line bundle $L$ with $H^0(C, L\otimes E_C^{\vee})=0$.   
\end{Thm}

\begin{proof}
We specialize to curves on a degree $k$ Halphen surface. As we shall explain, by semicontinuity via use of limit linear series, it will suffice to establish  the conclusion of the theorem for a general curve  $C\in \Lambda_g$ on a Halphen surface $X$ of index $k$ as above. Let $C$ degenerate inside $\Lambda_g$ to the transverse union $J\cup C_{g-1}$, where $C_{g-1}$ is a general curve from the Du Val linear system $\Lambda_{g-1}$. Note that $C_{g-1}$ and $J$ meet at the point $p_{10}^{(g-1)}$ which is the base point of $\Lambda_{g-1}$, whereas  $p_{10}^{(g)}$ lies on $J$ such that $p_{10}^{(g)}-p_{10}^{(g-1)}=p_1+\cdots+p_9\in J$ (with respect to the group law). In particular $\cO_J\bigl(p_{10}^{(g)}-p_{10}^{(g-1)}\bigr)$ is a $k$-torsion point. Furthermore, we let $C_{g-1}$ degenerate to the union $J\cup C_{g-2}$, where $C_{g-2}$ is a general member of the Du Val system $\Lambda_{g-2}$. Note that $C_{g-2}\cdot J=p_{10}^{(g-2)}$, where the point $p_{10}^{(g-2)}\in J$ is determined by the relation $p_{10}^{(g-1)}-p_{10}^{(g-2)}=p_1+\cdots+p_9$ with respect to the group law of $J$. Iterating this procedure, we arrive eventually at a stable curve 
\begin{equation}\label{eq:Y}
Y:=J_1\cup \ldots\cup J_g,
\end{equation}
where each component $J_i$ is a copy of the curve $J$ and 
$\bigl\{p^{(i)}\bigr\}=J_i\cap J_{i+1}$, where in the interest of easing the notation we write $p^{(i)}=p_{10}^{(i)}$ for $i=1, \ldots, g$. The difference $\cO_{J_i}\bigl(p^{(i)}-p^{(i-1)}\bigr) \in \mbox{Pic}^0(J_i)$ is torsion of order $k$ for $i=2, \ldots, g$.  We fix furthermore a point $p^{(0)}\in J_1\setminus \bigl\{p^{(1)}\bigr\}$ such that  $\cO_{J_1}\bigl(p^{(1)}-p^{(0)}\bigr)\in \mbox{Pic}^0(J_1)$ has order $k$. To summarize, $[Y]$ is a limit in $\mm_g$ of smooth curves $C\in \Lambda_g$ lying on a Halphen surface.

\vskip 3pt

We shall construct a limit linear series $\ell\in \overline{G}^r_d(Y)$ and a limit linear series $\mathfrak{a}\in \overline{G}^1_k(Y)$ such that both $\ell$ and $\mathfrak{a}$ smooth to linear series $L\in W^r_d(C)$ and $A=\cO_C(Z)\in W^1_k(C)$ on nearby smooth curves $C\in \Lambda_g$, such that (i) $L$ belongs to a component of $W^r_d(C)$ of dimension $\rho(g,r,d)$ and (ii) $H^0(C, L\otimes A^{\vee})= 0$. 

\vskip 3pt
We now specify the aspects $\ell_{J_1}, \ldots, \ell_{J_g}$ of $\ell$. We start by setting 
$$\ell_{J_1}:=\bigl(\cO_{J_1}\bigl(d\cdot p^{(1)}\bigr), V_{J_1}\bigr)\in G^r_d(J_1),$$
where $V_{J_1}=H^0\bigl(J_1, \cO_{J_1}((r+1)p^{(1)})\bigr)\subseteq H^0\bigl(J_1, \cO_{J_1}(dp^{(1)})\bigr)$. We observe that 
\begin{equation}\label{eq:vanJ1}
\alpha^{\ell_{J_1}}\bigl(p^{(0)}\bigr)=\bigl(\underbrace{0, \ldots, 0}_{r+1}\bigr) \ \mbox{ and } \  \alpha^{\ell_{J_1}}\bigl(p^{(1)}\bigr)=\bigl(\underbrace{d-r-1, \ldots, d-r-1}_{r}, d-r\bigr).
\end{equation}
Note that since $r+1\leq k-1$, certainly  $\cO_{J_1}\left((r+1)(p^{(1)}-p^{(0)})\right)$ is nontrivial. 

\vskip 3pt

Then on $J_i$ with $i=2, \ldots, r+2$, we choose the linear series $$\ell_{J_i}=\Bigl(\cO_{J_i}\bigl((2i-2)\cdot p^{(i-1)}+(d-2i+2)\cdot p^{(i)}\bigr), V_{J_i}\Bigr)\in G^r_d(J_i),$$ where 
$$
V_i:= H^0\bigl(\cO_{J_i}(i\cdot p^{(i-1)})\bigr)+H^0\bigl(\cO_{J_i}((r-i+2)\cdot p^{(i)})\bigr) \subseteq H^0\Bigl(\cO_{J_i}\bigl((2i-2)\cdot p^{(i-1)}+(d-2i+2)\cdot p^{(i)}\bigr)\Bigr).$$
Observe that $V_i$ is $(r+1)$-dimensional since the subspaces $H^0\bigl(\cO_{J_i}(i\cdot p^{(i-1)})\bigr)$ and $H^0\bigl(\cO_{J_i}((r-i+2)\cdot p^{(i)})\bigr)$ intersect along the $1$-dimensional subspace $\langle \sigma_i \rangle$, 
where $\sigma_i\in V_i$ is the unique section with $\mbox{div}(\sigma_i)=(2i-2)\cdot p^{(i-1)}+(d-2i+2)\cdot p^{(i)}$. Note that the ramification profile of $\ell_{J_i}$ equals 
$\alpha^{\ell_{J_i}}\bigl(p^{(i-1)}\bigr)=\bigl(\underbrace{i-2, \ldots, i-2}_{i-1}, \underbrace{i-1, \ldots, i-1}_{r-i+2}\bigr)$  and 
$\alpha^{\ell_{J_i}}\bigl(p^{(i)}\bigr)=\bigl(\underbrace{d-r-i, \ldots, d-r-i}_{r-i+1}, \underbrace{d-r-i+1, \ldots, d-r-i+1}_{i}\bigr)$. In particular, we have $\alpha^{\ell_{J_{r+2}}}\bigl(p^{(r+1)}\bigr)=\bigl(\underbrace{r, \ldots, r}_{r+1}\bigr)$.

\vskip 4pt

In the same way, we construct the aspects of $\ell$ for the next $(g-d+r-1)$ groups of subchains of $r+1$ elliptic curves of $Y$. Precisely, for $a=1+b(r+1)+i$, where $b\leq g-d+r-1$ and $1\leq i\leq r+1$, we choose the linear series $\ell_{J_a}\in G^r_d(J_a)$ such that 
\begin{equation}\label{eq:ram_a}
\alpha^{\ell_{J_a}}\bigl(p^{(a-1)}\bigr)=\bigl(\underbrace{br+i-1, \ldots, br+i-1}_{i}, \underbrace{br+i, \ldots, br+i}_{r+1-i}\bigr),
\end{equation}
where the ramification sequence of $\ell_{J_a}$ at the point $p^{(a)}$ is determined by the condition that $\ell$ form a refined limit linear series, that is, $\alpha_j^{\ell_{J_a}}\bigl(p^{(a)}\bigr)+\alpha_{r-j}^{\ell_{J_{a+1}}}\bigl(p^{(a)}\bigr)=d-r$, for $j=0, \ldots, r$. Note  that for $b=g-d+r-1$ and $i=r+1$, that is, for the component labeled by $1+(r+1)(g-d+r)$, we have the ramification 
$$\alpha^{\ell_{J_{1+(r+1)(g-d+r)}}}\bigl(p^{((r+1)(g-d+r))}\bigr)=\bigl(\underbrace{(g-d+r)r, \ldots, (g-d+r)r}_{r+1}\bigr).$$

Our choices imply the following equalities
\begin{equation}\label{eq:BN}
\rho\Bigl(1, r, d, \alpha^{\ell_{J_{i}}}(p^{(i-1)}), \alpha^{\ell_{J_i}}\bigl(p^{(i)}\bigr)\Bigr)=0, \ \mbox{ for } i=1, \ldots, (r+1)(g-d+r).
\end{equation}

For the last $\rho(g,r,d)$ components of $Y$, we choose linear series $\ell_{J_a}\in G^r_d(J_a)$ such that 
$$\alpha^{\ell_{J_a}}\bigl(p^{(a-1)}\bigr)=\bigl(\underbrace{a-(g-d+r-1), \ldots, a-(g-d+r-1)}_{r+1}\bigr),$$
for $a=1+(r+1)(g-d+r),\ldots, g$. The ramification sequence $\alpha^{\ell_{J_a}}\bigl(p^{(a)}\bigr)$ is again determined by the condition that the aspects of $\ell$ form a refined limit linear series.
In particular, observe that
\begin{equation}\label{eq:BN2}
\rho\Bigl(1, r,d,\alpha^{\ell_{J_a}}\bigl(p^{(a-1)}\bigr), \alpha^{\ell_{J_a}}\bigl(p^{(a)}\bigr)\Bigr)=1, \ 
\mbox{ for } a=1+(r+1)(g-d+r), \ldots, g,
\end{equation}
that is, we have an irreducible $1$-dimensional family of choices for $\ell_{J_a}$ in this range.
\vskip 3pt

We claim that the limit linear series $\ell\in \overline{G}^r_d(Y)$ just constructed is  smoothable to \emph{every} smooth curve of genus $g$, in particular to a curve $C\in \Lambda_g$ on a Halphen surface. We consider the stack $\sigma\colon \widetilde{\mathcal{G}}_d^r\rightarrow \mm_g$ of limit linear series over the versal deformation space of $Y$. We chose a curve 
\begin{equation}\label{eq:Y'}
Y':=J_1\cup \ldots\cup J_g,
\end{equation}
where this time $\bigl\{p^i\bigr\}=J_i\cap J_{i+1}$ and now we assume that $\cO_{J_i}\bigl(p^i-p^{i-1}\bigr)\in \mbox{Pic}^0(J_i)$ is \emph{not} a torsion point for $i=1, \ldots, g$. In the same way, one can construct a limit linear series $\ell'\in \overline{G}^r_d(Y')$ having the \emph{same} ramification profile as $\ell\in \overline{G}^r_d(Y)$ at the points of intersection $J_i\cap J_{i-1}$. Clearly, $\ell$ and $\ell'$ belong to the same irreducible component of $\widetilde{\mathcal{G}}^r_d$.

\vskip 3pt

As explained in \cite[Theorem 3.4]{EH1}, every component of $\widetilde{\cG}^r_d$ has dimension at least $3g-3+\rho(g,r,d)$. To conclude that a component of $\widetilde{\cG}^r_d$ containing the point $[Y, \ell]$ dominates $\mm_g$ it thus suffices to show that the local dimension of the fibre $\sigma^{-1}\bigl([Y']\bigr)$ at $\ell'$ has dimension precisely $\rho(g,r,d)$. Assume that $\tilde{\ell}$ is a limit linear series on $Y'$ lying in a component of $\overline{G}^r_d(Y')$ passing through $\ell$ and having dimension larger than $\rho(g,r,d)$.  Then we may assume that $\tilde{\ell}$ is refined and using the additivity of the adjusted Brill-Noether numbers \cite[Lemma 3.6]{EH1}, we write 
$$\rho(g,r,d)=\sum_{a=1}^g \rho\Bigl(1,r,d, \alpha^{\tilde{\ell}_{J_a}}\bigl(p^{a-1}\bigr), \alpha^{\tilde{\ell}_{J_a}}\bigl(p^{a}\bigr)\Bigr),$$
 and obtain that there exists $a\in \{1, \ldots,  g\}$ such that $\rho\Bigl(1,r,d, \alpha^{\tilde{\ell}_{J_a}}\bigl(p^{a-1}\bigr), \alpha^{\tilde{\ell}_{J_a}}\bigl(p^{a}\bigr)\Bigr)<0$. Then there exist integers $0\leq i_1<i_2\leq r$ such that 
$$\alpha_{i_1}^{\tilde{\ell}_{J_a}}\bigl(p^{a-1}\bigr)+\alpha_{r-i_1}^{\tilde{\ell}_{J_a}}\bigl(p^ {a}\bigr)=\alpha_{i_2}^{\tilde{\ell}_{J_a}}\bigl(p^{a-1}\bigr)+\alpha_{r-i_2}^{\tilde{\ell}_{J_a}}\bigl(p^{a}\bigr)=d-r,$$
in particular there exist two sections of the $J_a$-aspect of $\tilde{\ell}$ which vanish only at the points $p^{a-1}$ and $p^{a}$, implying that $\cO_{J_a}\bigl(p^{a-1}-p^{a}\bigr)\in \mbox{Pic}^0(J_a)$ is torsion, which is a contradiction. Thus $\ell$ is smoothable to a linear system $L\in W^r_d(C)$ on every smooth curve $C\in \Lambda_g$. Moreover, we may also assume $h^0(C,L)=r+1$.

\vskip 3pt

We now return to the curve $Y$ and construct  a limit linear series $\mathfrak{a}\in \overline{G}^1_k(Y)$ having the ramification profiles
$$\alpha^{\mathfrak{a}_{J_i}}\bigl(p^{(i-1)}\bigr)=\alpha^{\mathfrak{a}_{J_i}}\bigl(p^{(i)}\bigr)=(0, k-1), \ \mbox{ for } i=1, \ldots, g.$$
Since $Z\cap J=\emptyset$ it follows that $\mathfrak{a}$  smooths to $A=\cO_C(Z)\in W^1_k(C)$ for a neighboring smooth curve $C\in \Lambda_g$. Assuming now that $H^0(C, L\otimes A^{\vee})\neq 0$, we also obtain by semicontinuity that there exists a section 
$$0\neq \sigma\in V_{J_1}\cap H^0\bigl(J_1, \cO_{J_1}((d-k)\cdot p^{(1)})\bigr).$$
But using (\ref{eq:vanJ1}) we have $\mbox{ord}_{p^{(1)}}(\sigma)\geq d-r-1>d-k$, since $k\geq r+2$, which is a contradiction and thus finishes the proof.
\end{proof}

\begin{Rem}It remains an interesting question whether the conclusion of Theorem \ref{thm:existence_limit} holds for an arbitrary elliptic $K3$ surface like in Proposition \ref{thm-k3 surface}.
\end{Rem}

\section{Hurwitz-Brill-Noether general curves over number fields}\label{sec:numberfield}

The aim of this short section is to explain how curves on Halphen surfaces provide explicit examples of $k$-gonal curves defined over a number field which are general from the viewpoint of Hurwitz-Brill-Noether theory.  The strategy is inspired by the papers \cite{arbarello-explicit} and  \cite{farkas-tarasca}, in which explicit examples of smooth curves of genus $g$ defined over the rationals are provided  which satisfy the classical Brill-Noether-Petri theorem. 

\vskip 3pt

We fix a Halphen surface $X=\mathrm{Bl}_{p_1, \ldots, p_9}\bigl(\mathbf{P}^2\bigr)$ of index $k$ and consider the Du Val linear system $\Lambda_g=\bigl|3gh-gE_1-\cdots-gE_8-(g-1)E_9\bigr|$ of curves of genus $g$ defined as in (\ref{eq:duVal-system}). We use all the notation from the previous section.

\vskip 3pt

\begin{Def}
A smooth $k$-gonal curve $C$ of genus $g$ is said to be Hurwitz-Brill-Noether general if $\mbox{dim } W^r_{d}(C)=\rho_{k}(g, r, d)$ for all  integers $r, d>0$.
\end{Def}

Before proving the next result, we recall Pflueger's following definition \cite[Definition 2.5]{pflueger}. For a natural number $n$ we write $[n]:=\{1, \ldots, n\}$. Then a $k$-\emph{uniform displacement tableau} is a function $t\colon [r+1]\times [g-d+r]\rightarrow \mathbb Z_{>0}$ such that 
\begin{itemize} 
\item $t(x+1,y)>t(x,y) \ \mbox{ and } \ t(x, y+1)>t(x,y)$, for all $x, y$, and 
\item if $t(x, y)=t(x', y')$, then $x-y\equiv x'-y' \ \mbox{mod }  k.$
\end{itemize}
In other words, $t$ consists of labeled boxes, where two labels may coincide only if they are $k$ boxes apart.

\begin{Thm}\label{thm:duval}
Let $X$ be a Halphen surface of index $k\geq 2$. Then a general curve $C\in \Lambda_g$ is Hurwitz-Brill-Noether general. 
\end{Thm}
\begin{proof}
We first show that if $C\in \Lambda_g$ is a general curve, then  $\mbox{dim } W^r_d(C)\leq \rho_k(g,r,d)$. Assume there exists a component $Z$ of $W^r_d(C)$ of dimension $\mbox{dim}(Z)>\rho_k(g,r,d)$. 

\vskip 4pt

As in the proof of Theorem \ref{thm:existence_limit}, we let $C$ degenerate to the transversal union $J\cup C_{g-1}$, where $C_{g-1}$ is a general curve from the Du Val linear system $\Lambda_{g-1}$ and where $C_{g-1}$ and $J$ meet at the point $p_{10}^{(g-1)}$. Recall that  $p_{10}^{(g)}$ lies on $J$ and that  $\cO_J\bigl(p_{10}^{(g)}-p_{10}^{(g-1)}\bigr)$ is a $k$-torsion point. Under this degeneration, $Z$ specializes to a component $Z_{g-1}$ of the variety of limit linear series $\overline{G}^r_d\bigl(C_{g-1}\cup J\bigr)$ having $\mbox{dim}(Z_{g-1})>\rho_k(g,r,d)$.

\vskip 3pt

We consider the forgetful map $\pi_g \colon \overline{G}^r_d\bigl(C_{g-1}\cup J\bigr) \rightarrow G^r_d\bigl(C_{g-1}\bigr)$ retaining the $C_{g-1}$-aspect of each limit linear series. There exists a Schubert index 
$\bar{\alpha}^{g-1}$ of type $(r, d)$ maximal with respect to the property $\pi_g(Z_{g-1})\subseteq G^r_d\bigl(C_{g-1}, (p_{10}^{(g-1)}, \bar{\alpha}^{g-1})\bigr)$. Therefore for a general point $\ell=(\ell_{C_{g-1}}, \ell_J)\in Z_{g-1}$, we have that $\alpha^{\ell_{C_{g-1}}}\bigl(p_{10}^{(g-1)}\bigr)=\bar{\alpha}^{g-1}$.

\vskip 3pt

We let $C_{g-1}$ degenerate to the transversal union $J\cup C_{g-2}$, where $C_{g-2} \in \Lambda_{g-2}$. Then 
$C_{g-2}\cdot J =p_{10}^{(g-2)}$ and recall that  $p_{10}^{(g-2)}\in J$ is such that $p_{10}^{(g-1)}-p_{10}^{(g-2)}=p_1+\cdots +p_9$. 
We now consider the forgetful map $$\pi_{g-1}\colon \overline{G}^r_d\Bigl(C_{g-2}\cup J, (p_{10}^{(g-1)}, 
\bar{\alpha}^{g-1})\Bigr) \longrightarrow G^r_d(C_{g-2})$$ that retains the $C_{g-2}$-aspect of each limit linear series. The subvariety $\pi_{g-1}(Z_{g-1})$ then specializes  to a subvariety $Z_{g-2}\subset \overline{G}^r_d\bigl(C_{g-2}\cup J, (p_{10}^{(g-1)}, \bar{a}^{g-1})\bigr)$. We choose the Schubert index 
$\bar{\alpha}^{g-2}$ maximal with the property that
$\pi_{g-1}(Z_{g-2})\subseteq G^r_d\bigl(C_{g-2}, (p_{10}^{(g-2)}, \bar{\alpha}^{g-2})\bigr)$.

\vskip 4pt

Inductively, assume that we have found a Schubert index $\bar{\alpha}^{g-i+1}$ and  defined a subvariety $Z_{g-i}\subset \overline{G}^r_d\bigl(C_{g-i}\cup J, (p_{10}^{(g-i+1)}, \bar{\alpha}^{g-j+1})\bigr)$. If 
$$\pi_{g-i+1}\colon \overline{G}^r_d\Bigl(C_{g-i}\cup J, (p_{10}^{(g-i+1)}, \bar{\alpha}^{g-i+1})\Bigr)\longrightarrow G^r_d(C_{g-i})
$$ is the map retaining the $C_{g-i}$-aspect, let $\bar{\alpha}^{g-i}$ be the maximal Schubert index with the property $\pi_{g-i+1}(Z_{g-i})\subset G^r_d\bigl(C_{g-i}, (p_{10}^{(g-i)}, \bar{\alpha}^{g-i})\bigr)$.
Then we let $C_{g-i}$ degenerate to the union $C_{g-i-1}\cup J$, where $C_{g-i-1}\in \Lambda_{g-i-1}$ is a general element meeting $J$  at the point $p_{10}^{(g-i-1)}$. Let $Z_{g-i-1}\subset \overline{G}^r_d \bigl(C_{g-i-1}\cup J, (p_{10}^{(g-i)}, \bar{\alpha}^{g-i})\bigr)$ be the limiting subvariety of 
$\pi_{g-i-1}(Z_{g-i})$ under this degeneration.

\vskip 4pt

Just like in the proof of Theorem \ref{thm:existence_limit},  we consider the chain of elliptic curves 
$$Y:=J_1\cup \ldots  \cup J_g,$$
where $J_i\cong J$ and  $\bigl\{p^{(i)}\bigr\}=\bigl\{p^{(i)}_{10}\bigr\}=J_i\cap J_{i+1}$. We recall that $\cO_{J_i}\bigl(p^{(i)}-p^{(i-1)}\bigr) \in \mbox{Pic}^0(J_i)$ is $k$-torsion. Via the previous argument, after $g$ steps, we find that there exists a family of limit linear series on $Y$ having dimension strictly exceeding $\rho_k(r,d)$ consisting of limit linear series $\ell\in \overline{G}^r_d(Y)$ satisfying the ramification conditions $\alpha^{\ell_{J_{i}}}\bigl(p^{(i)}\bigr)=\bar{\alpha}^{i}=\bigl(\alpha_0^i, \alpha_1^i, \ldots, \alpha_r^i\bigr)$, for $i=1, \ldots, g$. Set $a_n^i:=\alpha_n^i+i$, for $n=0, \ldots, r$ for the corresponding vanishing sequences.

\vskip 4pt

One has the inequalities $a_n^i+ 1\geq a_n^{i-1}\geq a_n^{i}$ for all $n=0, \ldots, r$ and $i=1, \ldots, g$. Furthermore, if $a_n^{i-1}=a_n^{i}$, then the $J_i$-aspect of a general limit linear series of this family has as underlying line bundle $L_i\cong \cO_{J_i}\bigl(a_n^i\cdot p^{(i)}+(d-a_n^i)\cdot p^{(i-1)}\bigr)$. In particular, since $p^{i}-p^{i-1}$ is a $k$-torsion point on $J_i$, if $a_n^{i}=a_{n}^{i-1}$ and 
$a_{n'}^i=a_{n'}^{i-1}$ for two integers $0\leq n<n'\leq r$, then $a_{n'}^i-a_{n}^i \equiv 0 \mbox{ mod } k$.

\vskip 4pt

To a limit linear series $\ell$  we associate a tableau $t\colon [r+1]\times [g-d+r]\rightarrow [g]$ by setting $$t(n, j):=\mbox{min}\bigl\{i: \alpha_n^i+i=d-r+j\bigr\},$$
that is, $t(n, j)$ is the smallest integer $i$ such that in the non-increasing sequence consisting of  $\bar{\alpha}^0_n:=d-r, \bar{\alpha}^1_n, \ldots, \bar{\alpha}^i_n$ there exist $t(n, j)$ equalities.
Then, as explained above, $t$ defines a $k$-uniform displacement tableau. For each label $i$ appearing in $\mbox{Im}(t)$, the variety $G^r_d\bigl(J_i, (p^{(i-1)}, \alpha^{\ell_{J_i}}(p^{(i-1)})), (p^{(i)}, \bar{\alpha}^i)\bigr)$ appearing as a factor in $\overline{G}^r_d(Y)$ has dimension zero. It follows that
$\mbox{dim } \overline{G}^r_d(Y)$ is bounded from above by the number of  labels omitted from $[g]$ in a $k$-uniform displacement tableau on $[r+1]\times [g-d+r]$. Using Pflueger's result \cite[Section 3]{pflueger}, this quantity is bounded from above by $\rho_k(g, r,d)$, thus showing that for a general curve $C\in \Lambda_g$ we have $\mbox{dim } W^r_d(C) \leq \rho_k(g,r,d)$.

\vskip 4pt

To show that $\mbox{dim } W^r_d(C)\geq \rho_k(g, r, d)$ we  use that if $X$ is a Halphen surface of degree $k$ giving rise to the surface $\overline{X}\subseteq \mathbf{P}^g$ of degree $2g-2$, then there exists a family of degree $k$ polarized $K3$ surfaces $(X_t, H_t, E_t)$  such that the corresponding image $X_t\stackrel{|H_t|}\longrightarrow \mathbf{P}^g$ has as limit the surface $\overline{X}$. Since for a general curve $C_t\in |H_t|$ we have established in Corollary \ref{cor_dominance_maximal} that $\mbox{dim } W^r_d(C)=\rho_k(g, r, d)$, by semicontinuity $\mbox{dim } W^r_d(C)\geq \rho_k(g,r,d)$. 
\end{proof}

\vskip 4pt
As explained in \cite{arbarello-explicit}, one can write down  rational points $p_1, \ldots, p_9\in \mathbf{P}^2$ such that the general curve $C\in \Lambda_g$ is defined over $\mathbb{Q}$ and satisfies the Brill-Noether-Petri theorem. In the case of a Halphen surface $X$ of index $k$, the condition $\cO_J(J)\in \mbox{Pic}^0(J)[k]$  cannot be realized over the rationals. For instance, Mazur \cite{mazur} showed that if $J$ is an elliptic curve defined over $\mathbb{Q}$, then any prime $k$ dividing $E(\mathbb{Q})_{\mathrm{tors}}$ satisfies $k\leq 7$. However we have the following result:   

\begin{Thm}\label{prop:hbn_rationals}
For every prime $k$ there exists a Hurwitz-Brill-Noether general curve $[C, A]\in \cH_{g,k}$  defined over a number field  $K$ with $[K:\mathbb{Q}]\leq k^2-1$.     
\end{Thm}

\begin{proof}
We start with any elliptic curve $J\subseteq \mathbf{P}^2$ defined over $\mathbb{Q}$ and with $8$ rational points $p_1, \ldots, p_8$. Then $p_9\in J$ is determined by the condition that $p_1+\cdots+p_9$ is a $k$-torsion point with respect to the group law of $J$. We apply \cite[Theorems 2.1 and 5.1]{lozano-robledo}, to conclude that the field $K$ of the definition of $p_9$ satisfies
$$[K:\mathbb{Q}]\leq k^2-1.$$
Indeed, if $\rho \colon \mbox{Gal}(\overline{\mathbb Q}/\mathbb Q)\rightarrow GL_2\bigl(\mathbb F_k)\cong \mbox{Aut}\bigl(E[k]\bigr)$ denotes the Galois representation on the $k$-torsion points of $J$, then 
$[K:\mathbb Q]=|\mbox{Im}(\rho)|/|H|$, where $H$ is the subgroup of all matrices of type $\begin{pmatrix}
1 & a \\
0 & b
\end{pmatrix}$, where $a\in \mathbb F_k$ and $b\in \mathbb F_k^*$. Since $|H|=k(k-1)$, it follows that $[K:\mathbb Q]\leq k^2-1$, with equality if and only if $\rho$ is surjective, 
which in fact happens for all but finitely many primes $k$. 

\vskip 3pt

The general curve $C\in \Lambda_g$
defined by (\ref{eq:duVal-system}) will be then defined over $K$ and by Theorem \ref{thm:duval} is Hurwitz-Brill-Noether general.  
\end{proof}

\begin{Rem}
It is an interesting open question whether there exists a smooth $k$-gonal curve of genus $g$ defined over $ \mathbb{Q}$ which is Hurwitz-Brill-Noether general. As explained, curves on Halphen surface of degree $k$ will not provide the answer for all $k$. Similarly, we do not known whether there exists a degree $k$ elliptic $K3$ surface $X$ like in Proposition \ref{thm-k3 surface} defined over $\mathbb{Q}$.      
\end{Rem}

\bibliography{mybib}                     

\begin{thebibliography}{ACGH85}

\bibitem[ABFS16]{arbarello-explicit}
Enrico Arbarello, Andrea Bruno, Gavril Farkas, and Giulia Sacc\`a.
\newblock Explicit {B}rill-{N}oether-{P}etri general curves.
\newblock {\em Comment. Math. Helv.}, 91(3):477--491, 2016.

\bibitem[ABS17]{arbarello-bruno-sernesi}
Enrico Arbarello, Andrea Bruno, and Edoardo Sernesi.
\newblock On hyperplane sections of {K}3 surfaces.
\newblock {\em Algebr. Geom.}, 4(5):562--596, 2017.

\bibitem[ACGH85]{arbarello:geometry-of-algebraic-curves}
E.~Arbarello, M.~Cornalba, P.~A. Griffiths, and J.~Harris.
\newblock {\em Geometry of algebraic curves. {V}ol. {I}}, volume 267 of {\em
  Grundlehren der Mathematischen Wissenschaften [Fundamental Principles of
  Mathematical Sciences]}.
\newblock Springer-Verlag, New York, 1985.

\bibitem[Apr13]{aprodu:lazarsfeld}
M.~Aprodu.
\newblock Lazarsfeld-{M}ukai bundles and applications.
\newblock In {\em Commutative algebra}, pages 1--23. Springer, New York, 2013.

\bibitem[Arb24]{arbarello24}
Enrico Arbarello.
\newblock A remark on du {V}al linear systems.
\newblock {\em Rend. Circ. Mat. Palermo (2)}, 73(5):2161--2174, 2024.

\bibitem[Bal89]{ballico}
E.~Ballico.
\newblock A remark on linear series on general {$k$}-gonal curves.
\newblock {\em Boll. Un. Mat. Ital. A (7)}, 3(2):195--197, 1989.

\bibitem[Bay18]{bayer:brill-noether}
A.~Bayer.
\newblock Wall-crossing implies {B}rill-{N}oether: applications of stability
  conditions on surfaces.
\newblock In {\em Algebraic geometry: {S}alt {L}ake {C}ity 2015}, volume~97 of
  {\em Proc. Sympos. Pure Math.}, pages 3--27. Amer. Math. Soc., Providence,
  RI, 2018.

\bibitem[Bay19]{bayer:deformation}
Arend Bayer.
\newblock A short proof of the deformation property of {B}ridgeland stability
  conditions.
\newblock {\em Math. Ann.}, 375(3-4):1597--1613, 2019.

\bibitem[BF18]{bakker-farkas}
Benjamin Bakker and Gavril Farkas.
\newblock The {M}ercat conjecture for stable rank 2 vector bundles on generic
  curves.
\newblock {\em Amer. J. Math.}, 140(5):1277--1295, 2018.

\bibitem[BL17]{li-bn}
Arend Bayer and Chunyi Li.
\newblock Brill-{N}oether theory for curves on generic abelian surfaces.
\newblock {\em Pure Appl. Math. Q.}, 13(1):49--76, 2017.

\bibitem[BM14]{bayer:projectivity}
A.~Bayer and E.~Macr{\`{\i}}.
\newblock Projectivity and birational geometry of {B}ridgeland moduli spaces.
\newblock {\em J. Amer. Math. Soc.}, 27(3):707--752, 2014.

\bibitem[BMS16]{bayer:the-space-of-stability-conditions-on-abelian-threefolds}
A.~Bayer, E.~Macr\`\i, and P.~Stellari.
\newblock The space of stability conditions on abelian threefolds, and on some
  {C}alabi-{Y}au threefolds.
\newblock {\em Invent. Math.}, 206(3):869--933, 2016.

\bibitem[Bri07]{bridgeland:stability-condition-on-triangulated-category}
T.~Bridgeland.
\newblock Stability conditions on triangulated categories.
\newblock {\em Ann. of Math. (2)}, 166(2):317--345, 2007.

\bibitem[Bri08]{bridgeland:K3-surfaces}
T.~Bridgeland.
\newblock Stability conditions on {$K3$} surfaces.
\newblock {\em Duke Math. J.}, 141(2):241--291, 2008.

\bibitem[CD12]{cantat-dolgachev}
Serge Cantat and Igor Dolgachev.
\newblock Rational surfaces with a large group of automorphisms.
\newblock {\em J. Amer. Math. Soc.}, 25(3):863--905, 2012.

\bibitem[CM99]{coppens-martens}
M.~Coppens and G.~Martens.
\newblock Linear series on a general {$k$}-gonal curve.
\newblock {\em Abh. Math. Sem. Univ. Hamburg}, 69:347--371, 1999.

\bibitem[CNY23]{CNY-K3}
Izzet Coskun, Howard Nuer, and Kota Yoshioka.
\newblock The cohomology of the general stable sheaf on a {K}3 surface.
\newblock {\em Adv. Math.}, 426:Paper No. 109102, 85, 2023.

\bibitem[CNY25]{CNY-abelian}
Izzet Coskun, Howard Nuer, and Kota Yoshioka.
\newblock Weak {B}rill--{N}oether on abelian surfaces.
\newblock {\em Selecta Math. (N.S.)}, 31(3):Paper No. 48, 2025.

\bibitem[Cop25]{coppens-new}
Marc Coppens.
\newblock A picture of the irreducible components of $w^r_d(c)$ for a general
  $k$-gonal curve $c$, 2025, arXiv:2504.21141.

\bibitem[CPJ22]{cook-powell-jensen}
Kaelin Cook-Powell and David Jensen.
\newblock Components of {B}rill-{N}oether loci for curves with fixed gonality.
\newblock {\em Michigan Math. J.}, 71(1):19--45, 2022.

\bibitem[EH83]{EH2}
D.~Eisenbud and J.~Harris.
\newblock A simpler proof of the {G}ieseker-{P}etri theorem on special
  divisors.
\newblock {\em Invent. Math.}, 74(2):269--280, 1983.

\bibitem[EH86]{EH1}
David Eisenbud and Joe Harris.
\newblock Limit linear series: basic theory.
\newblock {\em Invent. Math.}, 85(2):337--371, 1986.

\bibitem[Fey22]{soheyla:restriction}
Soheyla Feyzbakhsh.
\newblock An effective restriction theorem via wall-crossing and {M}ercat's
  conjecture.
\newblock {\em Math. Z.}, 301(4):4175--4199, 2022.

\bibitem[FL81]{fulton-lazarsfeld}
W.~Fulton and R.~Lazarsfeld.
\newblock On the connectedness of degeneracy loci and special divisors.
\newblock {\em Acta Math.}, 146(3-4):271--283, 1981.

\bibitem[FL21]{soheyla-li}
Soheyla Feyzbakhsh and Chunyi Li.
\newblock Higher rank {C}lifford indices of curves on a {K}3 surface.
\newblock {\em Selecta Math. (N.S.)}, 27(3):Paper No. 48, 34, 2021.

\bibitem[FT17]{farkas-tarasca}
Gavril Farkas and Nicola Tarasca.
\newblock Du {V}al curves and the pointed {B}rill-{N}oether theorem.
\newblock {\em Selecta Math. (N.S.)}, 23(3):2243--2259, 2017.

\bibitem[GH80]{griffiths-harris}
Phillip Griffiths and Joseph Harris.
\newblock On the variety of special linear systems on a general algebraic
  curve.
\newblock {\em Duke Math. J.}, 47(1):233--272, 1980.

\bibitem[Gie82]{gieseker:petri}
D.~Gieseker.
\newblock Stable curves and special divisors: {P}etri's conjecture.
\newblock {\em Invent. Math.}, 66(2):251--275, 1982.

\bibitem[HL10]{huybrechts:geometry-of-moduli-space-of-sheaves}
D.~Huybrechts and M.~Lehn.
\newblock {\em The geometry of moduli spaces of sheaves}.
\newblock Cambridge Mathematical Library. Cambridge University Press,
  Cambridge, second edition, 2010.

\bibitem[JR21]{jensen-ranganathan}
David Jensen and Dhruv Ranganathan.
\newblock Brill-{N}oether theory for curves of a fixed gonality.
\newblock {\em Forum Math. Pi}, 9:Paper No. e1, 33, 2021.

\bibitem[Knu03]{knutsen-gonality-k3curves}
A.L. Knutsen.
\newblock Gonality and {C}lifford index of curves on {$K3$} surfaces.
\newblock {\em Arch. Math. (Basel)}, 80(3):235--238, 2003.

\bibitem[Lar21]{larson:inv}
Hannah~K. Larson.
\newblock A refined {B}rill-{N}oether theory over {H}urwitz spaces.
\newblock {\em Invent. Math.}, 224(3):767--790, 2021.

\bibitem[Laz86]{lazarsfeld-BNP}
Robert Lazarsfeld.
\newblock Brill-{N}oether-{P}etri without degenerations.
\newblock {\em J. Differential Geom.}, 23(3):299--307, 1986.

\bibitem[Ley12]{leyenson}
Maxim Leyenson.
\newblock On the {B}rill-{N}oether theory for {K}3 surfaces.
\newblock {\em Cent. Eur. J. Math.}, 10(4):1486--1540, 2012.

\bibitem[LLV25]{larson-larson-vogt}
Eric Larson, Hannah Larson, and Isabel Vogt.
\newblock Global {B}rill-{N}oether theory over the {H}urwitz space.
\newblock {\em Geom. Topol.}, 29(1):193--257, 2025.

\bibitem[LR13]{lozano-robledo}
\'Alvaro Lozano-Robledo.
\newblock On the field of definition of {$p$}-torsion points on elliptic curves
  over the rationals.
\newblock {\em Math. Ann.}, 357(1):279--305, 2013.

\bibitem[Mar01]{markman}
Eyal Markman.
\newblock Brill-{N}oether duality for moduli spaces of sheaves on {$K3$}
  surfaces.
\newblock {\em J. Algebraic Geom.}, 10(4):623--694, 2001.

\bibitem[Maz78]{mazur}
B.~Mazur.
\newblock Rational isogenies of prime degree (with an appendix by {D}.
  {G}oldfeld).
\newblock {\em Invent. Math.}, 44(2):129--162, 1978.

\bibitem[MS17]{macri-schmidt-lectures}
Emanuele Macr\`i and Benjamin Schmidt.
\newblock Lectures on {B}ridgeland stability.
\newblock In {\em Moduli of curves}, volume~21 of {\em Lect. Notes Unione Mat.
  Ital.}, pages 139--211. Springer, Cham, 2017.

\bibitem[Pfl17]{pflueger}
Nathan Pflueger.
\newblock Brill-{N}oether varieties of {$k$}-gonal curves.
\newblock {\em Adv. Math.}, 312:46--63, 2017.

\bibitem[SD74]{saint:projective-models-of-k3-surfaces}
B.~Saint-Donat.
\newblock Projective models of {$K-3$} surfaces.
\newblock {\em Amer. J. Math.}, 96:602--639, 1974.

\bibitem[Tod08]{toda-k3-surface}
Yukinobu Toda.
\newblock Moduli stacks and invariants of semistable objects on {$K3$}
  surfaces.
\newblock {\em Adv. Math.}, 217(6):2736--2781, 2008.

\end{thebibliography}
\bibliographystyle{halpha}  

\end{document}